\documentclass[12pt,reqno]{amsart}
\usepackage[margin=1in]{geometry}
\makeatletter
\def\section{\@startsection{section}{1}%
\z@{.7\linespacing\@plus\linespacing}{.5\linespacing}%
{\bfseries\normalfont\scshape
	\centering
}}
\def\@secnumfont{\bfseries}
\makeatother
\usepackage{hyperref}
\usepackage{amsmath,amsfonts,amsthm,txfonts}
\usepackage{dsfont}
\usepackage{mathrsfs}

\usepackage{accents}
\usepackage{trimclip}

\newcommand{\newhat}{\scalebox{1.5}[.75]{\trimbox{0pt 1.1ex}{\textasciicircum}}}

\newcommand{\stretchedhat}[1]{\accentset{\newhat}{#1}}
\def\shat{\stretchedhat}

\usepackage{graphicx,amssymb}
\usepackage{amsmath,amssymb,amsfonts}
\usepackage{hyperref}
\usepackage{amscd}
\usepackage{amsthm}
\usepackage{xcolor}
\numberwithin{equation}{section}
\newtheorem{Thm}{Theorem}[section]
\newtheorem{Lem}{Lemma}[section]
\newtheorem{Pro}{Proposition}[section]
\newtheorem{Rem}{Remark}[section]

\newtheorem{Ass}{Assumption}[section]

\newtheorem{Def}{Definition}[section]
\DeclareMathOperator*{\esssup}{ess\,sup}
\usepackage[shortlabels]{enumitem}

\usepackage{cancel}

\usepackage{cases} 

\allowdisplaybreaks

\begin{document}
\title[ Control of Landau-Lifshitz-Bloch Equation]
{Time Optimal Control Problem for the Landau-Lifshitz-Bloch equation}
\author[Sidhartha Patnaik and  Kumarasamy Sakthivel]{Sidhartha Patnaik and Kumarasamy Sakthivel}
\address{Department of Mathematics \\
Indian Institute of Space Science and Technology (IIST) \\
Trivandrum- 695 547, INDIA}
\email{sidharthpatnaik96@gmail.com, sakthivel@iist.ac.in}
\curraddr{}

\begin{abstract}
This paper investigates the time-optimal control problem for the Landau-Lifshitz-Bloch (LLB) equation, a macroscopic model that characterizes magnetization dynamics in ferromagnetic materials across a wide temperature range, including near and above the Curie temperature. We analyze the LLB system on bounded domains in one, two, and three dimensions, establishing the existence of optimal controls that drive the magnetization to a desired target state within a minimal time frame. Utilizing a Lagrange multiplier approach and an adjoint-based framework, we derive first-order necessary optimality conditions. Furthermore, we establish second-order sufficient conditions for local optimality, addressing the mathematical challenges posed by the system's inherent nonlinearities and the nonlinear appearance of the control in the effective magnetic field. These results provide a rigorous theoretical basis for the rapid manipulation of magnetic states, offering insights into the fundamental limits of control for nonlinear diffusion-relaxation processes in magnetism. Such findings are essential for advancing high-speed magnetic memory technologies and optimizing thermal magnetic switching in next-generation storage technologies.

\end{abstract}
\subjclass[2010]{35K20, 35Q56, 35Q60, 49J20}
\keywords{Landau-Lifshitz-Bloch equation, Magnetization dynamics, Optimal control, First-order optimality condition, Second-order optimality condition}
\maketitle

\section{Introduction}

The Landau-Lifshitz-Bloch equation is a macroscopic mathematical model that describes the dynamics of magnetization in ferromagnetic materials under the influence of external fields and thermal effects. The classical Landau-Lifshitz-Gilbert (LLG) equation only works at low temperatures and assumes that the magnetization magnitude stays the same. The LLB equation, on the other hand, can be used in high-temperature situations (see \cite{UADH}). It explicitly accounts for temperature-dependent phenomena such as longitudinal relaxation and variations in magnetization magnitude, making it particularly useful for studying magnetization behavior near and above the Curie temperature, where thermal fluctuations significantly impact magnetic dynamics. Because the LLB equation accounts for these temperature-dependent aspects, researchers widely use it to study ultrafast magnetization processes, especially those occurring at higher temperatures like heat-assisted magnetic recording (HAMR), where precise control of magnetization at elevated temperatures is essential (see, \cite{MKEG}). Moreover, in spintronics, the LLB equation plays an important role, where spin-dependent transport and ultrafast switching dynamics can unlock significant technological advancements.

The time-optimal control approaches applied to the LLB equation enable a detailed investigation into rapid magnetic switching processes driven by femtosecond laser pulses. Such techniques are particularly important for the emerging field of ultrafast spin dynamics, where rapid manipulation of spin states under thermal excitation is essential. Employing optimal control methodologies with the LLB model helps precisely determine the minimal switching times and required laser parameters, thus providing valuable insights that drive advancements in spintronic devices and ultrafast magnetic recording technologies (see \cite{UADH}).

Consider the magnetization $\bar{m}$, with initial saturation magnetization $\bar{m}_s^0$ at $t=0$, evolves according to the  LLB equation. If we normalize the spin polarization $m = \bar{m} / \bar{m}_s^0$,  the LLB system can be expressed as
\begin{equation} \label{MGE}
	m_t = \gamma m \times H_{\text{eff}} + D_L \frac{(m \cdot H_{\text{eff}})}{|m|^2} m - D_T \frac{m \times (m \times H_{\text{eff}})}{|m|^2},
\end{equation}
where  $|\cdot|$ denotes the Euclidean norm in $\mathbb{R}^3$,  $\gamma > 0$ is the gyromagnetic ratio, $D_L$ and $D_T$ are the longitudinal and transverse damping parameters. Moreover, the term $H_{\text{eff}}$ is the effective magnetic field, accounting for external fields, magnetic anisotropy, and thermal fluctuations.

Above the Curie temperature $T_C$, ferromagnetic materials transition into a paramagnetic state, disrupting magnetic ordering. In this regime, the damping parameters equalize, that is, $D_L = D_T = \kappa$, leading to a notable change in the material’s response to external fields.  
For longitudinal susceptibility $\chi_\|$ and an applied external field $u$, the effective field is given by 
\begin{equation} \label{H_EFF}
	H_{\text{eff}} = \Delta m - \frac{1}{\chi_\|} \left(1 + \frac{3}{5} \frac{T}{T - T_C} |m|^2 \right) m + u.
\end{equation}

Substituting \eqref{H_EFF} into \eqref{MGE}, using $D_L = D_T = \kappa$, and applying the vector identity $m \times (m \times H_{\text{eff}}) = (m \cdot H_{\text{eff}}) m - |m|^2 H_{\text{eff}}$, we obtain 
\begin{equation*}
	m_t = \nu m \times \Delta m + \nu m \times u + \kappa \left(\Delta m - \frac{1}{\chi_\|} \left(1 + \frac{3}{5} \frac{T}{T - T_C} |m|^2 \right) m + u \right).
\end{equation*}
For simplicity, we set the constants $\nu = \kappa = \chi_\| = \frac{3}{5} \frac{T}{T - T_C} = 1$. For more details about the derivation of the effective field, one may refer to \cite{KNL}.

For any smooth, bounded domain $\Omega\subset \mathbb{R}^n$, $n=1,2,3$ with initial magnetization $m_0$, control $u$, and homogeneous Neumann boundary condition, the LLB equation takes the following form:
\begin{equation}\label{NLP}
	\begin{cases}
		m_t- \Delta m= m \times \Delta m  + m \times u- \big(1+ |m|^2\big)\,m + u, \ \ \ (x,t)\in \Omega_T,\\
		
		\frac{\partial m}{\partial \eta}=0, \ \ \ \ \ \ \ (x,t) \in  \partial\Omega_T,\\
		
		m(\cdot,0)=m_0 \ \ \text{in} \ \Omega,
	\end{cases}	
\end{equation}
where $\eta$ is the outward unit  normal vector to the boundary $\partial \Omega$. For a detailed summary of the control problem formulation, one may refer to \cite{SPSK3}.

In order to derive the optimality conditions, we need solutions with higher regularity. To guarantee this regularity, we assume that the initial data  $m_0:\Omega \to \mathbb{R}^3$ satisfies the following conditions:
\begin{equation}\label{IC}
		m_0 \in H^2(\Omega),\ \ \ \frac{\partial m_0}{\partial \eta}=0 \ \ \text{on} \ \partial \Omega. 
\end{equation}

\noindent \textbf{Time Optimal Control Problem.}
Let $m_{\Omega}:\Omega \to \mathbb{R}^3$ be the desired magnetic moment and $\delta$ be a positive real number. We consider the time-optimal control of the LLB system corresponding to the objective functional:
\begin{equation}\label{T-CF}
	\mathcal J(u):= \frac{1}{2} \ T^*(u)^2 +\frac{1}{2} \|u\|^2_{W^{1,2}(0,T;H^1(\Omega);H^1(\Omega)^*)},			
\end{equation}
where $T^*(u):= \inf\ \big\{t\in [0,T]  \ :\ \|m(t)-m_{\Omega}\|_{L^2(\Omega)} \leq \delta\big\}.$ If $T^*(u)$ exists for some control $u$, we will refer to it as the ``minimal time" associated with that control. In order to exclude the trivial case, we impose the condition $\|m_0-m_{\Omega}\|_{L^2(\Omega)}>\delta$, where $\delta$ comes from the definition of the cost functional \eqref{T-CF}.

The LLG system has been the focus of numerous studies. In \cite{FAAS}, the authors proved the existence of weak solutions for the Landau-Lifshitz (LL) equation on bounded and unbounded domains, but without an external field, and showed that these solutions are generally not unique. While the studies  \cite{GCPF}, \cite{GCPF2}, and \cite{MFTT} addressed local and global existence of regular solutions, fewer works explore optimal control problems for the LLG equation. The paper \cite{TDMKAP} examined the solvability of the control problem for the 1D LLG equation, established the existence of an optimal solution, and derived a first-order necessary optimality condition. In our earlier works \cite{SPSK} and \cite{SPSK2}, we explored various approaches to the norm optimal control problem for the LLG equation. In \cite{SPSK}, we addressed the optimal control problem by taking an admissible control set that satisfies a smallness assumption, whereas in \cite{SPSK2}, we considered a broader class of controls and tackled the problem by applying the external magnetic field through a finite number of fixed coils. Because of such a control configuration, we can get the second-order optimality condition along a sharp critical cone. For more details on the solvability and computational aspects of micromagnetism, one may refer to \cite{AP}.

There are very few articles available on the LLB system. The existence of weak solutions and their regularity properties were studied in \cite{KNL}. In \cite{QLBG}, the authors established the existence and uniqueness of smooth solutions for the LLB system in $\mathbb{R}^2$ without a control parameter in the effective field. Furthermore, for $\mathbb{R}^3$, they showed that smooth solutions exist under a smallness assumption on the initial data. In \cite{SGUM}, the authors established the existence of weak martingale and strong solutions for the stochastic LLB equation in dimensions 1, 2, and 3, along with a weak relaxed optimal control for a general lower semicontinuous cost functional using Young measures. A deterministic version of the control problem was studied in our previous work \cite{SPSK3}, where we established the existence of strong solutions, formulated the control problem with a finite set of magnetic field coils, and derived necessary and sufficient optimality conditions, along with global optimality and uniqueness results.

The time-optimal control problem for semilinear parabolic PDEs has received minimal attention in the literature. In the pioneering works \cite{VB1,VB2}, the maximum principle for time-optimal control of 2D Navier-Stokes equations was established. These studies demonstrate that the velocity field of a viscous incompressible fluid in a bounded domain can be driven to a target state in minimal time using a constrained distributed force. The optimal control satisfies the maximum principle and exhibits a bang-bang structure. In \cite{EFC}, the author investigated both norm and time-optimal control problems for the variable density Navier-Stokes equations, establishing results on existence, uniqueness, regularity, and first-order optimality conditions. Similarly, \cite{CTA} addressed a time-optimal control problem for the Navier–Stokes–Voigt equations, demonstrating the existence of optimal solutions and deriving first-order necessary and second-order sufficient optimality conditions. Related work for the 3D Navier-Stokes-$\alpha$ system is presented in \cite{DTS}.   Moreover, in \cite{GW}, the authors investigated the existence of a time-optimal control for certain semilinear parabolic equations. In the paper \cite{KKWL}, they explored the time optimal control problems for an internally controlled heat equation with pointwise control constraints, identifying two formulations: (i) reaching a target set at a fixed time while delaying control activation, and (ii) achieving the target in the shortest time.  For a more detailed discussion on time-optimal control problems, one may refer to \cite{YLMB,HMNO}.

{\bf Main Contributions of the Paper.} To the best of the author's knowledge, this is the first work addressing the time-optimal control of the LLB system describing magnetization dynamics in ferromagnetic materials. By employing the existence and uniqueness of regular solutions for \eqref{NLP} (see \cite{SPSK3}), we proved the existence of solutions for the time-optimal control problem. We also derived a first-order necessary condition and a second-order sufficient optimality condition by using the approach developed in \cite{FT1,EFC,CTA}. The main difficulty of this paper is due to the non-linearities $m\times \Delta m$ and $|m|^2m$ arising in \eqref{NLP}, and the control itself arises non-linearly in the model. These nonlinearities further complicate the form of the first and second variational systems (see \eqref{T-LS2} and \eqref{ER-SY}), whose solvability is critical for the analysis of first- and second-order optimality conditions. The required technical estimates for these systems have been derived through a unified solvability result (Lemma \ref{L-SER}).  For proving the existence of a regular solution $m\in W^{1,2}(0, T; H^3(\Omega), H^1(\Omega)),$ we have only taken the control in $L^2(0, T; H^1(\Omega)). $ The control in the functional $\mathcal J(u)$ could be taken in $L^2(0, T; H^1(\Omega))$ and it is sufficient to prove the existence of a time-optimal control (Theorem \ref{T-EOOC}).  However, for first- and second-order conditions, we need the control  $u\in W^{1,2}(0, T; H^1(\Omega), H^1(\Omega)^\ast)$, which allows us to get the continuity in time $C([0, T];L^2(\Omega)),$ leading to the fact that  $m$ is continuously differentiable in time, and also it enables us to get the continuous differentiability of the linearized system \eqref{Z-KK}.

The manuscript is organized as follows. Section \ref{MR} introduces the function spaces, key inequalities, and the main results of the paper. The existence of optimal control is discussed in Section \ref{SEC-EOC}. Subsection \ref{SEC-LS} addresses the solvability of some general linear and nonlinear systems, using which the solvability of several systems throughout this paper is justified. The analysis of the first-order optimality conditions is presented in Subsection \ref{S-FOOC}. Finally, Section \ref{SEC-SO} focuses on the important result of second-order local optimality conditions.

\section{Main Results and Inequalities}\label{MR}

\subsection{Main Results}
In this subsection, we list the function spaces and inequalities used consistently in our paper and provide a brief overview of the main findings and key results of our study.

For any subset $\Omega \subset \mathbb{R}^n$, the Lebesgue space \( L^p(\Omega) \) is defined as the set of functions \( u: \Omega \to \mathbb{R} \) such that \( \|u\|_{L^p} < \infty \), where  
\[ 
\|u\|_{L^p(\Omega)} = \left( \int_{\Omega} |u(x)|^p \,dx \right)^{1/p}\ \ \text{for}\  1 \leq p < \infty, \ \ \ \ \   \text{and} \ \ \ \ \   \|u\|_{L^\infty(\Omega)} = \esssup_{\substack{x\, \in\, \Omega}} |u(x)|.
\]
Moreover, the Sobolev space \( W^{m,p}(\Omega) \) consists of functions \( u \in L^p(\Omega) \) whose weak derivatives up to order $m$ also belong to \( L^p(\Omega) \), with the associated norm 
\[
\| u \|_{W^{m,p}(\Omega)} = \left( \sum_{|\alpha| \leq m} \| D^\alpha u \|_{L^p(\Omega)}^p \right)^{1/p}\ \text{for}\   1 \leq p < \infty, \ \ \text{and }\ \ 
\| u \|_{W^{m,\infty}(\Omega)} = \max_{|\alpha| \leq m} \| D^\alpha u \|_{L^\infty(\Omega)}\ \ \text{for }\ p=\infty,
\]
where $D^{\alpha}$ is the weak derivative of multi-index order $\alpha$.
For $p=2$, we use the notation $H^m(\Omega):=W^{m,2}(\Omega)$, and the corresponding norm $\|\cdot\|_{H^m(\Omega)}=\|\cdot\|_{W^{m,2}(\Omega)}$, which defines a Hilbert space.

The time-dependent function space \( L^p(0,T;H^m(\Omega)) \) is defined as the set of functions \( u: (0,T) \times \Omega \to \mathbb{R} \) such that for almost every \( t \in (0,T) \), the function \( u(t, \cdot) \) belongs to \( H^m(\Omega) \), and the norm  
\[
\| u \|_{L^p(0,T;H^m(\Omega))} =
\begin{cases}
	\left( \int_0^T \| u(t) \|_{H^m(\Omega)}^p dt \right)^{1/p}, & \text{for } 1 \leq p < \infty, \\
	\operatorname{ess\,sup}_{t \in (0,T)} \| u(t) \|_{H^m(\Omega)}, & \text{for } p = \infty,
\end{cases}
\]
is finite. 

If we consider \( H^1(\Omega)^* \) to be the dual space of \( H^1(\Omega) \), then the duality pairing between them is represented as $\langle f, v \rangle_{H^1(\Omega)^*, H^1(\Omega)} \ \text{for any }  f \in H^1(\Omega)^*\ \text{and} \  v \in H^1(\Omega)$, and the dual norm is given by
\[
\| f \|_{H^1(\Omega)^*} = \sup_{\substack{v \in H^1(\Omega) \\ \| v \|_{H^1(\Omega)} = 1}} \big| \langle f, v \rangle_{H^1(\Omega)^*, H^1(\Omega)}\big|.
\]
Furthermore, the dual pairing between \( L^2(0,T;H^1(\Omega)) \) and \( L^2(0,T;H^1(\Omega)^*) \) is given by
\[
\langle f, u \rangle_{L^2(0,T;H^1(\Omega)^*), L^2(0,T;H^1(\Omega))} = \int_0^T \langle f(t), u(t) \rangle_{H^1(\Omega)^*, H^1(\Omega)} \, dt,
\]
and the dual norm on $L^2(0,T;H^1(\Omega)^*)$ is defined as  $\|f\|_{L^2(0,T;H^1(\Omega)^*)} = \left(\int_0^T \|f(t)\|_{H^1(\Omega)^*}^2\, dt\right)^{1/2}$.

For any functional $f$ and $g$ in $H^1(\Omega)^*$, Riesz representation theorem guarantees the existence of two unique functions $u_f$ and $u_g$ in $H^1(\Omega)$ such that, for every $v\in H^1(\Omega)$,
$$\langle f,v\rangle_{H^1(\Omega)^*\times H^1(\Omega)} = \big(u_f,v\big)_{H^1(\Omega)}\ \ \text{and}\ \  \langle g,v\rangle_{H^1(\Omega)^*\times H^1(\Omega)} = \big(u_g,v\big)_{H^1(\Omega)}.$$ 
Based on this, we define the inner product in $L^2(0,T;H^1(\Omega)^*)$ space as follows
$$\langle \langle f,g\rangle\rangle_{L^2(0,T;H^1(\Omega)^*} = \int_0^T \langle f,g \rangle_{H^1(\Omega)^*\times H^1(\Omega)^*} \ dt  =\int_0^T (u_f,u_g)_{H^1(\Omega)}\ dt.$$

Since we require a certain continuity property of the control function with respect to time, we define the admissible control space as 
	$$\mathcal{U}:=W^{1,2}(0,T;H^1(\Omega),H^1(\Omega)^*)=\left\{u\in L^2(0,T;H^1(\Omega))\ | \ u_t\in L^2(0,T;H^1(\Omega)^*)\right\},$$
with norm
$$\|u\|_{\mathcal{U}}=\|u\|_{W^{1,2}(0,T;H^1(\Omega),H^1(\Omega)^*)}=\|u\|_{L^2(0,T;H^1(\Omega))}+\|u_t\|_{L^2(0,T;H^1(\Omega)^*)},$$
and inner product
$$((u,v))_{\mathcal{U}}=(u,v)_{L^2(0,T;H^1(\Omega))}+\langle \langle u_t,v_t\rangle \rangle_{L^2(0,T;H^1(\Omega)^*)}.$$
	
\noindent	We define the following sets of admissible solution spaces:
	\begin{flalign*}
		\mathcal{M}&:= W^{1,2}(0,T;H^3(\Omega),H^1(\Omega))\ =\left\{m \in L^2(0,T;H^3(\Omega)) \ |\ m_t \in L^2(0,T;H^1(\Omega))\right\}.
	\end{flalign*}
	A function $m\in \mathcal{M}$ satisfying system \eqref{NLP} almost everywhere corresponding to a control $u\in \mathcal{U}$ is called a \textbf{regular solution}.
	


\begin{Thm}(Existence of Regular Solution for n\,=\,1,2,3)\label{T-RS} 
	Let $u\in L^2(0,T;H^1(\Omega))$ be the control and $m_0$ be the initial data satisfying condition \eqref{IC}. For $n=2$, the system \eqref{NLP} admits a unique global regular solution $m\in \mathcal{M}$. Moreover, there exist constants $M(m_0,u,\Omega,T,n)$ and $C(\Omega,T)>0$ such that the following estimate holds:
	\begin{equation}\label{SSEE2}
		\|m\|^2_{L^\infty(0,T;H^2(\Omega))} + \|m\|^2_{L^2(0,T;H^3(\Omega))} \leq 	M(m_0,u,\Omega,T),
	\end{equation}
	\begin{flalign*}
		\text{where}\ \ M(m_0,u,\Omega,T) &:= \left(\|m_0\|^2_{L^2(\Omega)}+\|\Delta m_0\|^2_{L^2(\Omega)}+\|u\|^2_{L^2(0,T;H^1(\Omega))}\right)&\nonumber\\
		&  \times  \exp\left\{C \left[1+\|m_0\|^2_{H^1(\Omega)} + \|u\|^2_{L^2(0,T;L^2(\Omega))} \right]^2 \ \exp\left\{C\left(\|u\|^2_{L^2(0,T;H^1(\Omega))}\right)\right\}\right\}.
	\end{flalign*}
	For $n=3$, the system \eqref{NLP} admits a unique ``local in time" regular solution for every control in $L^2(0,T;H^1(\Omega))$. However, there exists a constant $\widetilde{C}(\Omega)>0$ such that under the following smallness assumption:
	\begin{equation}\label{SAOIC}
		M(m_0,u,\Omega,T) < \frac{1}{\widetilde{C}^{\frac{1}{2}}},
	\end{equation}
	the solution exists globally and satisfies the energy estimate \eqref{SSEE2}.
\end{Thm}
The proof of Theorem \ref{T-RS} can be found in Theorem 2.2, \cite{SPSK3}. The same proof holds for the case $n=1$. However, note that for $n=1$ we can get the optimality results just with the strong solutions, that is, solutions in $W^{1,2}(0,T;H^2(\Omega),L^2(\Omega))$ space with a control function in $L^2(0,T;L^2(\Omega))$. 

Furthermore, as a consequence of Theorem \ref{T-RS}, in order to deal with the solvability of \eqref{NLP} for $n=3$, we define a set
$$\mathcal{U}_R:= \Big\{u\in \mathcal{U} \text{ such that } u \text{ satisfies the smallness assumption \eqref{SAOIC} for a suitable } m_0 \Big\}.$$



	
	Let $\mathcal{U}_{ad}$ be a non-empty, closed, convex subset of $\mathcal{U}$ for $n=1,2$ and of $\mathcal{U}_R$ for $n=3$. 
For any control $u\in \mathcal{U}_{ad}$, an admissible pair $(T_u,u)$ is defined as a solution to system \eqref{NLP}, with the corresponding state $m_u\in \mathcal{M}$ and time $T_u=T^*(u)$ belongs to $(0,T]$. The set of all admissible pairs is denoted by $\mathcal{A} \subset (0,T] \times \mathcal{U}_{ad}$. Thus, the time optimal control problem is formulated as follows:
	\begin{equation*}\label{OCP}
		\text{(TOCP)}\begin{cases}
			\text{minimize}\ \mathcal J(T,u),\\
			(T,u) \in \mathcal{A}.
		\end{cases}	
	\end{equation*}
	
\begin{Ass}\label{ASS-T}
		Assume that there exists a control $u\in \mathcal{U}_{ad}$ such that $T^*(u)$ exists in $[0,T]$.
\end{Ass}
	\noindent This assumption is crucial, as it guarantees the existence of a control $u$ for which the cost functional $\mathcal{J}(T^*(u),u)<+\infty$. Now, we state the existence of a time-optimal control for TOCP.
	%
		
	\begin{Thm}(Existence of Optimal Control)\label{T-EOOC}
	Suppose for $n=1,2,3$ the initial data $m_0$ satisfies the condition \eqref{IC}, and for $n=3$, it also satisfies the smallness condition \eqref{SAOIC}. Under these conditions and  Assumption \ref{ASS-T}, the time optimal control problem (TOCP) admits a solution $\widetilde{u} \in \mathcal{U}_{ad}$, with the corresponding optimal time $\widetilde{T}$. 
	\end{Thm}
\noindent The proof of this theorem is given in Section \ref{SEC-EOC}.

Next, we proceed to state the first-order and second-order optimality conditions using the adjoint problem approach. Our proofs rely on the following two key assumptions:
\begin{equation}\label{CON-1}
	0<\widetilde{T}<T,
\end{equation}
	and
	\begin{equation}\label{CON-2}
\big(\widetilde{m}(\widetilde{T})-m_{\Omega},\widetilde{m}_t(\widetilde{T})\big)_{L^2(\Omega)}<0.		
	\end{equation}

\begin{Rem}\label{MT-WPN}
	Since $\widetilde{u} \in W^{1,2}(0,T;H^1(\Omega);H^1(\Omega)^*)$ and $\widetilde{m} \in W^{1,2}(0,T;H^3(\Omega);H^1(\Omega))$, Theorem 2 in Section 5.9 of \cite{LCE} implies that $\widetilde{u} \in C([0,T];L^2(\Omega))$ and $\widetilde{m} \in C([0,T];H^2(\Omega))$, respectively. Consequently, when we substitute this continuity property in equation \eqref{NLP}, it follows that the time derivative $\widetilde{m}_t\in C([0,T];L^2(\Omega))$. Hence, the value $\widetilde{m}_t(\widetilde{T})$ is well-defined.
\end{Rem}
\noindent The assumptions \eqref{CON-1} and \eqref{CON-2} can be explored as follows.  From an application standpoint, we require the trajectory $\widetilde m$ to \emph{enter} the $\delta$-tube 
\[\mathcal T_\delta := \bigl\{\, m \in L^2(\Omega;\mathbb R^3) : \|m - m_\Omega\|_{L^2(\Omega)} \le \delta \,\bigr\}\]
around the desired state $m_\Omega$ within the horizon $(0,T)$. Consistent with the objective functional \eqref{T-CF}, we consider controls of moderate size whose minimal time $T^*(u)$ lies in $[0,T]$. To rule out mere boundary touching, we impose a strict inward crossing at the minimal time $\widetilde{T}$, that is, if we denote $\phi(t):=\tfrac12\|\,\widetilde m(t)-m_\Omega\,\|_{L^2(\Omega)}^2$, then assumption (2.4) follows from \[\phi'(\widetilde T)=\bigl(\widetilde m(\widetilde T)-m_\Omega,\ \widetilde{m}_t(\widetilde T)\bigr)_{L^2(\Omega)}<0,\]
which ensures entry into the interior of $\mathcal T_\delta$ rather than a grazing contact.

To exclude the trivial case at $t=0$, we assume $\|\widetilde{m}(0)-m_\Omega\|_{L^2(\Omega)}>\delta$, that is, $T^*(\widetilde u)=\widetilde T>0$. Finally, to preclude terminal grazing (first contact occurring only at $t=T$), we assume $\widetilde T<T$, which guarantees interior reachability strictly before the final time.

\begin{Pro}\label{L-EQ}
	Suppose $u$ be any control in $W^{1,2}(0,T;H^1(\Omega),H^1(\Omega)^*)$ and $m\in \mathcal{M}$ be its associated state. If $T^*(u)=\inf\, \big\{t\in [0,T]: \| m(t)-m_{\Omega}\|_{L^2(\Omega)}\leq \delta \big\}$ exists in $[0,T]$, then we have $\|m(T^*(u))-m_{\Omega}\|_{L^2(\Omega)}=\delta$.
\end{Pro}	
\noindent The proof is straightforward and follows directly from standard arguments (see Lemma 4.1 of \cite{CTA}).

\noindent As a consequence of Proposition \ref{L-EQ}, we can reformulate our time-optimal control problem as follows:
\begin{equation}\label{N-TOCP}
	\begin{cases}
		\text{minimize}	\ \ \mathcal{J}=\frac{1}{2} \ T^{\#^2} +\frac{1}{2} \  \|u\|^2_{W^{1,2}(0,T;H^1(\Omega);H^1(\Omega)^*)},\\
		\text{subject to} \ \ \ T^\# \in [0,T],\\
		\hspace{2cm} (m,u) \ \text{satisfies system}\ \eqref{NLP},\\
		\hspace{2cm} \|m(T^\#)-m_{\Omega}\|_{L^2(\Omega)}=\delta.
	\end{cases}
\end{equation}
Next, we introduce the Lagrange multipliers $\lambda$ and $\phi(x,t)$, and define the Lagrangian functional $\mathscr{L}$ for the time optimal control problem \eqref{N-TOCP} as follows:
\begin{align*}
	\mathscr{L}(\tau,u,m,\lambda,\phi)&= \mathcal{J}(u) - \int_{\Omega_T} 	\big( m_t- \Delta m- m \times \Delta m  - m \times u+ \big(1+ |m|^2\big)\,m - u\big)\, \phi\ dx\ dt\\
	&\hspace{1cm}  - \lambda\left(\frac{1}{2}\|m(\tau)-m_{\Omega}\|^2_{L^2(\Omega)}-\frac{1}{2}\delta^2\right).
\end{align*}
If we consider $\widetilde{T},\widetilde{u}$, and $\widetilde{m}$ as the optimal time, control and state, respectively, then the Gateaux derivative of $\mathscr{L}$ with respect to $\tau,m$ satisfies
$$D_{(t,m)}\mathscr{L}(\widetilde{T},\widetilde{u},\widetilde{m},\lambda,\phi)[\mathcal{S},y-\widetilde{m}] \geq 0\ \ \ \ \ \ \text{for all sufficiently smooth y with }\ y(0)=m_0.$$
Now, for any $y$ such that $y(0)=0$, substituting $y+\widetilde{m}$ in place of $y$ in the above inequality, and noting that the inequality holds for both $y$ and $-y$, we obtain  
\begin{align}\label{ASD-1}
	&D_{(t,m)}\mathscr{L}(\widetilde{T},\widetilde{u},\widetilde{m},\lambda,\phi)[\mathcal{S},y]= \widetilde{T}\mathcal{S}- \lambda\ \Big(\widetilde{m}(\widetilde{T})-m_{\Omega},\widetilde{m}_t(\widetilde{T})\,\mathcal{S}\Big)_{L^2(\Omega)} \nonumber\\
	&\hspace{1cm} -\int_{\Omega_{\widetilde{T}}} \phi\ \left( y_t- \Delta y- \widetilde{m} \times \Delta y  -y \times \Delta \widetilde{m} - y \times \widetilde{u} + \big(1+ |\widetilde{m}|^2\big)\,y +2\big(\widetilde{m} \cdot y\big)\, \widetilde{m} \right)\ dx\ dt\nonumber\\
	&\hspace{1cm} -\lambda\  \Big(\widetilde{m}(\widetilde{T})-m_{\Omega},y(\widetilde{T})\Big)_{L^2(\Omega)} =0\ \ \ \ \ \ \ \ \ \text{for all y with }\ y(0)=0.   
\end{align}
	Now, by taking the multiplier $\phi$ to be 0, we can find that
	\begin{equation}\label{ASD-2}
	\lambda=\frac{\widetilde{T}}{\Big(\widetilde{m}(\widetilde{T})-m_{\Omega},\widetilde{m}_t(\widetilde{T})\Big)_{L^2(\Omega)}}.		
	\end{equation}
Now, doing space and time integration by parts of the respective terms of \eqref{ASD-1}, and using the cross product property $a\cdot(b\times c) = -b\cdot(a \times c)$, we derive
\begin{align*}
	&\int_{\Omega_{\widetilde{T}}} y\ \Big({\phi}_t+\Delta \phi + \Delta (\phi \times \widetilde{m})-\phi \times \Delta \widetilde{m} -\phi \times \widetilde{u} - \left(1+|\widetilde{m}|^2\right)\phi -2\, \widetilde{m}\, \big(\widetilde{m}\cdot\phi \big) \Big) \ dx\ dt\nonumber\\
	&\hspace{1cm} - \big(\phi(\widetilde{T}),y(\widetilde{T})\big) -\int_{\partial \Omega_{\widetilde{T}}} y \left(\frac{\partial \phi}{\partial \eta} +\frac{\partial \phi}{\partial \eta}\big(\phi \times  \widetilde{m}\big)\right)\ ds(x)\ dt - \lambda\ \Big(\widetilde{m}(\widetilde{T})-m_{\Omega},y(\widetilde{T})\Big)_{L^2(\Omega)}=0.
\end{align*}

Finally, following the formal Lagrange method similar to Section 2.10 of \cite{FT}, and substituting the value of $\lambda$ from equation \eqref{ASD-2}, we can formulate the following adjoint system:
	\begin{equation}\label{AS}
		\begin{cases}
			\phi_t + \Delta \phi  + \Delta (\phi \times \widetilde{m})+(\Delta  \widetilde{m}\times \phi)-(\phi \times \widetilde{u}) -\left(1+|\widetilde{m}|^2\right)\phi -2\ \big(\widetilde{m}\cdot \phi\big) \ \widetilde{m} =0 \ \ \ \text{in}\ \Omega\times [0,\widetilde{T}],\vspace{0.2cm}\\
			\frac{\partial \phi}{\partial \eta}=0 \ \ \ \text{in}\ \partial \Omega\times [0,\widetilde{T}],\vspace{0.1cm}\\
			\phi(\widetilde{T})=-\frac{\widetilde{T}\ \big(\widetilde{m}(\widetilde{T})-m_{\Omega}\big)}{\big(\widetilde{m}(\widetilde{T})-m_{\Omega},\widetilde{m}_t(\widetilde{T})\big)_{L^2(\Omega)}}\ \ \ \text{in} \ \Omega.
		\end{cases}
	\end{equation}

	
\noindent  Next, we define the weak formulation of the adjoint problem as follows:

	
	
	\begin{Def}[Weak formulation]\label{AWSD}
		Suppose $\widetilde{m}\in \mathcal{M}$, $\widetilde{u}\in \mathcal{U}$, $m_{\Omega}\in L^2(\Omega)$ and $\widetilde{m}_t(\widetilde{T})\in L^2(\Omega)$. A function $\phi \in \mathcal{U}$ is said to be a weak solution of system \eqref{AS} if for every $\vartheta \in L^2(0,\widetilde{T};H^1(\Omega))$, the following holds:
		\begin{align}\label{WEFO}
			(i)&\int_0^{\widetilde{T}} \langle \phi'(t),\vartheta (t)\rangle_{H^1(\Omega)^*\times H^1(\Omega)}dt - \int_{\Omega_{\widetilde{T}}} \nabla \phi \cdot \nabla  \vartheta\ dx \ dt- \int_{\Omega_{\widetilde{T}}}\nabla (\phi \times \widetilde{m})\cdot \nabla  \vartheta\ dx \ dt-\int_{\Omega_{\widetilde{T}}} \big(\phi \times \widetilde{u}\big)\cdot \vartheta\ dx\ dt\nonumber\\
			&\hspace{1cm}+ \int_{\Omega_{\widetilde{T}}}(\Delta \widetilde{m}\times \phi)\cdot \vartheta\ dx\ dt- \int_{\Omega_{\widetilde{T}}} \left( 1 + |\widetilde{m}|^2 \right) \phi\cdot \vartheta\  dx\ dt -2\int_{\Omega_{\widetilde{T}}} \big((\widetilde{m} \cdot \phi)\ \widetilde{m}\big)\cdot \vartheta \ dx\ dt =0,  \vspace{0.1cm}\\
			(ii)& \ \ \phi({\widetilde{T}})=-\frac{\widetilde{T}\ \big(\widetilde{m}(\widetilde{T})-m_{\Omega}\big)}{\big(\widetilde{m}(\widetilde{T})-m_{\Omega},\widetilde{m}_t(\widetilde{T})\big)_{L^2(\Omega)}}\ \ \ \text{ in the trace sense.}\nonumber
    	\end{align}
	\end{Def}
	

	\begin{Thm}\label{T-AWS}
		If $(\widetilde{T},\widetilde{u})$ constitutes an admissible pair and $\widetilde{m}$ is the corresponding regular solution, then there exists a unique weak solution $\phi \in \mathcal{U}$ for the adjoint system \eqref{AS}, in the sense of Definition \ref{AWSD}. Moreover, the following inequality holds:
		\begin{flalign*}
			\|\phi\|^2_{L^{\infty}(0,T;L^2(\Omega))}&+ \|\phi\|^2_{L^2(0,T;H^1(\Omega))} + \|\phi_t\|^2_{L^2(0,T;H^1(\Omega)^*)}\nonumber\\
			& \leq \|\phi(\widetilde{T}) \|^2_{L^2(\Omega)} \exp\bigg\{C  \left(1+  \|\widetilde{m}\|^4_{L^\infty(0,T;H^1(\Omega))} + \|\widetilde{m}\|^2_{L^2(0,T;H^2(\Omega))} +\|\widetilde{u}\|^2_{L^2(0,T;L^2(\Omega))} \right) \bigg\}.	
		\end{flalign*}
	\end{Thm}
\noindent The proof of Theorem \ref{T-AWS} follows from the arguments of Lemma 4.3 of \cite{SPSK3}.
	
	Now, for any convex subset $S$ of a Hilbert space $\mathcal{H}$, we define the normal cone and the polar cone of tangents of $S$ at $x\in S$ as follows:
	\begin{align*}
	\mathcal{N}_S(x)&=\{ \xi\in \mathcal{H} \ : \ (\xi,y-x)\leq 0\ \ \ \ \ \ \forall \ y\in S\},\\
	\mathcal{P}_S(x)&= \{ \xi \in \mathcal{H}\ :\ (\xi,y)\leq 0\ \ \ \ \ \ \  \forall \ y\in \mathcal{N}_S(x)\}.
	\end{align*}

A feasible direction $\upsilon \in \mathcal{H}$ at $x\in S$ is an element such that $x+\epsilon \upsilon \in S$ for small enough $\epsilon \in \mathbb{R}^+$. Moreover, we denote the cone of feasible directions at $x\in S$  as $\mathcal{F}_S(x)$. Moreover, convexity of $S$ implies that $\overline{\mathcal{F}_S(x)}=\mathcal{P}_S(x)$.

	Now we are ready to state the first-order optimality condition satisfied by the optimal control $\widetilde{u}\in \mathcal{U}_{ad}$ and optimal time $\widetilde{T}$.
	
	\begin{Thm}[First-Order Optimality Condition]\label{T-FOOCT}
		Suppose $\widetilde{T}$ is the optimal time corresponding to the optimal control $\widetilde{u}$ and state $\widetilde{m}$ satisfying conditions \eqref{CON-1} and \eqref{CON-2}. Moreover, assume that there exists a control $u\in \mathcal{U}_{ad}$ such that the cost functional $\mathcal{J}(T^*(u),u)<+\infty$. If $\phi \in \mathcal{U}$ is the weak solution of the adjoint system \eqref{AS} corresponds to the admissible pair $(\widetilde{m}, \widetilde{u})$, then the triplet $(\phi, \widetilde{m}, \widetilde{u})$ satisfies the following variational inequality: 		
		\begin{equation}\label{FOOC}
		\mathcal{Y}(h):=	\int_{\Omega_{\widetilde{T}}} \big(\phi + \phi \times \widetilde{m}\big)\cdot \ h\  dx\ dt  +\big(\big(\widetilde{u},h\big)\big)_{W^{1,2}(0,T;H^1(\Omega),H^1(\Omega)^*)}\geq 0, \ \ \ \ \  \forall \ h \in \mathcal{P}_{\mathcal{U}_{ad}}(\widetilde{u}).	
		\end{equation}
	\end{Thm}
	The proof of this theorem is given in Subsection \ref{S-FOOC}.
	
	Moving forward, we will explore the local second-order optimality condition that the optimal control adheres to. First, we will define the set 
	\begin{equation}\label{SET-EQ0}
		\Upsilon(\widetilde{u}):= \Big\{  h\in W^{1,2}(0,T;H^1(\Omega),H^1(\Omega)^*) \ :\ 	\mathcal{Y}(h)=0\big\}.
	\end{equation}	
	\begin{Thm}[Second-Order Optimality Condition]\label{T-SOOC}
		Suppose $\widetilde{u}$ is a control in $\mathcal{U}_{ad}$, and $\widetilde{m}$ is the corresponding state solution. Let $z_h\in \mathcal{M}$ be the regular solution of the system \eqref{T-LS} corresponding to any $h\in \mathcal{U}$. Assume that, along with conditions \eqref{CON-1} and \eqref{CON-2}, the function $t\mapsto \widetilde{m}(t)$ is twice Fr\'echet differentiable at $\widetilde{T}$. If the triplet $(\widetilde{T},\widetilde{u},\widetilde{m})$ satisfies the first-order optimality condition \eqref{FOOC} and the following inequality:
		\begin{equation}\label{SO-ME}
			\mathcal{Q}(h):=\mathcal{D}_{\widetilde{u}}[h]^2+2\ \widetilde{T} \ \mathcal{G}_{\widetilde{u}}[h,h] + \|h\|^2_{W^{1,2}(0,T;H^1(\Omega),H^1(\Omega)^*)}>0,
		\end{equation}
		for all $h\in \Big(\mathcal{P}_{\mathcal{U}_{ad}}(\widetilde{u}) \cap \Upsilon(\widetilde{u})\Big)\backslash \{0\} $, where 
		\begin{equation}\label{SOC-D}
			\mathcal{D}_{\widetilde{u}}[h]=- \frac{\big( z_h ( \widetilde{T}) , \widetilde{m}(\widetilde{T})-m_{\Omega} \big)_{L^2(\Omega)}}{\big(\widetilde{m}(\widetilde{T})-m_{\Omega},\widetilde{m}_t(\widetilde{T})\big)_{L^2(\Omega)}},
		\end{equation}
		and 
		\begin{align}\label{SOC-GN}
			&\big(\widetilde{m}_t(\widetilde{T}), \widetilde{m}(\widetilde{T})-m_{\Omega}\big)_{L^2(\Omega)}\,\mathcal{G}_{\widetilde{u}}[h,h] =-\frac{1}{2}\  \big\|\, \widetilde{m}_t(\widetilde{T})\, \mathcal{D}_{\widetilde{u}}[h] + z_h(\widetilde{T})\,\big\|^2_{L^2(\Omega)}\nonumber\\
			&\hspace{0.4cm}-\bigg( \frac{1}{2}\ \widetilde{m}_{tt}(\widetilde{T})\,\mathcal{D}_{\widetilde{u}}[h]^2  + z_{h_t}(\widetilde{T}) \mathcal{D}_{\widetilde{u}}[h]    ,\ \widetilde{m}(\widetilde{T}) -m_{\Omega}\bigg)_{L^2(\Omega)}\nonumber\\
			&\hspace{0.4cm}+\big(\widetilde{m}_t(\widetilde{T}), \widetilde{m}(\widetilde{T})-m_{\Omega}\big)_{L^2(\Omega)}\ \frac{1}{\widetilde{T}}\int_{\Omega_{\widetilde{T}}} \Big( z_h\times \Delta z_h+ z_h\times h -2 \, \big(z_h\cdot \widetilde{m}\big)\,z_h - |z_h|^2\, \widetilde{m} \Big)\ \phi \ dx\ dt,
		\end{align}
		then $\widetilde{u}$ is a local optimal control. That is, there exist $\epsilon>0$ and $\sigma>0$ such that the following inequality holds:
		\begin{equation}\label{SO-GC}
			\mathcal{J}(u) \geq \mathcal{J}(\widetilde{u}) +\epsilon \ \|u-\widetilde{u}\|^2_	{W^{1,2}(0,T;H^1(\Omega),H^1(\Omega)^*)},
		\end{equation}
		for every $u\in \mathcal{U}_{ad}$ with $\|u-\widetilde{u}\|_{W^{1,2}(0,T;H^1(\Omega),H^1(\Omega)^*)}\leq \sigma$.
	\end{Thm}
A detailed proof of this theorem is presented in Section \ref{SEC-SO}.

		The validity of the term $\widetilde{m}_{tt}(\widetilde{T})\,\mathcal{D}_{\widetilde{u}}[h]^2 $ relies on our assumption on the time regularity of the state variable $m$ stated in Theorem \ref{T-SOOC}. Though we have not given a detailed proof of twice differentiability of $\widetilde{m}$ at $\widetilde T,$ we have proved a lemma that gives under what conditions we can indeed achieve this temporal regularity of $\widetilde{m}$.		
		\begin{Lem}\label{L-MTT} Let $m\in W^{1,2}(0,T;H^3(\Omega),H^1(\Omega))$ be the regular solution to the system \eqref{NLP}. Assume that the control $u\in L^2(0,T;H^3(\Omega))$ and the initial data $m_0\in H^4(\Omega)$. Then, the corresponding solution has the regularity $m\in L^2(0,T;H^5(\Omega))\cap C([0,T];H^4(\Omega))$ with $m_t\in C([0,T];H^2(\Omega))$.\\			
		Moreover, if we assume further regularity on the control $u_t\in C([0,T];L^2(\Omega))$, then the temporal regularity $m_{tt}\in C([0,T];L^2(\Omega))$ holds. 
		\end{Lem}	
        The proof of Lemma \ref{L-MTT} is given in Appendix \ref{app}. 
		Nevertheless, to maintain the generality of our results, we have chosen to impose this time regularity as an assumption in Theorem \ref{T-SOOC} rather than putting all these higher-order regularity conditions on the control and initial data.

\subsection{Inequalities}
	
	\noindent In this subsection, we introduce several equivalent norms and key inequalities that are fundamental to our analysis. We start by examining the essential properties of the cross product and providing equivalent norm estimates, which are outlined in the following lemmas.
	\begin{Lem}\label{CPP}
		Let $a,b$ and $c$ be three vectors of $\mathbb{R}^3$, then the following vector identities hold: $a\cdot(b \times c)=-(b \times a)\cdot c$  and $a \cdot (a \times b)=0$.  Moreover, assume that $1 \leq r,s \leq \infty, \ (1/r)+(1/s)=1$ and $p\geq 1$, then if $f\in L^{pr}(\Omega)$ and $g\in L^{ps}(\Omega),$ we have
		\begin{equation*}\label{ES0}
			\|f \times g\|_{L^p(\Omega)} \leq \|f\|_{L^{pr}(\Omega)} \|g\|_{L^{ps}(\Omega)}.	
		\end{equation*}	
	\end{Lem}
	
	\begin{Lem}[see, \cite{KW}]\label{EN}
		Let $\Omega$ be a bounded smooth domain in $\mathbb{R}^n$ and $k \in \mathbb{N}$. There exists a constant $C_{k,n}>0$ such that for all $m \in H^{k+2}(\Omega)$ and $\frac{\partial m}{\partial \eta}\big|_{\partial \Omega}=0,$ it holds that
		\begin{equation*}\label{ES1}
			\|m\|_{H^{k+2}(\Omega)} \leq C_{k,n} \left(\|m\|_{L^2(\Omega)}+ \|\Delta m\|_{H^k(\Omega)}\right).	
		\end{equation*}
	\end{Lem}
	\noindent Using Lemma \ref{EN}, we can define an equivalent norm on $H^{k+2}(\Omega)$ as follows 
	$$\|m\|_{H^{k+2}(\Omega)}:=\|m\|_{L^2(\Omega)}+\|\Delta m\|_{H^k(\Omega)}.$$
	
	\begin{Pro}\label{PROP1}
	Let $\Omega$ be a regular bounded subset of  $\mathbb{R}^2$ or $\mathbb{R}^3$. There exists a constant $C>0$ depending on $\Omega$ such that for all $m \in H^2(\Omega)$ with $\frac{\partial m}{\partial \eta}\big|_{\partial\Omega}=0,$ we have
	\begin{eqnarray} 
		\|m\|_{L^\infty(\Omega)} &\leq& C\ \left(\|m\|^2_{L^2(\Omega)}+ \|\Delta m\|^2_{L^2(\Omega)}\right)^{\frac{1}{2}},\label{ES2}\\
		\|\nabla m\|_{L^s(\Omega)} &\leq& C\  \|\Delta m\|_{L^2(\Omega)}, \ \ \forall \ s \in [1,6],\label{ES3}\\
		\|D^2m\|_{L^2(\Omega)} &\leq& C\ \|\Delta m\|_{L^2(\Omega)}.\label{ES4}	
	\end{eqnarray}
	Moreover, for every $m \in H^3(\Omega)$ with $\frac{\partial m}{\partial \eta}\big|_{\partial\Omega}=0,$ we have
	\begin{eqnarray}
		\|\Delta m\|_{L^2(\Omega)} &\leq& C\ \|\nabla \Delta m\|_{L^2(\Omega)},\label{ES7}\\ 	
		\|D^3m\|_{L^2(\Omega)} &\leq& C\ \|\nabla \Delta m\|_{L^2(\Omega)},\label{ES10}\\
		\|D^2m\|_{L^3(\Omega)} &\leq& C\ \|\Delta m\|^\frac{1}{2}_{L^2(\Omega)} \|\nabla \Delta m\|^\frac{1}{2}_{L^2(\Omega)}.\label{ES9} 		
	\end{eqnarray}
	\end{Pro}
	The proof of this proposition can be found in Proposition 2.1, \cite{SPSK}.
	
\begin{Lem}\label{PROP2}
Let $\Omega$ be a regular bounded domain of $\mathbb{R}^n$ for $n=2,3$. Then there exists a constant $C>0$ depending on $\Omega$ and $T$ such that 
\begin{enumerate}[(\roman*)]
\item  for $\xi \in L^\infty(0,T;H^2(\Omega))$ and $\zeta\in L^2(0,T;H^3(\Omega))$,
\begin{equation}\label{EE-2}
	\|\xi \times \Delta \zeta\|^2_{L^2(0,T;H^1(\Omega))}\leq  C\ \|\xi\|^2_{L^\infty(0,T;H^2(\Omega))}\  \|\zeta\|^2_{L^2(0,T;H^3(\Omega))},
\end{equation}
\item for $\xi\in L^\infty(0,T;H^2(\Omega))$ and $\omega \in L^2(0,T;H^1(\Omega))$,
\begin{equation}\label{EE-3}
	\|\xi \times \omega\|^2_{L^2(0,T;H^1(\Omega))} \leq  C\ \|\xi\|^2_{L^\infty(0,T;H^2(\Omega))}\  \|\omega\|^2_{L^2(0,T;H^1(\Omega))},
\end{equation}
\item  for $\xi,\zeta,\omega \in L^\infty(0,T;H^2(\Omega))$,
\begin{equation}\label{EE-6}
	\|(\xi\cdot \zeta)\ \omega\|^2_{L^2(0,T;H^1(\Omega))} \leq C\ \|\xi\|^2_{L^\infty(0,T;H^2(\Omega))}\ \|\zeta\|^2_{L^\infty(0,T;H^2(\Omega))}\ \|\omega\|^2_{L^2(0,T;H^2(\Omega))}.
\end{equation}
\end{enumerate}
\end{Lem}
The proofs of the first two estimates in Lemma \ref{PROP2} are given in Lemma 2.3 of \cite{SPSK2}. Moreover, the final estimate \eqref{EE-6} can be derived analogously.

\section{Existence of Optimal Control}\label{SEC-EOC}			
\begin{proof}[Proof of Theorem \ref{T-EOOC}]	
	
	From Assumption \ref{ASS-T}, the existence of a control $u\in \mathcal{U}_{ad}$ and the corresponding minimal time $T^*(u)$ with non-negative cost functional implies that there exists a minimizing sequence $\{T^*(u_k),u_k\}\subseteq \mathcal{A}$ and an infimum value $\alpha$ such that 
	\begin{equation}\label{EOC-1}
		\inf_{(T^*(u),u)\in \mathcal{A}} \mathcal{J}(u)=\lim_{k\to \infty} \mathcal{J}(u_k)=\alpha <+\infty.
	\end{equation} 
	As $T^*(u_k)$ is a bounded sequence of real numbers, let us denote $\displaystyle \inf_k T^*(u_k)= \widetilde{T}$. So, there exists a subsequence (again represented as $u_k$) such that 
	\begin{equation}\label{EOC-2}
		\lim_{k\to \infty} T^*(u_k)= \widetilde{T}. 
	\end{equation}
	Since $\{u_k\}$ is a minimizing sequence for the cost functional $\mathcal{J}$, it  is bounded in $\mathcal{U}$. Therefore, we can extract a subsequence (again represented as $\{u_k\}$) such that $u_k \overset{w}{\rightharpoonup}\widetilde{u}$ weakly in $\mathcal{U}$ for some element $\widetilde{u}\in \mathcal{U}$. Moreover, as the set $\mathcal{U}_{ad}$ is a closed and convex set, it is weakly closed, that is, $\widetilde{u}\in \mathcal{U}_{ad}$. Furthermore, by virtue of estimate \eqref{SSEE2}, we find that $\{m_k\}$  is bounded  in $\mathcal{M}$. As a consequence, there exists a weak limit $\widetilde{m}\in\mathcal{M}$ such that $m_k \rightharpoonup \widetilde{m}$ weakly in $\mathcal{M}$. Finally, by using the Aubin–Lions–Simon compactness theorem, we conclude that the sequence $\{m_k\}$ is relatively compact in the space $C([0,T];H^1(\Omega)) \cap L^2(0,T;H^2(\Omega))$. As a result, there exist subsequences (again denoted as $\{ m_k\}$ and $\{u_k\}$ ) such that
	\begin{eqnarray} \left\{\begin{array}{ccccl}
			u_k &\overset{w}{\rightharpoonup} & \widetilde{u} \ &\mbox{weakly in}&  L^2(0,T;H^1(\Omega)),\\
			m_k &\overset{w}{\rightharpoonup} & \widetilde{m} \  &\mbox{weakly in}&  L^2(0,T;H^3(\Omega)),\\
			m_k &\overset{*}{\rightharpoonup} & \widetilde{m} \  &\mbox{weak$^*$ in}&  L^\infty(0,T;H^2(\Omega)),\\
			(m_k)_t &\overset{w}{\rightharpoonup} & \widetilde{m}_t \ &\mbox{weakly in}&  L^2(0,T;H^1(\Omega))\ \text{and}\\
			m_k &\overset{s}{\to} & \widetilde{m} \  &\mbox{strongly in}& C([0,T];H^1(\Omega))\cap L^2(0,T;H^2(\Omega))\ \ \ \ \mbox{as} \  \ k\to \infty. \label{P2}
		\end{array}\right.	
	\end{eqnarray}
	
	Therefore, implementing the convergences established in \eqref{P2}, we can show the convergences of the nonlinear terms $m_k\times \Delta m_k$ , $m_k \times u_k$, $\big(1+|m_k|^2\big)m_k$ to $\widetilde{m} \times \Delta \widetilde{m}$, $\widetilde{m} \times \widetilde{u}$ and $\big(1+|\widetilde{m}|^2\big)\widetilde{m}$ respectively. For more details on these convergences, one may refer to Lemma 4.1 of \cite{SPSK}. This leads us to conclude that  $\widetilde{m}$ is a regular solution of system \eqref{NLP}.

	Now, the main objective is to prove that, within the finite time interval $[0,T]$, the control $\widetilde{u}$ drives the state $\widetilde{m}$ to a distance $`\delta\,'$ close to the desired state $m_{\Omega}$ in the $L^2(\Omega)$ space, that is, $\|\widetilde{m}(\widetilde{T})-m_{\Omega}\|\leq \delta$. Note that the constant $\delta$ is coming from the definition of the cost functional \eqref{T-CF}.

	First, let us denote $T^*(u_k)=T^*_k$, then we can formulate
	$$\big\|\widetilde{m}(\widetilde{T}) -m_{\Omega}\big\|_{L^2(\Omega)} \leq \big\|\widetilde{m} (\widetilde{T})-\widetilde{m}(T^*_k)\big\|_{L^2(\Omega)}+\big\|\widetilde{m}(T^*_k)-m_k(T^*_k)\big\|_{L^2(\Omega)}+\big\|m_k(T_k^*)-m_{\Omega}\big\|_{L^2(\Omega)}$$
	Since $T^*_k\to \widetilde{T}$ and $\widetilde{m}\in C([0,T];L^2(\Omega))$, the first term on the right-hand side of the inequality converges to $0$ as $k \to \infty$. Similarly, the strong convergence of $m_k$ to $\widetilde{m}$ in $C(0,T;L^2(\Omega))$ ensures that the second term also converges to $0$. Furthermore, since $T^*_k$ is the optimal time for $u_k$, so as $k\to \infty$, we have $\|\widetilde{m}(\widetilde{T})-m_{\Omega}\|_{L^2(\Omega)}\leq \delta $. This guarantees that $\mathcal{J}(\widetilde{u})$ is well-defined. 
	
	Now, using the strong convergence of $T_k^*$ from equation \eqref{EOC-2}, we have $\displaystyle T^*(\widetilde{u})^2 = \lim_{k\to\infty} T^*(u_k)^2$. Also, from the weak sequential lower semi-continuity of $\{u_k\}$, we have
	$$\|\, \widetilde{u}\, \|_{W^{1,2}(0,T;H^1(\Omega),H^1(\Omega)^*)} \leq \liminf_{k \to \infty} \  \| u_k\|_{W^{1,2}(0,T;H^1(\Omega),H^1(\Omega)^*)}.$$
	Combining these two estimates and substituting \eqref{EOC-1}, we find that 
	\begin{equation}\label{EOC-3}
		\mathcal{J}(\widetilde{u})\leq \liminf_{k \to \infty} \mathcal{J}(u_k)=\alpha.
	\end{equation}
	Moreover, since $T^*(\widetilde{u})$ is well-defined and $\alpha$ represents the infimum of the functional $\mathcal{J}$ over $\mathcal{U}_{ad}$, we have
	\begin{equation}\label{EOC-4}
	\alpha =\inf_{u\in \mathcal{U}_{ad}} \mathcal{J}(u) \leq \mathcal{J}(\widetilde{u}).
	\end{equation}
	Therefore, combining \eqref{EOC-3} and \eqref{EOC-4}, we get $\mathcal{J}(\widetilde{u})=\inf_{u\in \mathcal{U}_{ad}} \mathcal{J}(u)$. The proof is thus established.
\end{proof}


\section{Linearized System and First-Order Optimality Condition}

\subsection{Linearized System}\label{SEC-LS}
Let $\widetilde{m}\in \mathcal{M}$ be the regular solution of \eqref{NLP} associated with the control  $\widetilde{u}$. Before going to the first-order and second-order optimality conditions, we will demonstrate the solvability of the following linearized system:
\begin{equation}\label{CLE}
	(L-LLG)\begin{cases}
		\begin{array}{l}
			\mathcal{L}_{\widetilde{u}}z=f \ \ \ \ \text{in}\ \Omega_T,\\
			\frac{\partial z}{\partial \eta}=0 \ \ \ \ \ \  \text{in}\ \partial \Omega_T, \ \ \ z(x,0)=z_0\ \ \text{in}\ \Omega,
		\end{array}
	\end{cases}	
\end{equation}
where the operator $\mathcal{L}_{\widetilde{u}}$ is defined as
\begin{equation}\label{CLO}
	\mathcal{L}_{\widetilde{u}}z:= z_t-\Delta z -z\times \Delta \widetilde{m} -\widetilde{m}\times \Delta z - z \times \widetilde{u}+2\ (\widetilde{m}\cdot z)\ \widetilde{m}+ \left(1+|\widetilde{m}|^2\right)z.	
\end{equation}
\begin{Lem}\label{L-SLS}
	For any function $f$ in $L^2(0,T;H^1(\Omega))$ and $z_0\in H^2(\Omega)$, there exists a unique regular solution $z \in  W^{1,2}(0,T;H^3(\Omega);H^1(\Omega))$ to the linearized system \eqref{CLE}. Moreover, the following estimate holds:
	\begin{align}\label{LSSE}
		\mathcal{B}(z_{\infty},z_2,{z_t}_2) \leq& \left(\|z_0\|^2_{L^2(\Omega)}+\|\Delta z_0\|^2_{L^2(\Omega)} +  \|f\|^2_{L^2(0,T;H^1(\Omega))}\right) \times \ \exp\bigg\{C\   \Big(1+\|\widetilde{m}\|^2_{L^\infty(0,T;H^2(\Omega))}\nonumber\\
		&+\|\widetilde{m}\|^2_{L^\infty(0,T;H^1(\Omega))}\|\widetilde{m}\|^2_{L^2(0,T;H^2(\Omega))}+\|\widetilde{m}\|^2_{L^2(0,T;H^3(\Omega))}+ \|\widetilde{u}\|^2_{L^2(0,T;H^1(\Omega))}\Big) \bigg\},
	\end{align}
	where $\mathcal{B}(z_{\infty},z_2,{z_t}_2):=\|z\|^2_{L^{\infty}(0,T;H^2(\Omega))}+ \|z\|^2_{L^2(0,T;H^3(\Omega))}+\|z_t\|^2_{L^2(0,T;H^1(\Omega))}$.
\end{Lem}
\noindent For the proof of this lemma, one may refer to Lemma 4.1 of \cite{SPSK3}.

Let $\varphi=(\varphi_1,\varphi_2,\varphi_3,\varphi_4,\varphi_5,\varphi_6,\varphi_7,\varphi_8)$, $\psi=(\psi_1,\psi_2)$, $h$, $\vartheta=(\vartheta_1,\vartheta_2,\vartheta_3,\vartheta_4)$ and $\upsilon=(\upsilon_1,\upsilon_2)$ are functions of $(x,t)$  and $\rho$ is any positive constant. 
Consider another auxiliary system given by 
\begin{equation}\label{ER-SF}
	\begin{cases}
		\begin{array}{l}
			\zeta_t-\Delta \zeta =\mathscr{F}_1(\rho,\varphi)  + \mathscr{F}_2(\rho,\psi,\zeta) + \mathscr{F}_3(\rho,h,\zeta) + \mathscr{F}_4(\rho,\vartheta,\zeta) + \mathscr{F}_5(\rho,\upsilon,\zeta)+ \mathscr{F}_6(\rho,\zeta), \\

			\frac{\partial \zeta}{\partial \eta}=0 \ \ \ \ \  \text{in}\ \partial \Omega_T, \ \ \ \ \ \ \  \ \ \zeta(x,0)=0\ \ \text{in}\ \Omega,
		\end{array}
	\end{cases}	
\end{equation}
where
\begin{align*}
\mathscr{F}_1(\rho,\varphi)&=\rho \ \varphi_1 + \rho \ \varphi_2 \times \Delta \varphi_3  + \rho \ \varphi_4 \times \varphi_5 +\rho \ (\varphi_6\cdot \varphi_7)\,\varphi_8,  \\
\mathscr{F}_2(\rho,\psi,\zeta)&= \rho\ \psi_1 \times \Delta \zeta + \rho \ \zeta \times \Delta \psi_2 + \rho \ \zeta \times \Delta \zeta,\\
\mathscr{F}_3(\rho,h,\zeta) &= \rho \ \zeta \times h ,\\
\mathscr{F}_4(\rho,\vartheta,\zeta) &= \rho \ (\zeta \cdot \vartheta_1) \,\vartheta_2 + \rho\ (\vartheta_3 \cdot \vartheta_4)\,\zeta   ,\\
\mathscr{F}_5(\rho,\upsilon,\zeta) &= \rho \ |\zeta|^2\, \upsilon_1 +\rho \ (\zeta \cdot \upsilon_2) \, \zeta , \\
\mathscr{F}_6(\rho,\zeta) &= \rho \ |\zeta|^2\,\zeta.   
\end{align*}

The following lemma plays a crucial role in the proof of first- and second-order optimality conditions.

	\begin{Lem}\label{L-SER}
		Suppose the coefficient functions $\varphi_1,\varphi_5,h \in L^2(0,T;H^1(\Omega))$, $\upsilon_1,\upsilon_2\in L^4(0,T;H^1(\Omega))$, $\varphi_6,\vartheta_1,\vartheta_3\in L^{\infty}(0,T;H^1(\Omega))$,  $\varphi_8,\vartheta_2,\vartheta_4\in L^2(0,T;H^2(\Omega))$, $\varphi_2,\varphi_4,\varphi_7  \in L^{\infty}(0,T;H^2(\Omega))$, $\varphi_3,\psi_2 \in L^2(0,T;H^3(\Omega))$, and $\psi_1 \in L^{\infty}(0,T;H^2(\Omega))\cap L^2(0,T;H^3(\Omega)) $.
	For any $0<\sigma<1$, there exists a constant $\rho_0$ such that system \eqref{ER-SF} has a unique regular solution $\zeta\in \mathcal{M}$ corresponding to every $\rho\in [0,\rho_0]$, and the solution satisfies the following estimate: 
	\begin{equation}\label{SE-ZETA}
		\|\zeta\|^2_{L^{\infty}(0,T;H^2(\Omega))} + \|\zeta\|^2_{L^2(0,T;H^3(\Omega))} + \|\zeta_t\|^2_{L^2(0,T;H^1(\Omega))} \leq \mathscr{C}_{mod}(\Omega,T,\rho, \sigma,\varphi,\psi,h,\vartheta,\upsilon),
	\end{equation}

	where
	\begin{align*}
	&\mathscr{C}_{mod}(\Omega,T,\rho, \sigma,\varphi,\psi,h,\vartheta,\upsilon) := \frac{\rho^2}{\sigma^2}\, \Big( \sigma + \|\varphi_1\|^2_{L^2(0,T;H^1(\Omega))} +\|\varphi_2\|^2_{L^\infty(0,T;H^2(\Omega))}\|\varphi_3\|^2_{L^2(0,T;H^3(\Omega))}\\
	&+\|\varphi_4\|^2_{L^{\infty}(0,T;H^2(\Omega))} \ \|\varphi_5\|^2_{L^2(0,T;H^1(\Omega))} +\|\varphi_6\|^2_{L^\infty(0,T;H^1(\Omega))}\ \|\varphi_7\|^2_{L^\infty(0,T;H^2(\Omega))}\ \|\varphi_8\|^2_{L^2(0,T;H^2(\Omega))}\Big)^3\\
	&\times \exp\left\{C(\Omega) \ \Big ( T+  \rho^2\, \Big(\|\psi_1\|^2_{L^\infty(0,T;H^2(\Omega))}+ \|\psi_1\|^2_{L^2(0,T;H^3(\Omega))}\!+\!\|\psi_2\|^2_{L^2(0,T;H^3(\Omega))} \!+\!\|\vartheta_1\|^2_{L^\infty(0,T;H^1(\Omega))}\|\vartheta_2\|^2_{L^2(0,T;H^2(\Omega))}\Big.\right. \nonumber\\
	&\left.\left.	+ \|\vartheta_3\|^2_{L^\infty(0,T;H^1(\Omega))}\|\vartheta_4\|^2_{L^2(0,T;H^2(\Omega))}+\|h\|^2_{L^2(0,T;H^1(\Omega))}+\ \|\upsilon_1\|^4_{L^4(0,T;H^1(\Omega))} +\|\upsilon_2\|^4_{L^4(0,T;H^1(\Omega))}  \right)\Big)\right\},
	\end{align*}
 Moreover, the solution $\zeta $ converges to $0$ in $\mathcal{M}$ as $\rho \to 0^+$.
\end{Lem} 
\begin{proof}[Proof of Lemma \ref{L-SER}]
	
	Let $w_j$ represent the $j^{th}$ eigenfunction associated with the eigenvalue $\lambda_j$ of the operator $-\Delta + I$ under the Neumann boundary condition, where $I$ is the identity operator. Specifically, this means $(-\Delta +I) w_j=\lambda_j w_j$ with $\frac{\partial w_j}{\partial \eta}\big|_{\partial \Omega}=0$, and the set $\{w_j\}^\infty_{j=1}$ forms an orthonormal basis for $L^2(\Omega)$ space. Define $W_n:= span \{w_1,w_2,...,w_n\}$ as a finite-dimensional subspace of $L^2(\Omega)$ and let $\mathbb{P}_n:L^2(\Omega)\to W_n$ denote the orthogonal projection. Now, we consider the Galerkin approximation of the system:
		\begin{equation}\label{GA-SS}
		\begin{cases}
			(\zeta_n)_t- \Delta \zeta_n=\mathbb{P}_n \left[\mathscr{F}_1(\rho,\varphi)  + \mathscr{F}_2(\rho,\psi,\zeta_n) + \mathscr{F}_3(\rho,h,\zeta_n) + \mathscr{F}_4(\rho,\vartheta,\zeta_n) + \mathscr{F}_5(\rho,\upsilon,\zeta_n)+ \mathscr{F}_6(\rho,\zeta_n) \right],\\
			\zeta_n(0)=0,
		\end{cases}
	\end{equation}
	where $\zeta_n(t)=\sum_{k=1}^{n}a_{kn}(t)\ w_k\in W_n$. Note that for each fixed $t,$ the term $a_{kn}(t)$ is a vector in $\mathbb{R}^3$. The product in $\zeta_n(t)$ is a standard scalar-vector product, that is, if we represent the canonical basis of $\mathbb{R}^3$ as $\{e_j\}^3_{j=1}$, then the term $\zeta_n(t)$ is given by $$\zeta_n(x,t)
	= \sum_{k=1}^n \sum_{j=1}^3 a_{kn}^{(j)}(t)\,w_k(x)\,e_j.$$
		
 The equation \eqref{GA-SS} can be reduced to a system of ordinary differential equations (ODEs):
	\begin{equation}\label{ODE}  
		\frac{d}{dt}a_{kn}(t)= F_k(t,a_{n}), \ \ a_{k}(0)=0, \ \ \ \ \ k=1,2,...,n,
	\end{equation}
	where $a_{n}=(a_{1n}(t),a_{2n}(t),...,a_{nn}(t))^T$ and 
	\begin{align*}
		F_k(t,a_{n}) = -(\lambda_k&-1)\ a_{kn}(t) +  \rho \int_{\Omega}\big(\varphi_1 +  \varphi_2 \times \Delta \varphi_3  + \varphi_4 \times \varphi_5 + (\varphi_6\cdot \varphi_7)\,\varphi_8\big) \, w_k\ dx\\
		&-\rho \sum_{r=1}^n(\lambda_r-1)  \int_{\Omega} \big(\psi_1 \times a_{rn}(t) \big)\ w_r\,w_k \ dx +\rho \sum_{r=1}^n  \int_{\Omega} \big(a_{rn}(t) \times \Delta \psi_2 \big)\ w_r\,w_k \ dx\\
		&- \rho\ \sum_{r,s=1}^{n} (\lambda_s-1) \big(a_{rn}(t)\times a_{sn}(t)\big)\int_\Omega w_r\,  w_s\,  w_k\ dx +\rho\ \sum_{r=1}^n \int_{\Omega} \big(a_{rn}(t)\times h\big)\ w_r\ w_k\ dx\\ 
		&+\rho\ \sum_{r=1}^{n} \int_{\Omega} \big(a_{rn}(t) \cdot \vartheta_1 \big)\,\vartheta_2  \ w_r \, w_k\ dx+\rho\  \sum_{r=1}^{n} \int_{\Omega} \big(\vartheta_3 \cdot \vartheta_4 \big)\,a_{rn}(t)  \ w_r \, w_k\ dx\\
		& +\rho \sum_{r,s=1}^{n} \int_\Omega \big(a_{rn}(t)\cdot a_{sn}(t)\big)\,\upsilon_1\, w_r\, w_s\, w_k\ dx+\rho \sum_{r,s=1}^{n} \int_\Omega \big(a_{rn}(t)\cdot \upsilon_2 \big)\,a_{sn}(t)\, w_r\, w_s\, w_k\ dx \\
		& + \rho \sum_{r,s,v=1}^{n} \Big(\big(a_{rn}(t) \cdot a_{sn}(t)\big)\ a_{vn}(t)\Big)\int_\Omega  w_r \ w_s \ w_v\ w_k\ dx.	
	\end{align*}	
	
	Let us assume that $\varphi_i, \psi_i, h, \vartheta_i$ and $\upsilon_i$ all belong to $C([0,T];L^2(\Omega))$, with the additional regularity assumption $\varphi_3,\psi_2\in C([0,T];H^2(\Omega))$. Then the function $F_k(t,a_n)$ is continuous on $[0,T]\times \mathbb{R}^n$. Therefore, by the existence theory for ODEs (see \cite{PH}), there exists a solution  $a_n \in C^1([0,t_m);\mathbb{R}^n)$ to system \eqref{GA-SS}, where $t_m\in (0,T]$ is the maximal existence time. We will then establish, using an \emph{a priori} estimate, that this maximal time $t_m$ is indeed equal to $T$. Finally, using the density argument, we can extend the existence result to hold for every function   $\varphi_i, \psi_i, h, \vartheta_i$ and $\upsilon_i$ in their respective functional spaces as specified in the statement of the lemma.
	
	Now, taking the $L^2(\Omega)$ inner product of \eqref{GA-SS} with $\zeta_n$, we have
	\begin{align}\label{ES-E1}
		\frac{1}{2} \frac{d}{dt} &\|\zeta_n(t)\|^2_{L^2(\Omega)} + \|\nabla \zeta_n(t)\|^2_{L^2(\Omega)}=\sum_{k=1}^6 \int_{\Omega} \mathscr{F}_k\cdot \zeta_n(t)\ dx.
	\end{align} 
	Next, we will individually estimate each term on the right-hand side. For the first term applying H\"older's inequality, followed by the embedding $H^1(\Omega) \hookrightarrow L^p(\Omega)$ for $p\in [2,6]$, and the norm estimate \eqref{ES2}, we obtain
	\begin{align*}
		&\left | \int_{\Omega} \mathscr{F}_1(\rho,\varphi)\cdot \zeta_n(t)\ dx\ \right |\leq \rho\ \Big( \|\varphi_1(t)\|_{L^1(\Omega)}+\|\varphi_2(t)\|_{L^2(\Omega)}\,\|\Delta \varphi_3(t)\|_{L^2(\Omega)}&\\
		&\hspace{1.5cm}+\|\varphi_4(t)\|_{L^2(\Omega)}\,\|\varphi_5(t)\|_{L^2(\Omega)} +\|\varphi_6(t)\|_{L^3(\Omega)}\,\|\varphi_7(t)\|_{L^3(\Omega)}\,\|\varphi_8(t)\|_{L^3(\Omega)}\Big)\, \|\zeta_n(t)\|_{L^{\infty}(\Omega)}\\
		&\hspace{1cm}\leq \rho^2\ C_1(\Omega)\ \Big(   \|\varphi_1(t)\|^2_{L^1(\Omega)}+\|\varphi_2(t)\|^2_{L^2(\Omega)}\,\|\varphi_3(t)\|^2_{H^2(\Omega)}+\|\varphi_4(t)\|^2_{L^2(\Omega)}\,\|\varphi_5(t)\|^2_{L^2(\Omega)}&\\
		&\hspace{1.5cm} +\|\varphi_6(t)\|^2_{H^1(\Omega)}\,\|\varphi_7(t)\|^2_{H^1(\Omega)}\,\|\varphi_8(t)\|^2_{H^1(\Omega)}\Big)+ \|\zeta_n(t)\|^2_{H^2(\Omega)}.
	\end{align*}
	Similarly, applying the cross product property $\zeta_n\cdot (\zeta_n \times h)=0$ and $\zeta_n\cdot (\zeta_n \times \Delta \zeta_n)=0$ for the second and third terms on the right-hand side of \eqref{ES-E1}, we left with
	\begin{align*}
		&\left | \int_{\Omega} \mathscr{F}_2(\rho,\psi,\zeta_n)\cdot \zeta_n(t)\ dx+\int_{\Omega} \mathscr{F}_3(\rho,h,\zeta_n)\cdot \zeta_n(t)\ dx \ \right |\leq \rho\ \|\psi_1(t)\|_{L^4(\Omega)}\,\|\Delta \zeta_n(t)\|_{L^2(\Omega)}\, \|\zeta_n(t)\|_{L^4(\Omega)}&\\
		&\hspace{1.5cm}\leq \rho^2\ C_2(\Omega)\ \|\psi_1(t)\|^2_{H^1(\Omega)}\, \|\zeta_n(t)\|^2_{H^1(\Omega)}+\|\zeta_n(t)\|^2_{H^2(\Omega)}.			
	\end{align*}
Finally, for the last three terms applying H\"older's inequality and the embedding $H^1(\Omega) \hookrightarrow L^4(\Omega)$, we obtain
\begin{align*}
&\left | \int_{\Omega} \mathscr{F}_4(\rho,\vartheta,\zeta_n)\cdot \zeta_n(t)\ dx \right |\leq \rho\ \left( \|\vartheta_1(t)\|_{L^4(\Omega)}\ \|\vartheta_2(t)\|_{L^4(\Omega)}+\|\vartheta_3(t)\|_{L^4(\Omega)}\ \|\vartheta_4(t)\|_{L^4(\Omega)}\right) \|\zeta_n(t)\|^2_{L^4(\Omega)}&\\
&\hspace{1.5cm}\leq \rho^2\,C_3(\Omega)\, \left( \|\vartheta_1(t)\|^2_{H^1(\Omega)}\ \|\vartheta_2(t)\|^2_{H^1(\Omega)}+\|\vartheta_3(t)\|^2_{H^1(\Omega)}\ \|\vartheta_4(t)\|^2_{H^1(\Omega)}\right) \|\zeta_n(t)\|^2_{H^1(\Omega)}+\|\zeta_n(t)\|^2_{H^1(\Omega)},	
\end{align*}
and
\begin{align*}
&\left | \int_{\Omega} \mathscr{F}_5(\rho,\upsilon,\zeta_n)\cdot \zeta_n(t)\ dx + \int_{\Omega} \mathscr{F}_6(\rho,\zeta_n)\cdot \zeta_n(t)\ dx\right|&\\
&\hspace{1.5cm} \leq \rho\ \left( \|\upsilon_1(t)\|_{L^4(\Omega)}+\|\upsilon_2(t)\|_{L^4(\Omega)} +\|\zeta_n(t)\|_{L^4(\Omega)}\right) \|\zeta_n(t)\|^3_{L^4(\Omega)}&\\
&\hspace{1.5cm}\leq \rho^2\, C_4(\Omega)\, \left( \|\upsilon_1(t)\|^4_{H^1(\Omega)}+ \|\upsilon_2(t)\|^4_{H^1(\Omega)}\right) \|\zeta_n(t)\|^2_{H^1(\Omega)}+\|\zeta_n(t)\|^2_{H^1(\Omega)}+\rho^2\,\|\zeta_n(t)\|^6_{H^1(\Omega)}.		
\end{align*}
Therefore, substituting all these estimates in equation \eqref{ES-E1}, we find
\begin{align}\label{ES-E2}
\frac{1}{2} \frac{d}{dt} &\|\zeta_n(t)\|^2_{L^2(\Omega)} + \|\nabla \zeta_n(t)\|^2_{L^2(\Omega)}\leq  \rho^2\ C(\Omega)\ \Big(   \|\varphi_1(t)\|^2_{L^1(\Omega)}+\|\varphi_2(t)\|^2_{L^2(\Omega)}\,\|\varphi_3(t)\|^2_{H^2(\Omega)}&\nonumber\\
&\hspace{0.7cm}+\|\varphi_4(t)\|^2_{L^2(\Omega)}\,\|\varphi_5(t)\|^2_{L^2(\Omega)} +\|\varphi_6(t)\|^2_{H^1(\Omega)}\,\|\varphi_7(t)\|^2_{H^1(\Omega)}\,\|\varphi_8(t)\|^2_{H^1(\Omega)}\Big)\nonumber\\
&\hspace{0.7cm} +\rho^2\,C(\Omega)\, \left(\|\psi_1(t)\|^2_{H^1(\Omega)}+ \|\vartheta_1(t)\|^2_{H^1(\Omega)}\ \|\vartheta_2(t)\|^2_{H^1(\Omega)}+\|\vartheta_3(t)\|^2_{H^1(\Omega)}\ \|\vartheta_4(t)\|^2_{H^1(\Omega)}\right. \nonumber\\
&\hspace{0.7cm} \left. +\,\|\upsilon_1(t)\|^4_{H^1(\Omega)}+ \|\upsilon_2(t)\|^4_{H^1(\Omega)} \right) \|\zeta_n(t)\|^2_{H^1(\Omega)}+C(\Omega)\ \|\zeta_n(t)\|^2_{H^2(\Omega)}+\rho^2\,C(\Omega)\|\zeta_n(t)\|^6_{H^1(\Omega)},
\end{align} 
where $\displaystyle C(\Omega)= \max_{1\leq i \leq 4} \Big\{C_i(\Omega),4\Big\}$.

Furthermore, by applying the $L^2(\Omega)$ inner product of \eqref{GA-SS} with $\Delta^2 \zeta_n$, and integrating by parts over the spatial domain, we obtain
		\begin{align}\label{ES-E3}
			\frac{1}{2} \frac{d}{dt} &\|\Delta \zeta_n(t)\|^2_{L^2(\Omega)} + \|\nabla\Delta  \zeta_n(t)\|^2_{L^2(\Omega)} =-\sum_{k=1}^6 \int_{\Omega} \nabla \mathscr{F}_k\cdot \nabla \Delta \zeta_n(t)\ dx.
		\end{align} 
		Next, we proceed to estimate each term on the right-hand side of equation \eqref{ES-E3} individually. For the first term, by applying H\"older's inequality in combination with the Sobolev embeddings $H^1(\Omega) \hookrightarrow L^p(\Omega)$ for $p\in [2,6]$ and $H^2(\Omega) \hookrightarrow L^{\infty}(\Omega)$, we derive
		\begin{align*}
		&\left | \int_{\Omega} \nabla \mathscr{F}_1(\rho,\varphi)\cdot \nabla\Delta \zeta_n(t)\ dx\ \right | \leq \rho\ \Big( \|\nabla \varphi_1(t)\|_{L^2(\Omega)}+\|\nabla\varphi_2(t)\|_{L^4(\Omega)}\,\|\Delta \varphi_3(t)\|_{L^4(\Omega)}\\
		&\hspace{1.5cm}+\|\varphi_2(t)\|_{L^\infty(\Omega)}\,\|\nabla\Delta \varphi_3(t)\|_{L^2(\Omega)}+\|\nabla\varphi_4(t)\|_{L^4(\Omega)}\,\|\varphi_5(t)\|_{L^4(\Omega)}+\|\varphi_4(t)\|_{L^\infty(\Omega)}\,\|\nabla\varphi_5(t)\|_{L^2(\Omega)}&\\
		&\hspace{1.5cm} +\|\nabla\varphi_6(t)\|_{L^2(\Omega)}\,\|\varphi_7(t)\|_{L^\infty(\Omega)}\,\|\varphi_8(t)\|_{L^\infty(\Omega)} +\|\varphi_6(t)\|_{L^6(\Omega)}\,\|\nabla\varphi_7(t)\|_{L^6(\Omega)}\,\|\varphi_8(t)\|_{L^6(\Omega)}\\
		&\hspace{1.5cm} +\|\varphi_6(t)\|_{L^6(\Omega)}\,\|\varphi_7(t)\|_{L^6(\Omega)}\,\|\nabla\varphi_8(t)\|_{L^6(\Omega)}\Big)\, \|\nabla\Delta \zeta_n(t)\|_{L^2(\Omega)}\\
		&\hspace{0.7cm}	\leq \epsilon \,\|\nabla \Delta \zeta_n(t)\|^2_{L^2(\Omega)} +\rho^2\ \bar C_1(\Omega, \epsilon)\ \Big( \|\varphi_1(t)\|^2_{H^1(\Omega)}+\|\varphi_2(t)\|^2_{H^2(\Omega)}\,\| \varphi_3(t)\|^2_{H^3(\Omega)}+\|\varphi_4(t)\|^2_{H^2(\Omega)}\,\|\varphi_5(t)\|^2_{H^1(\Omega)}&\\
		&\hspace{1.5cm} +\|\varphi_6(t)\|^2_{H^1(\Omega)}\,\|\varphi_7(t)\|^2_{H^2(\Omega)}\,\|\varphi_8(t)\|^2_{H^2(\Omega)} \Big).
		\end{align*}
	Similarly, we estimate the second term on the right-hand side of equation \eqref{ES-E3} by following the same line of reasoning as for the first term. Additionally, we utilize cross product property $\big( \psi_1 \times \nabla \Delta \zeta_n \big)\cdot \nabla \Delta \zeta_n=0 $, along with the estimates provided in equations \eqref{ES3} and \eqref{ES9} to get
\begin{align*}
&\left | \int_{\Omega} \nabla \mathscr{F}_2(\rho,\psi,\zeta_n)\cdot \nabla\Delta \zeta_n(t)\ dx\ \right | \leq \rho\ \Big( \|\nabla \psi_1(t)\|_{L^\infty(\Omega)}\,\|\Delta \zeta_n(t)\|_{L^2(\Omega)} +\|\nabla \zeta_n(t)\|_{L^4(\Omega)}\, \|\Delta \psi_2(t)\|_{L^4(\Omega)} &\\
&\hspace{1.5cm}+ \| \zeta_n(t)\|_{L^\infty(\Omega)}\, \|\nabla \Delta \psi_2(t)\|_{L^2(\Omega)} +\|\nabla \zeta_n(t)\|_{L^6(\Omega)}\, \|\Delta \zeta_n(t)\|_{L^3(\Omega)}\Big)\ \|\nabla \Delta \zeta_n(t)\|_{L^2(\Omega)}\\
&\hspace{0.7cm} \leq  \epsilon\,\|\nabla \Delta \zeta_n(t)\|^2_{L^2(\Omega)} +\rho^2\ \bar C_2(\Omega, \epsilon)\ \Big(\|\psi_1(t)\|^2_{H^3(\Omega)}+\|\psi_2(t)\|^2_{H^3(\Omega)}\Big)\ \|\zeta_n(t)\|^2_{H^2(\Omega)}+\rho^4\,\bar C_2(\Omega, \epsilon) \,\|\Delta \zeta_n(t)\|^6_{L^2(\Omega)}.
\end{align*}			
		Continuing in the same way for the third term, we find
		\begin{align*}
			&\left | \int_{\Omega} \nabla \mathscr{F}_3(\rho,h,\zeta_n)\cdot \nabla\Delta \zeta_n(t)\ dx\ \right | \leq \rho\ \|\nabla \Delta \zeta_n(t)\|_{L^2(\Omega)}\Big( \|\nabla \zeta_n(t)\|_{L^4(\Omega)}\,\|h(t)\|_{L^4(\Omega)}&\\
			& \hspace{1.5cm}+\|\zeta_n(t)\|_{L^\infty(\Omega)}\, \|\nabla h(t)\|_{L^2(\Omega)} \Big)  \leq  \epsilon\,\|\nabla \Delta \zeta_n(t)\|^2_{L^2(\Omega)} +\rho^2\ \bar C_3(\Omega, \epsilon)\  \|h(t)\|^2_{H^1(\Omega)}\,\|\zeta_n(t)\|^2_{H^2(\Omega)}.
		\end{align*}	
		Finally, for the remaining terms on the right-hand side of equation \eqref{ES-E3}, we proceed as we did for the term $\big(\varphi_6 \cdot \varphi_7\big)\varphi_8$ and obtain  	
		\begin{align*}
			&\left | \int_{\Omega} \nabla \mathscr{F}_4(\rho,\vartheta,\zeta_n)\cdot \nabla\Delta \zeta_n(t)\ dx+\int_{\Omega} \nabla \mathscr{F}_5(\rho,\upsilon,\zeta_n)\cdot \nabla\Delta \zeta_n(t)\ dx \right | &\\
			&\hspace{0.7cm} \leq     \epsilon\, \|\nabla \Delta \zeta_n(t)\|^2_{L^2(\Omega)}+\rho^2\ \bar C_4(\Omega, \epsilon)\ \left(\|\vartheta_1(t)\|^2_{H^1(\Omega)}\,\|\vartheta_2(t)\|^2_{H^2(\Omega)}+\|\vartheta_3(t)\|^2_{H^1(\Omega)}\,\|\vartheta_4(t)\|^2_{H^2(\Omega)}\right)\, \|\zeta_n(t)\|^2_{H^2(\Omega)}\\
			&\hspace{1cm}+\rho^2\ \bar C_5(\Omega, \epsilon)\ \left(\|\upsilon_1(t)\|^4_{H^1(\Omega)}+\|\upsilon_2(t)\|^4_{H^1(\Omega)}\right)\, \|\zeta_n(t)\|^2_{H^2(\Omega)}+\frac{\rho^2}{2}\ \bar C_5(\Omega, \epsilon)\ \|\zeta_n(t)\|^6_{H^2(\Omega)},
		\end{align*}
		and
		\begin{align*}
			&\left | \int_{\Omega} \nabla \mathscr{F}_6(\rho,\zeta_n)\cdot \nabla\Delta \zeta_n(t)\ dx\ \right | \leq 3\,\rho\ \| \zeta_n(t)\|^2_{L^6(\Omega)} \|\nabla \zeta_n(t)\|_{L^6(\Omega)}\|\nabla \Delta \zeta_n(t)\|_{L^2(\Omega)}&\\
			&\hspace{1cm}\leq \epsilon\,\|\nabla \Delta \zeta_n(t)\|^2_{L^2(\Omega)}+\frac{\rho^2}{2}\ \bar C_6(\Omega, \epsilon)\,\| \zeta_n(t)\|^4_{H^1(\Omega)} \,\|\zeta_n(t)\|^2_{H^2(\Omega)}.
		\end{align*}	
		
Next, we substitute all the aforementioned estimates into equation \eqref{ES-E3}, select $\epsilon=\frac{1}{10}$, and combine the resulting expression with equation \eqref{ES-E2}, we get 
\begin{flalign}\label{FRES}
	&\frac{d}{dt} \left(\|\zeta_n(t)\|^2_{L^2(\Omega)} + \|\Delta \zeta_n(t)\|^2_{L^2(\Omega)} \right) +  \|\zeta_n(t)\|^2_{H^3(\Omega)} \left( 1- C^*(\Omega) \ \rho^2\ \| \zeta_n(t)\|^4_{H^2(\Omega)}-C^*(\Omega)\ \rho^4\ \|\Delta \zeta_n(t)\|^4_{L^2(\Omega)}\right) &\nonumber\\
	&\hspace{0.7cm} \leq \rho^2\ C^*(\Omega)\ \Big( \|\varphi_1(t)\|^2_{H^1(\Omega)}+\|\varphi_2(t)\|^2_{H^2(\Omega)}\,\| \varphi_3(t)\|^2_{H^3(\Omega)}+\|\varphi_4(t)\|^2_{H^2(\Omega)}\,\|\varphi_5(t)\|^2_{H^1(\Omega)}\nonumber\\
	&\hspace{1.5cm} +\|\varphi_6(t)\|^2_{H^1(\Omega)}\,\|\varphi_7(t)\|^2_{H^2(\Omega)}\,\|\varphi_8(t)\|^2_{H^2(\Omega)} \Big)+\rho^2\ C^*(\Omega)\, \Big( \|\psi_1(t)\|^2_{H^3(\Omega)}+\|\psi_2(t)\|^2_{H^3(\Omega)}\nonumber\\
	&\hspace{1.5cm} + \|h(t)\|^2_{H^1(\Omega)}+\|\vartheta_1(t)\|^2_{H^1(\Omega)}\,\|\vartheta_2(t)\|^2_{H^2(\Omega)}+\|\vartheta_3(t)\|^2_{H^1(\Omega)}\,\|\vartheta_4(t)\|^2_{H^2(\Omega)}\nonumber\\
	&\hspace{1.5cm}+\|\upsilon_1(t)\|^4_{H^1(\Omega)}+\|\upsilon_2(t)\|^4_{H^1(\Omega)}  \Big)\ \|\zeta_n(t)\|^2_{H^2(\Omega)} +  C^*(\Omega)\ \|\zeta_n(t)\|^2_{H^2(\Omega)},
\end{flalign}
where $\displaystyle C^*(\Omega)= C(\Omega)+\max_{1\leq i \leq 6} \bar C_i(\Omega)$.\\
Now, before proceeding further let us define
	\begin{align*}
	&\mathscr{C}(\Omega,T,\rho, \sigma, \varphi,\psi,h,\vartheta,\upsilon) :=\rho^2\ C^*(\Omega) \ \Big( \sigma + \|\varphi_1\|^2_{L^2(0,T;H^1(\Omega))} +\|\varphi_2\|^2_{L^\infty(0,T;H^2(\Omega))}\|\varphi_3\|^2_{L^2(0,T;H^3(\Omega))}\\
	&\hspace{0.4cm}+\|\varphi_4\|^2_{L^\infty(0,T;H^2(\Omega))} \|\varphi_5\|^2_{L^2(0,T;H^1(\Omega))}+\|\varphi_6\|^2_{L^\infty(0,T;H^1(\Omega))}\|\varphi_7\|^2_{L^\infty(0,T;H^2(\Omega))}\|\varphi_8\|^2_{L^2(0,T;H^2(\Omega))}\Big)\\
	&\hspace{0.7cm}\times \exp\left\{C^*(\Omega)\ \Big(T + \rho^2\, \Big( \|\psi_1\|^2_{L^2(0,T;H^3(\Omega))}+\|\psi_2\|^2_{L^2(0,T;H^3(\Omega))} +\|\vartheta_1\|^2_{L^\infty(0,T;H^1(\Omega))}\|\vartheta_2\|^2_{L^2(0,T;H^2(\Omega))}\Big.\right. \nonumber\\
	&\hspace{1cm}\left.\left.	+ \|\vartheta_3\|^2_{L^\infty(0,T;H^1(\Omega))}\|\vartheta_4\|^2_{L^2(0,T;H^2(\Omega))}+\|h\|^2_{L^2(0,T;H^1(\Omega))}+\ \|\upsilon_1\|^4_{L^4(0,T;H^1(\Omega))} +\|\upsilon_2\|^4_{L^4(0,T;H^1(\Omega))}  \right)\Big)\right\}.
	\end{align*}

Fix any $0< \sigma<1,$ choose $0<\rho_0<1$ such that $1-2\,C^*(\Omega)\  \rho_0^2\, \mathscr{C}(\Omega,T, \rho_0, \sigma, \varphi,\psi,h,\vartheta,\upsilon)^2 \geq \frac{1}{2}$.  
Since $\mathscr{C}$ is a monotonically increasing function corresponding to the variable $\rho$, therefore for any fixed $\rho$ in $[0,\rho_0],$ the following inequality holds.
\begin{equation}\label{MAIN-C}
	1- 2\,C^*(\Omega)\ \rho^2\, \mathscr{C}(\Omega,T, \rho, \sigma, \varphi,\psi,h,\vartheta,\upsilon)^2 \geq \frac{1}{2}.	
\end{equation}

\noindent We want to show that for any such fixed values of $\sigma$ and $\rho$, the following inequality holds:
		\begin{equation}\label{CLM}
			\|\zeta_n(t)\|^2_{L^2(\Omega)} + \|\Delta \zeta_n(t)\|^2_{L^2(\Omega)} \leq \mathscr{C}(\Omega,T,\rho, \sigma,\varphi,\psi,h,\vartheta,\upsilon), \ \ \ \ \text{for all } t\in[0,T].
		\end{equation}
		Clearly, for $t=0$ this inequality holds as
		\begin{equation*}
		\|\zeta_n(0)\|^2_{L^2(\Omega)} + \|\Delta \zeta_n(0)\|^2_{L^2(\Omega)}= 0 <\mathscr{C}(\Omega,T,\rho, \sigma,\varphi,\psi,h,\vartheta,\upsilon).
		\end{equation*}
		Since $a_n\in C^1([0,t_m);\mathbb{R}^n)$ in \eqref{ODE}, we have $\zeta_n \in C([0,t_m);H^2(\Omega))$. As a consequence, there exists a time $\tau>0$ such that  
		\begin{align}\label{CL-T1}
			&\|\zeta_n(t)\|^2_{L^2(\Omega)} + \|\Delta \zeta_n(t)\|^2_{L^2(\Omega)} < \mathscr{C}(\Omega,T,\rho, \sigma,\varphi,\psi,h,\vartheta,\upsilon)  \ \ \ \text{for all}\ t\in[0,\tau].
		\end{align}
		Let $T_{max}$ be the maximum of  all $\tau$ such that \eqref{CL-T1} holds. Then by the definition of $T_{max}$, we have
		\begin{align}\label{CL-T2}
			&\|\zeta_n(t)\|^2_{L^2(\Omega)} + \|\Delta \zeta_n(t)\|^2_{L^2(\Omega)} \leq  \mathscr{C}(\Omega,T,\rho, \sigma, \varphi ,\psi,h,\vartheta,\upsilon)\ \ \ \ \text{for all}\ t\in[0,T_{max}].
		\end{align}
		If $T_{max}=T$, then we are done. If not, let us assume that $T_{max}<T$. 
		
	To estimate the second term on the left-hand side of \eqref{FRES}, we use the equivalent norm estimate for $\|\zeta_n(t)\|^2_{H^2(\Omega)}$ from Lemma \ref{EN} and substitute \eqref{CL-T2}. From our selection of $\rho_0$ in inequality \eqref{MAIN-C}, we get the following estimate for 		
$\rho\in [0,\rho_0]$ and $t\in[0,T_{max}]$:
		\begin{align}\label{CL-T3}
			&\|\zeta_n(t)\|^2_{H^3(\Omega)}\left(1-C^*(\Omega)\ \rho^2 \ \|\zeta_n(t)\|^4_{H^2(\Omega)} - C^*(\Omega)\ \rho^4 \ \|\Delta \zeta_n(t)\|^4_{L^2(\Omega)}\right) &\nonumber\\ 
			&\hspace{2cm}\geq \|\zeta_n(t)\|^2_{H^3(\Omega)}  \bigg(1- 2\,C^*(\Omega) \, \rho^2\,\|\zeta_n(t)\|^4_{H^2(\Omega)}\bigg)&\nonumber\\ 
			&\hspace{2cm}\geq \|\zeta_n(t)\|^2_{H^3(\Omega)}  \bigg(1- 2\,C^*(\Omega) \ \rho^2\, \mathscr{C}(\Omega,T,\rho, \sigma,\varphi,\psi,h,\vartheta,\upsilon)^2\bigg)\geq \frac{1}{2} \|\zeta_n(t)\|^2_{H^3(\Omega)}. 
		\end{align}
		Substituting this estimate in \eqref{FRES}, 
		and applying Gronwall's inequality, we obtain 
			\begin{align*}
			&\|\zeta_n(t)\|^2_{L^2(\Omega)} + \|\Delta \zeta_n(t)\|^2_{L^2(\Omega)} \leq \rho^2\  C^*(\Omega)\    \Big(  \|\varphi_1\|^2_{L^2(0,T;H^1(\Omega))} +\|\varphi_2\|^2_{L^\infty(0,T;H^2(\Omega))}\|\varphi_3\|^2_{L^2(0,T;H^3(\Omega))}\\
			&\hspace{0.4cm}+\|\varphi_4\|^2_{L^\infty(0,T;H^2(\Omega))} \|\varphi_5\|^2_{L^2(0,T;H^1(\Omega))}+\|\varphi_6\|^2_{L^\infty(0,T;H^1(\Omega))}\|\varphi_7\|^2_{L^\infty(0,T;H^2(\Omega))}\|\varphi_8\|^2_{L^2(0,T;H^2(\Omega))}\Big)\\
			&\hspace{0.7cm}\times \exp\left\{ C^*(\Omega)\  \Big( T+ \rho^2\, \Big( \|\psi_1\|^2_{L^2(0,T;H^3(\Omega))}+\|\psi_2\|^2_{L^2(0,T;H^3(\Omega))} +\|\vartheta_1\|^2_{L^\infty(0,T;H^1(\Omega))}\|\vartheta_2\|^2_{L^2(0,T;H^2(\Omega))}\Big.\right. \nonumber\\
			&\hspace{.71cm}\left.\left.	+ \|\vartheta_3\|^2_{L^\infty(0,T;H^1(\Omega))}\|\vartheta_4\|^2_{L^2(0,T;H^2(\Omega))}+\|h\|^2_{L^2(0,T;H^1(\Omega))}+\ \|\upsilon_1\|^4_{L^4(0,T;H^1(\Omega))} +\|\upsilon_2\|^4_{L^4(0,T;H^1(\Omega))}  \right)\Big)\right\}\\
			&< \mathscr{C}(\Omega,T,\rho,\sigma,\varphi,\psi,h,\vartheta,\upsilon).
		\end{align*}

		Now, by the continuity of $\|\zeta_n(t)\|^2_{H^2(\Omega)}$, there exists a time $T^\#>T_{max}$ such that \eqref{CLM} holds on $[0,T^\#]$. This contradicts the definition of $T_{max}$. Therefore, we must have $T_{max}=T$, which establishes our claim \eqref{CLM}, that is
		\begin{equation}\label{CL-T5}
			\|\zeta_n\|^2_{L^\infty(0,T;H^2(\Omega))} \leq \mathscr{C}(\Omega,T,\rho, \sigma,\varphi,\psi,h,\vartheta,\upsilon).
		\end{equation}
		Next, substituting \eqref{CL-T3} in \eqref{FRES}, taking integration over $[0,T]$ and using the uniform bound \eqref{CL-T5}, we derive
		\begin{equation}\label{CL-T6}
			\int_0^T \|\zeta_n(\tau)\|^2_{H^3(\Omega)} d\tau \leq \mathscr{C}(\Omega,T,\rho, \sigma,\varphi,\psi,h,\vartheta,\upsilon).
		\end{equation}
		Furthermore, by taking $L^2(0,T;H^1(\Omega))$ norm of $\big(\zeta_n\big)_t$ in equation  \eqref{GA-SS}, we derive
		\begin{align*}
			&\|\big(\zeta_n\big)_t\|^2_{L^2(0,T;H^1(\Omega))} \leq \|\zeta_n\|^2_{L^2(0,T;H^3(\Omega))} + \rho^2\ \Big( \|\varphi_1\|^2_{L^2(0,T;H^1(\Omega))} +\|\varphi_2\|^2_{L^\infty(0,T;H^2(\Omega))} \|\varphi_3\|^2_{L^2(0,T;H^3(\Omega))}\\
			&\hspace{0.5cm} +\|\varphi_4\|^2_{L^\infty(0,T;H^2(\Omega))}\ \|\varphi_5\|^2_{L^2(0,T;H^1(\Omega))}  +\|\varphi_6\|^2_{L^\infty(0,T;H^1(\Omega))} \|\varphi_7\|^2_{L^\infty(0,T;H^2(\Omega))} \|\varphi_8\|^2_{L^2(0,T;H^2(\Omega))} \Big)\\
			&\hspace{0.5cm} +\rho^2\ \|\zeta_n\|^2_{L^2(0,T;H^3(\Omega))}\ \|\psi_1\|^2_{L^\infty(0,T;H^2(\Omega))}+\rho^2\ \|\zeta_n\|^2_{L^\infty(0,T;H^2(\Omega))}\ \Big( \|\psi_2\|^2_{L^2(0,T;H^3(\Omega))} \nonumber\\
			&\hspace{0.5cm} +\|\zeta_n\|^2_{L^2(0,T;H^3(\Omega))} + \|h\|^2_{L^2(0,T;H^1(\Omega))} +\|\vartheta_1\|^2_{L^\infty(0,T;H^1(\Omega))}\|\vartheta_2\|^2_{L^2(0,T;H^2(\Omega))} \nonumber\\
			&\hspace{0.5cm}+ \|\vartheta_3\|^2_{L^\infty(0,T;H^1(\Omega))}\|\vartheta_4\|^2_{L^2(0,T;H^2(\Omega))}+ \|\upsilon_1\|^4_{L^2(0,T;H^1(\Omega))} +\|\upsilon_2\|^4_{L^2(0,T;H^1(\Omega))} \Big) +\rho^2\, \|\zeta_n\|^6_{L^\infty(0,T;H^2(\Omega))}.
		\end{align*}
		Substituting the bounds for $\{\zeta_n\}$ in $\mathcal{M}$ from estimate \eqref{CL-T5} and \eqref{CL-T6} in the above estimate, we find
		\begin{equation}\label{CL-T7}
			\int_0^T \|\big(\zeta_n\big)_t\|^2_{H^1(\Omega)}\leq \mathscr{C}_{mod}(\Omega,T,\rho, \sigma,\varphi,\psi,h,\vartheta,\upsilon),
		\end{equation}
		where 
		\begin{align*}
			&\mathscr{C}_{mod}(\Omega,T,\rho, \sigma,\varphi,\psi,h,\vartheta,\upsilon) := \frac{\rho^2}{\sigma^2}\, \Big( \sigma + \|\varphi_1\|^2_{L^2(0,T;H^1(\Omega))} +\|\varphi_2\|^2_{L^\infty(0,T;H^2(\Omega))}\|\varphi_3\|^2_{L^2(0,T;H^3(\Omega))}\\
			&\hspace{0.4cm}+\|\varphi_4\|^2_{L^\infty(0,T;H^2(\Omega))} \|\varphi_5\|^2_{L^2(0,T;H^1(\Omega))}+\|\varphi_6\|^2_{L^\infty(0,T;H^1(\Omega))}\|\varphi_7\|^2_{L^\infty(0,T;H^2(\Omega))}\|\varphi_8\|^2_{L^2(0,T;H^2(\Omega))}\Big)^3\\
			&\hspace{0.5cm}\times \exp\left\{ C(\Omega) \ \Big( T+\ \rho^2\, \Big(\|\psi_1\|^2_{L^\infty(0,T;H^2(\Omega))}+ \|\psi_1\|^2_{L^2(0,T;H^3(\Omega))}+\|\psi_2\|^2_{L^2(0,T;H^3(\Omega))} +\|\vartheta_1\|^2_{L^\infty(0,T;H^1(\Omega))}\|\vartheta_2\|^2_{L^2(0,T;H^2(\Omega))}\Big.\right. \nonumber\\
			&\hspace{0.6cm}\left.\left.	+ \|\vartheta_3\|^2_{L^\infty(0,T;H^1(\Omega))}\|\vartheta_4\|^2_{L^2(0,T;H^2(\Omega))}+\|h\|^2_{L^2(0,T;H^1(\Omega))}+\ \|\upsilon_1\|^4_{L^4(0,T;H^1(\Omega))} +\|\upsilon_2\|^4_{L^4(0,T;H^1(\Omega))}  \right) \Big)\right\},
		\end{align*}
		in which we have utilized  $(\sigma+\phi^* )\leq \frac{1}{\sigma} (\sigma +\phi^*)^2\leq \frac{1}{\sigma^2} (\sigma +\phi^*)^3$ for any $\phi^*\geq 0$, and    $\rho^2 < \rho$, $\frac{1}{\sigma}< \frac{1}{\sigma^2}$ for $\rho,\sigma <1$.
		
		Finally, using the bounds from \eqref{CL-T5}, \eqref{CL-T6} and \eqref{CL-T7}, along with Alaoglu's weak star compactness theorem and the reflexive weak compactness theorem (Theorem 4.18, \cite{JCR}), we obtain the following convergences:
		\begin{eqnarray*}\left\{\begin{array}{cccll}
				\zeta_n &\overset{w}{\rightharpoonup} & \zeta \  &\mbox{weakly in}& \ L^2(0,T;H^3(\Omega)),\\
				\zeta_n &\overset{w^\ast}{\rightharpoonup} & \zeta \ &\mbox{weak$^*$ in}& \ L^{\infty}(0,T;H^2(\Omega)),\\
				(\zeta_n)_t &\overset{w}{\rightharpoonup} & \zeta_t \ &\mbox{weakly in}& \ L^2(0,T;H^1(\Omega)),\vspace{0.1cm}\\
				\zeta_n &\to &\zeta \ & \mbox{strongly in}& \ L^2(0,T;H^2(\Omega))\cap C(0,T;H^1(\Omega))  \ \ \ \
				\ \mbox{as} \  \ n\to \infty. 
			\end{array}\right.	
		\end{eqnarray*}

\noindent Next, using all these convergences, we will show that for any $v\in L^2(0,T;L^2(\Omega))$, the following convergences holds:
	\begin{enumerate}[label=(\roman*)]
		\item $\displaystyle \int_{\Omega_T} \Delta \zeta_n \cdot v \ dx\ \ dt \to \int_{\Omega_T} \Delta \zeta\cdot v \ dx\ dt$,\vspace{0.1in}
		\item $\displaystyle \int_{\Omega_T}  \mathbb{P}_n\big(\psi_1 \times \Delta \zeta_n\big)\cdot v \ dx\ dt \to \int_{\Omega_T}  \big(\psi_1 \times \Delta \zeta\big)\cdot v \ dx\ dt$,\vspace{0.1in}
		\item $\displaystyle\int_{\Omega_T}  \mathbb{P}_n\big(\zeta_n \times h\big)\cdot v \ dx\ dt  \to \int_{\Omega_T}  \big(\zeta \times h\big)\cdot v \ dx\ dt$,\vspace{0.1in}
		\item $\displaystyle \int_{\Omega_T}  \mathbb{P}_n\big((\zeta_n\cdot \vartheta_1)\ \vartheta_2 \big)\cdot v \ dx\ dt \to \int_{\Omega_T}  \big((\zeta\cdot \vartheta_1)\ \vartheta_2\big)\cdot v \ dx\ dt$,\vspace{0.1in}
		\item $\displaystyle \int_{\Omega_T}  \mathbb{P}_n\big((\zeta_n\cdot \upsilon_2)\ \zeta_n\big)\cdot v \ dx\ dt \to \int_{\Omega_T}  \big((\zeta\cdot \upsilon_2)\ \zeta\big)\cdot v \ dx\ dt$,\vspace{0.1in}
		\item $\displaystyle \int_{\Omega_T}  \mathbb{P}_n\left(|\zeta_n|^2\zeta_n\right)\cdot v \ dx\ dt \to \int_{\Omega_T}  \left(|\zeta|^2\zeta \right)\cdot v \ dx\ dt$.
	\end{enumerate}
	
		First, let us take any fixed $v \in L^2(0,T;W_n)$, where $W_n$ is the span of the first $n$ eigen functions $w_1,w_2,\cdots,w_n$ as defined earlier. Since we already know that $\zeta_n \to \zeta$ strongly in $L^2(0,T;H^2(\Omega))$, therefore convergence (i) is straightforward. Now, for the convergence of (ii), taking the H\"older's inequality, and using the embedding $H^1(\Omega)\hookrightarrow L^4(\Omega)$, we find
		\begin{align*}
			\bigg|\int_{\Omega_T}  &\Big[\big(\psi_1\times  \Delta \zeta_n\big)-\big(\psi_1\times \Delta \zeta\big)\Big]\cdot v\ dx\ dt\bigg| 
			\leq  \int_0^T \|\psi_1(t)\|_{L^4(\Omega)} \  \|\Delta \big(\zeta_n-\zeta\big)(t)\|_{L^2(\Omega)} \  \|v(t)\|_{L^4(\Omega)}\ dt\\  
			&\leq C\ \|\psi_1\|_{L^\infty(0,T;H^1(\Omega))}\  \|\zeta_n-\zeta\|_{L^2(0,T;H^2(\Omega))}\  \|v\|_{L^2(0,T;H^1(\Omega))}  \ \to 0 \ \ \ \mbox{as} \ n \to \infty.
		\end{align*}
	For the (v) convergence, again appealing to H\"older's inequality followed by $H^1(\Omega) \hookrightarrow L^4(\Omega)$, we obtain
		\begin{align*}
			\bigg|\int_{\Omega_T}   &\Big[\big( \zeta_n\cdot \upsilon_2\big)\ \zeta_n -\big( \zeta\cdot \upsilon_2\big)\ \zeta\Big]\cdot v\ dx\ dt\bigg|\\
			&\leq  \int_0^T \|\zeta_n(t)-\zeta(t)\|_{L^4(\Omega)} \ \|\upsilon_2(t)\|_{L^4(\Omega)} \  \left( \|\zeta_n(t)\|_{L^4(\Omega)} \  +\|\zeta(t)\|_{L^4(\Omega)} \  \right)  \|v(t)\|_{L^4(\Omega)}\ dt\\  
			&\leq C\ \|\zeta_n-\zeta\|_{L^2(0,T;H^1(\Omega))}\  \|\upsilon_2\|_{L^\infty(0,T;H^1(\Omega))}\ \left(\|\zeta_n\|_{L^\infty(0,T;H^1(\Omega))}+\|\zeta\|_{L^\infty(0,T;H^1(\Omega))}\right)\   \|v\|_{L^2(0,T;H^1(\Omega))}\\
			&\ \ \ \ \ \ \ \to 0\ \  \ \mbox{as} \ n \to \infty.
		\end{align*}
		The rest of the convergences (iii), (iv), and (vi) that follow from similar arguments. Now, since $\{w_n\}$ is dense in $H^1(\Omega)$ space, therefore, these convergences hold for every $v\in L^2(0,T;H^1(\Omega))$.
		
		

	Finally, by utilizing these convergences, we confirm that $\zeta $ is a regular solution to the system \eqref{ER-SF}. Additionally, the uniqueness of the regular solution can be established using classical methods. Moreover, employing weak sequential lower semi-continuity and leveraging \eqref{CL-T6} and \eqref{CL-T7}, we derive the desired estimate \eqref{SE-ZETA}.
Furthermore, from the definition of the constant $\mathscr{C}$, we observe that $\zeta\to 0$ in $\mathcal{M}$ as $\rho \to 0^+$.
\end{proof}

\subsection{First-Order Optimality Condition}\label{S-FOOC}

	Next, we proceed to establish the first-order necessary optimality condition for the problem TOCP in the form of a variational inequality by employing the classical adjoint problem approach.

	\begin{proof}[Proof of Theorem \ref{T-FOOCT}]
We will prove this theorem in five steps. As a consequence of Theorem \ref{T-RS}, we define a function, known as the control-to-state operator, which maps each control in $\mathcal{U}_{\text{ad}}$ to a corresponding solution in $\mathcal{M}$ for the system \eqref{NLP}.

	\noindent \textbf{Step I:} The control-to-state map is differentiable at $\widetilde{u}$. 
	
	For $\rho>0$ and for any $\widetilde{m}\in \mathcal{M}$ as well as $\widetilde{u},h\in \mathcal{U}$, we formally define the linearized systems as follows:
	\begin{equation}\label{T-LS}
		\ \ \ \ \ \ \ \ \ \ \ \begin{cases}
			\mathcal{L}_{\widetilde{u}}z= h+ \widetilde{m} \times h  \ \  \ \text{in} \ \ \Omega_T,\\
			\frac{\partial z}{\partial \eta}=0 \ \ \ \ \text{on} \ \partial \Omega_T,\\
			z(0)=0 \ \ \text{in} \ \Omega,
		\end{cases}
	\end{equation}
	where $\mathcal{L}_{\widetilde{u}}$ is the linearized operator associated with $\widetilde{u}$ defined in \eqref{CLO}.
		
	Since $h + \widetilde{m} \times h\in L^2(0,T;H^1(\Omega))$, it follows from Lemma \ref{L-SLS} that the system \eqref{T-LS} admits a unique regular solution $z\in \mathcal{M}$. 

\begin{Lem}\label{PAN-1}
For $\widetilde{u},h \in L^2(0,T;H^1(\Omega))$, $\widetilde{m},z\in \mathcal{M}$ and $\rho>0$, consider the following system:
	\begin{equation}\label{T-LS2}
		\ \ \ \ \ \ \ \begin{cases}
			\mathcal{L}_{\widetilde{u}}\, \eta_\rho = \displaystyle \sum_{i=1}^6 \mathcal{K}_i\ \ \ \text{in} \ \ \Omega_T,\\
			\frac{\partial \eta_\rho}{\partial \eta}=0 \ \ \ \ \text{on} \ \partial \Omega_T,\\
			\eta_\rho(0)=0 \ \ \text{in} \ \Omega,
		\end{cases}
	\end{equation}
	where $\mathcal{L}_{\widetilde{u}}$ is the linearized operator defined in \eqref{CLO}, and
	\begin{flalign*}
		\quad \quad \quad \quad \quad  \mathcal{K}_1&=  \rho \, z \times \Delta z +\rho \, z \times h    - 2\, \rho \, \big(\widetilde{m} \cdot z\big)z  - \rho\, |z|^2\widetilde{m} - \rho^2 \, |z|^2z,&   \\
		\mathcal{K}_2&=\rho \, z \times \Delta \eta_\rho + \rho \, \eta_\rho \times \Delta z + \rho \, \eta_\rho \times \Delta \eta_\rho, \\
		\mathcal{K}_3&=  \rho \, \eta_\rho \times h,  \\
		\mathcal{K}_4&= -2\, \rho \, \big(\widetilde{m} \cdot z\big) \eta_\rho   - 2\, \rho \, \big(\widetilde{m}\cdot \eta_\rho\big) z -2\ \rho \, \big(z\cdot \eta_\rho\big) \widetilde{m}  - \rho^2 \, |z|^2 \eta_\rho -2 \rho^2 \, \big(z\cdot \eta_\rho\big)z,  \\
		\mathcal{K}_5&=-\ 2\ \rho\, \big(\widetilde{m}\cdot \eta_\rho\big) \eta_\rho  -2 \rho^2 \big(z\cdot \eta_\rho\big)\eta_\rho-\rho\, |\eta_\rho|^2 \widetilde{m}  -\ \rho^2 \,  |\eta_\rho|^2z,\\
		\mathcal{K}_6&= - \,\rho^2\, |\eta_\rho|^2 \eta_\rho.
	\end{flalign*}
	There exists a constant $\rho_0>0$ such that the system \eqref{T-LS2} has a unique regular solution\  $\eta_{\rho}\in \mathcal{M}$ for every $0 <\rho \leq \rho_0$ . Moreover,  $\eta_\rho\to 0$ in $\mathcal{M}$ as $\rho \to 0^+$ 
\end{Lem}
\noindent The proof of this lemma is given at the end of this subsection.

	For any feasible direction $h$ at $\widetilde{u}$, let us denote $u_{\rho}=\widetilde{u}+\rho \, h $ and  $m_{\rho}$ represents the state corresponding to the control $u_{\rho}$. From the solvability of systems \eqref{NLP}, \eqref{T-LS}, and \eqref{T-LS2}, it follows that $\widetilde{m}$, $z$, and $\eta_{\rho}$ are the unique regular solutions of these systems. Multiplying these equations respectively by $1$, $\rho$, and $\rho$, and then adding them together, we obtain an identity that can be directly compared with \eqref{NLP}. This yields the representation
		\begin{equation}\label{FO-E1}
	m_{\rho}=\widetilde{m} + \rho \, z + \rho \, \eta_{\rho}.		
		\end{equation}

\noindent Since $\eta_\rho\to0$ in $\mathcal{M}$, from \eqref{FO-E1} and the solvability of systems \eqref{T-LS} and \eqref{T-LS2}, it follows that
$$\frac{\|m_\rho-\widetilde{m}-\rho z\|_{\mathcal{M}}}{\rho}\to 0\ \ \ \ \ \ \text{as}\ \rho\to 0^+.$$
This establishes that the control-to-state map is Gateaux differentiable at $\widetilde{u}$ in the direction of $h$ , that is, $$z=D_um(\widetilde{u})[h].$$

\noindent \textbf{Step II:} \ For small enough $\rho>0$, $T^*(u_\rho)$ is well defined and lies in $(0,T)$.

	The solvability of system \eqref{T-LS}, \eqref{T-LS2} and the fact that $\eta_{\rho}\to 0$ in $\mathcal{M}$ ensures that $m_\rho$ converges to $\widetilde{m}$ in $W^{1,2}(0,T;H^3(\Omega),H^1(\Omega))$. This implies that $\|m_\rho-\widetilde{m}\|_{C(0,T;H^2(\Omega))}\to 0$ as $\rho \to 0^+$. Therefore, for any $\epsilon>0$, we can choose $\rho>0$ sufficiently small such that $\|m_\rho(t)-\widetilde{m}(t)\|_{L^2(\Omega)}< \epsilon$. 
	 
	Under assumptions \eqref{CON-1} and \eqref{CON-2}, there exists a time $\shat{T} \in (0,T)$ such that $\|\widetilde{m}(\shat{T})-m_{\Omega}\|_{L^2(\Omega)} < \delta $. Moreover, by continuity of $\widetilde{m}$, there exists $\sigma>0$ such that $\|\widetilde{m}(t)-\widetilde{m}(\shat{T})\|_{L^2(\Omega)}<\epsilon$ whenever $|t-\shat{T}|<\sigma$.  Therefore, by choosing $\epsilon <(\delta- \|\widetilde{m}(\shat{T})-m_{\Omega}\|_{L^2(\Omega)})/2$, we can deduce
		$$ 	\|m_\rho(t)-m_{\Omega}\|_{L^2(\Omega)} \leq \| m_\rho(t) -\widetilde{m}(t)\|_{L^2(\Omega)}+ \|\widetilde{m}(t)- \widetilde{m}(\shat{T})\|_{L^2(\Omega)} + \|\widetilde{m}(\shat{T})-m_{\Omega}\|_{L^2(\Omega)} < \delta,$$  
for $t\in(\shat{T}-\sigma,\shat{T}+\sigma)\subset (0,T)$ and $\rho>0$ sufficiently small. Consequently, continuity of $\|m_\rho(t)-m_{\Omega}\|_{L^2(\Omega)}$ guarantees the well-definedness of $T^*(u_\rho)$ for sufficiently small $\rho>0$. Moreover, by Lemma \ref{L-EQ}, we have $\|m_\rho(T^*(u_\rho))-m_{\Omega}\|_{L^2(\Omega)}=\delta$.

The well-definedness of $T_\rho:=T^*(u_\rho)$ is a consequence of step II. From this point forward, we will implicitly assume that $\rho$ is sufficiently small to ensure $T_{\rho}$ is well-defined.

		\noindent \textbf{Step III:} \ Next, we will show that $T_\rho\to \widetilde{T}$ as $\rho \to 0^+$.

Let us prove this by the method of contradiction. Assume that there exists a sub-sequence $T_{\rho_n}$ such that $T_{\rho_n}\to T^\#$ as $n \to \infty$, for some element $T^\# \neq \widetilde{T}$ in $[0,T]$. Note that $\rho_n$ is a particular sequential path along $\rho\to0^+$. Now, there will be two cases such that $T^\# < \widetilde{T}$ or $\widetilde{T}>T^\#$. 

\underline{Case I:} \ $T^\# <\widetilde{T}$. Since $\widetilde{m}$ is a continuous function from $[0,T]$ to $L^2(\Omega)$, for any $\epsilon >0$ there exists a number $n_1 \in \mathbb{N}$ such that $\|\widetilde{m}(T^\#)-\widetilde{m}(T_{\rho_n})\|_{L^2(\Omega)}<\epsilon$ for every $n\geq n_1$. Moreover, by the strong convergence $m_{\rho_n}\to \widetilde{m}$ in $C([0,T];L^2(\Omega))$, we can find a number $n_2\in \mathbb{N}$ such that $\|\widetilde{m}(T_{\rho_n})-m_{\rho_n}(T_{\rho_n})\|_{L^2(\Omega)} <\epsilon$ for every $n\geq n_2$. Finally, as 
$T_{\rho_n}=T^\#(u_{\rho_n})$ is minimal time for the control $u_{\rho_n}$, so we can write $\|m_{\rho_n}(T_{\rho_n})-m_{\Omega}\|_{L^2(\Omega)}=\delta$. Combining all these results, we can write
$$\|\widetilde{m}(T^\#)-m_{\Omega}\|_{L^2(\Omega)}\leq \|\widetilde{m}(T^\#)-\widetilde{m}(T_{\rho_n})\|_{L^2(\Omega)}+\|\widetilde{m}(T_{\rho_n})-m_{\rho_n}(T_{\rho_n})\|_{L^2(\Omega)}+\|m_{\rho_n}(T_{\rho_n})-m_{\Omega}\|_{L^2(\Omega)} < 2\epsilon+\delta$$
for $n\geq \max\{n_1,n_2\}$. Since $\epsilon>0$ is arbitrary, we get a time $T^\#<\widetilde{T}$ such that $\|\widetilde{m}(T^\#)-m_{\Omega}\|_{L^2(\Omega)}\leq \delta $. This is a contradiction, since $\widetilde{T}$ is the optimal time for the control $\widetilde{u}$.

\underline{Case II:}\ $T^\#>\widetilde{T}$. Because of assumption \eqref{CON-2}, $\|\widetilde{m}(t)-m_{\Omega}\|_{L^2(\Omega)}$ is strictly decreasing  in a neighbourhood of $\widetilde{T}$. Therefore, there exist constants $\epsilon>0$ and $\tau_1,\tau_2>0$ with $\tau_1<\tau_2<T^\#-\widetilde{T}$ such that $\|\widetilde{m}(t)-m_{\Omega}\|_{L^2(\Omega)}< \delta -\epsilon$ on the interval $[\widetilde{T}+\tau_1,\widetilde{T}+\tau_2]$. Moreover, since $\{m_{\rho_n}\}$ converges to $\widetilde{m}$ in $C([0,T];L^2(\Omega))$, there exists a constant $n_0$ depending on $\epsilon$ such that $\|m_{\rho_n}(t)-\widetilde{m}(t)\|_{L^2(\Omega)}< \epsilon$ for all $n \geq n_0$ and $t\in [0,T]$. Using all these results, we obtain the following inequality,
$$\|m_{\rho_n}(t)-m_{\Omega}\|_{L^2(\Omega)} \leq \|m_{\rho_n}(t)-\widetilde{m}(t)\|_{L^2(\Omega)}+\|\widetilde{m}(t)-m_{\Omega}\|_{L^2(\Omega)}<\delta\ \ \ \ \forall\ n\geq n_0\ \ \ \text{and}\ \ t\in [\widetilde{T}+\tau_1,\widetilde{T}+\tau_2], $$
which is a contradiction as $T^\#(u_{\rho_n})=T_{\rho_n} \to T^\#>\widetilde{T}+\tau_2$.

Therefore, there exists no subsequence of the set $\{T_\rho\}$ that converges to some element other than $\widetilde{T}$. Therefore, $T_\rho$ must converge to $\widetilde{T}$ as $\rho\to0^+$.

	\noindent \textbf{Step IV:} The operator $T^*:\rho \mapsto T^*(u_\rho)=T_\rho$ has a right derivative at $\rho=0$.

	Since $\big(\widetilde{m}_t(\widetilde{T})\ ,\  \widetilde{m}(\widetilde{T})-m_{\Omega} \big)_{L^2(\Omega)} \neq  0$ from assumption \eqref{CON-2}, we can choose a real number $ \mathcal{D}_{\widetilde{u}}\in \mathbb{R}$ such that the following equality holds:
	\begin{equation}\label{TD-1}
		\Big( z ( \widetilde{T}) +  \mathcal{D}_{\widetilde{u}}\ \widetilde{m}_t(\widetilde{T})\ ,\  \widetilde{m}(\widetilde{T})-m_{\Omega} \Big)_{L^2(\Omega)}=0.
	\end{equation}
In order to prove the right derivative of the operator $T^*$ at point $\widetilde{T}$, we will show that the remainder term $\Psi_{\rho}$ goes to $0$ in the following equality:
\begin{equation}\label{TD-2}
	T_\rho=\widetilde{T} +\rho \ \mathcal{D}_{\widetilde{u}} + \rho\ \Psi_\rho.
\end{equation} 
As $T_\rho$ is the minimal time for the control $u_\rho$, so using $\|m_\rho(T_\rho)-m_{\Omega}\|^2_{L^2(\Omega)}=\delta^2$ and \eqref{FO-E1}, we can write
$$\|\widetilde{m}(T_\rho)+\rho \ z(T_\rho)+ \rho\ \eta_\rho -m_{\Omega}\|^2=\delta^2.$$
By adding and subtracting $\widetilde{m}(\widetilde{T})$ inside the norm and using the equality $\|\widetilde{m}(\widetilde{T})-m_{\Omega}\|_{L^2(\Omega)}=\delta$, we get
\begin{align}\label{FO-E2}
\frac{1}{\rho}	\Big(\widetilde{m}(T_\rho)&-\widetilde{m}(\widetilde{T})\,,\,\widetilde{m}(T_\rho)+\widetilde{m}(\widetilde{T}\Big)-2m_{\Omega})_{L^2(\Omega)}+2\ \Big(z(T_\rho)+\eta_\rho(T_\rho)\,,\,\widetilde{m}(\widetilde{T})-m_{\Omega}\Big)_{L^2(\Omega)}\nonumber\\
	&= -\rho\ \|z(T_\rho)+\eta_\rho(T_\rho)\|^2_{L^2(\Omega)} + 2\ \Big(\widetilde{m}(T_\rho)-\widetilde{m}(\widetilde{T})\,,\,z(T_\rho)+\eta_\rho(T_\rho)\Big)_{L^2(\Omega)}.
\end{align}
By multiplying and dividing $T_\rho-\widetilde{T}$ in the first term and using the equality \eqref{TD-2}, we have
$$\big(\mathcal{D}_{\widetilde{u}}+\Psi_\rho\big)\  \bigg(\frac{\widetilde{m}(T_\rho)-\widetilde{m}(\widetilde{T})}{T_\rho-\widetilde{T}}\ ,\ \widetilde{m}(T_\rho)+\widetilde{m}(\widetilde{T})-2m_{\Omega}\bigg)_{L^2(\Omega)}.$$
By substituting this into \eqref{FO-E2}, letting $\rho\to 0^+$ and using the equality $\displaystyle \lim_{\rho\to0^+} \frac{\widetilde{m}(T_\rho)-\widetilde{m}(\widetilde{T})}{T_\rho-\widetilde{T}}=\widetilde{m}_t(\widetilde{T})$, along with the convergence $T_\rho\to \widetilde{T}$ established in Step III and using \eqref{TD-1}, we obtain
$$\lim_{\rho \to 0^+} \Psi_{\rho} = - \frac{\Big( z ( \widetilde{T}) +  \mathcal{D}_{\widetilde{u}}\ \widetilde{m}_t(\widetilde{T})\,,\, \widetilde{m}(\widetilde{T})-m_{\Omega} \Big)_{L^2(\Omega)}}{\Big(\widetilde{m}(\widetilde{T})-m_{\Omega} \,,\, \widetilde{m}_t(\widetilde{T})\Big)_{L^2(\Omega)}}=0,$$
where we have also used the continuity of $\widetilde{m}$ in $C([0,T];H^1(\Omega))$. Thus from \eqref{TD-2}, it is clear that implies $T^*$ has a right derivative at $\rho=0$, which is written as
$$\mathcal{D}_{\widetilde{u}}=D_uT^*(\widetilde{u})[h].$$

	\noindent \textbf{Step V:} Finally, we will establish the first-order necessary optimality condition using the above four steps. 
	
	Note that, since $\widetilde{u}$ is an optimal control for (TOCP) and $\widetilde{T}$ is the optimal time associated with the control $\widetilde{u}$, it follows that 
	\begin{equation*}
		\frac{\mathcal{J}(u_\rho)-\mathcal{J}(\widetilde{u})}{\rho}= \frac{1}{2}\frac{(T_\rho-\widetilde{T})}{\rho}(T_\rho+\widetilde{T}) + \big(\big(\widetilde{u},h\big)\big)_{W^{1,2}(0,T;H^1(\Omega),H^1(\Omega)^*)} +\frac{1}{2} \rho \ \|h\|^2_{W^{1,2}(0,T;H^1(\Omega),H^1(\Omega)^*)}\geq 0.
	\end{equation*}
By taking $\rho\to 0^+$ and using the fact that $\lim_{\rho\to 0^+} \frac{T_\rho-\widetilde{T}}{\rho}=\mathcal{D}_{\widetilde{u}}$ from equation \eqref{TD-2}, we obtain a primitive form of the first-order variational inequality as follows:
\begin{equation}\label{FO-E3}
	\mathcal{D}_{\widetilde{u}} \ \widetilde{T} + \big(\big(\widetilde{u},h\big)\big)_{W^{1,2}(0,T;H^1(\Omega),H^1(\Omega)^*)}\geq 0.
\end{equation}
Considering the $L^2$ inner product of system \eqref{T-LS} with $\phi$ and doing integration by parts of some of the terms, we get
	\begin{align}\label{FO-E4}
	\int_{0}^{\widetilde{T}} &\big( z_t,\phi\big)_{L^2(\Omega)} \ dt +\int_{\Omega_{\widetilde{T}}} \nabla z \cdot \nabla \phi \ dx\ dt- \int_{\Omega_{\widetilde{T}}} \big(z \times \Delta \widetilde{m} \big)\cdot \phi \ dx\ dt + \int_{\Omega_{\widetilde{T}}}  \nabla z \cdot \nabla \big(\phi \times \widetilde{m}\big)  \ dx\ dt  \nonumber\\
	& - \int_{\Omega_{\widetilde{T}}}(z \times \widetilde{u})\cdot \phi \ dx\  dt +2  \int_{\Omega_{\widetilde{T}}}\big(\widetilde{m}\cdot z\big)\widetilde{m} \cdot \phi \ dx\ dt +\int_{\Omega_{\widetilde{T}}}  \big(1+|\widetilde{m}|^2\big) z \cdot \phi \ dx\ dt\nonumber\\
	&=\int_{\Omega_{\widetilde{T}}}(\widetilde{m}\times h)\cdot \phi \ dx\ dt + \int_{\Omega_{\widetilde{T}}} h \cdot \phi \ dx\ dt. \ \ \ \ \ \ 
\end{align}
Furthermore, substituting $\vartheta=z$ in the weak formulation \eqref{WEFO} of the adjoint system \eqref{AS} and using the equality
	\begin{equation*}
\int_0^{\widetilde{T}} \big\langle\phi_t, z\big\rangle_{H^1(\Omega)^*\times H^1(\Omega)} \ dt= -	\int_{0}^{\widetilde{T}} \big( z_t,\phi\big) \ dt + \int_\Omega \phi(\widetilde{T}) \cdot z(\widetilde{T})\ dx,
\end{equation*}
we obtain
\begin{align}\label{FO-E6}
	&-	\int_{0}^{\widetilde{T}} \big( z_t,\phi\big) \ dt  - \int_{\Omega_{\widetilde{T}}} \nabla \phi \cdot \nabla  z\ dx \ dt- \int_{\Omega_{\widetilde{T}}}\nabla (\phi \times \widetilde{m})\cdot \nabla  z\ dx \ dt\nonumber\\
	&\hspace{1cm} + \int_{\Omega_{\widetilde{T}}}(\Delta \widetilde{m}\times \phi)\cdot z\ dx\ dt -\int_{\Omega_{\widetilde{T}}} \big(\phi \times \widetilde{u}\big)\cdot z\ dx\ dt- \int_{\Omega_{\widetilde{T}}} \left( 1 + |\widetilde{m}|^2 \right) \phi\cdot z\  dx\ dt\nonumber\\
	&\hspace{1cm} -2\int_{\Omega_{\widetilde{T}}} \big((\widetilde{m} \cdot \phi)\ \widetilde{m}\big)\cdot z \ dx\ dt =- \int_\Omega \phi (\widetilde{T}) \cdot z(\widetilde{T})\ dx,
\end{align}
	with $\phi(\widetilde{T})$ as given in Definition \ref{AWSD}. Finally, using the vector identities
	$a\cdot(b\times c)=-c\cdot (b\times a)$, combining equations \eqref{FO-E4} and \eqref{FO-E6}, and applying the equality \eqref{TD-1}, we obtain
	\begin{align}\label{FO-E7}
		\int_{\Omega_{\widetilde{T}}} \big(\phi \times \widetilde{m}\big)\cdot h\ dx\ dt &+ \int_{\Omega_{\widetilde{T}}} \phi \cdot h\ dx\ dt = \int_{\Omega} \phi(\widetilde{T})\cdot z(\widetilde{T}) \ dx\nonumber\\
		&= -\int_{\Omega} \frac{\widetilde{T}\ \big(\widetilde{m}(\widetilde{T})-m_{\Omega}\big)}{\big(\widetilde{m}(\widetilde{T})-m_{\Omega},\widetilde{m}_t(\widetilde{T})\big)_{L^2(\Omega)}}\cdot z(\widetilde{T})\ dx= \mathcal{D}_{\widetilde{u}}\ \widetilde{T}.
	\end{align}
Then by substituting the above identity in \eqref{FO-E3}, we obtain that for every feasible direction $h$ at $\widetilde{u}$ the following inequality holds:
$$\int_{\Omega_{\widetilde{T}}} \big(\phi \times \widetilde{m}\big)\cdot h\ dx\ dt + \int_{\Omega_{\widetilde{T}}} \phi \cdot h\ dx\ dt +\big(\big(\widetilde{u},h\big)\big)_{W^{1,2}(0,T;H^1(\Omega),H^1(\Omega)^*)}\geq 0.$$
Finally, since the set of feasible directions $\mathcal{F}_{\mathcal{U}_{ad}}(\widetilde{u})$ is dense in $\mathcal{P}_{\mathcal{U}_{ad}}(\widetilde{u})$, the desired result \eqref{FOOC} follows.
\end{proof}
	
	
	

\begin{proof}[Proof of Lemma \ref{PAN-1}]
	By analyzing the systems \eqref{ER-SF} and \eqref{T-LS2}, we can find that each term of $\mathcal{K}_i$ belongs to one of the class of $\mathcal{F}_i$'s. To be specific, we will compare the first expressions $\mathcal{K}_1$ with $\mathcal{F}_1$. The first two terms $z\times \Delta z$ and $ z \times h$ resembles $\varphi_2 \times \Delta \varphi_3$ and $\varphi_4 \times \varphi_5$, respectively. Furthermore, the final three terms of $\mathcal{K}_1$, that is, $\big(\widetilde{m} \cdot z\big)z$, $|z|^2\widetilde{m}$ and  $|z|^2z$ are similar to the term $\big(\varphi_6\cdot \varphi_7\big)\, \varphi_8$ of $\mathcal{F}_1$.
	
	Similarly, proceeding for the rest of the $\mathcal{K}_i$ terms, we get that each term of $\mathcal{K}_i$'s finds resemblance with one of the terms of the corresponding $\mathcal{F}_i$ expression in equation \eqref{ER-SF}. Moreover, since  $h \in L^2(0,T;H^1(\Omega))$ and $\widetilde{m},z\in \mathcal{M}$, the regularity requirements by the functions $\varphi_i$'s, $\psi_i$'s, $\vartheta_i$'s, $\zeta_i$'s in Lemma \ref{L-SER} are automatically satisfied. 
	
	Finally, by applying Lemma \ref{L-SER}, it can be established that the system \eqref{T-LS2} admits a unique regular solution for sufficiently small $\rho$. Furthermore, by deriving the energy estimate analogous to Lemma \ref{L-SER}, we can find that $\eta_\rho\to 0$ in $\mathcal{M}$ as $\rho \to 0^+$. 
\end{proof}

\section{Second-Order Optimality Condition}\label{SEC-SO}

The following lemma will be essential in proving the second-order optimality condition, as it establishes the convergence of the time derivative of the solution to the linearized system.
\begin{Lem}\label{Z-K}
For any $\widetilde{m}\in \mathcal{M}$ and $\widetilde{u},h_k \in \mathcal{U}$, let $z_k\in \mathcal{M}$ be the regular solution of the following linearized system:
\begin{equation}\label{Z-KK}
	z_{kt}=\Delta z_k +z_k\times \Delta \widetilde{m} +\widetilde{m}\times \Delta z_k + z_k \times \widetilde{u}-2\ (\widetilde{m}\cdot z_k)\ \widetilde{m}- \left(1+|\widetilde{m}|^2\right)z_k+ h_k+ \widetilde{m} \times h_k.
\end{equation}
If $z_k$ converges strongly to $z$ in $L^2(0,T;H^2(\Omega))\cap C([0,T];H^1(\Omega))$, and $h_k$ converges to $h$ strongly in $C([0,T];H^1(\Omega)^*)$ then $z_{kt}$ converges to $z_t$ strongly in $C([0,T];H^1(\Omega)^*)$ space. 

\end{Lem}
\begin{proof}[Proof of Lemma \ref{Z-K}]
	We will prove the convergence of $z_{kt}$ by showing that each term on the right-hand side of \eqref{Z-KK} converges to the respective terms of \eqref{T-LS} in $C([0,T];H^1(\Omega)^*)$ space. 	
	
	From Lemma \ref{L-SLS}, it is clear that $z_k\in W^{1,2}(0,T;H^3(\Omega),H^1(\Omega))$, and therefore, there exists a function $z\in \mathcal{M}$ such that $z_k\to z$ strongly in $L^2(0,T:H^2(\Omega)) \cap C([0,T];H^1(\Omega))$.
	
	Let us start with the first term on the right-hand side of \eqref{Z-KK}. Using the definition of dual norm and integration by parts, we obtain
	\begin{flalign*}
		\|\Delta z_k-\Delta z\|_{C([0,T];H^1(\Omega)^*)}& = \max_{t\in[0,T]} \ \sup_{\|v\|_{H^1(\Omega)=1}} \big|\big\langle \Delta z_k(t)-\Delta z(t), v\big\rangle_{H^1(\Omega)^*\times H^1(\Omega)}\big|&\\
		& 	= \max_{t\in[0,T]} \ \sup_{\|v\|_{H^1(\Omega)=1}} \big|\big( \nabla z_k(t)-\nabla z(t), \nabla v\big)_{L^2(\Omega)}\big|\\
		& 	\leq  \max_{t\in[0,T]} \ \sup_{\|v\|_{H^1(\Omega)=1}} \big\|\nabla  z_k(t)-\nabla z(t)\big\|_{L^2(\Omega)}\ \|\nabla v\|_{L^2(\Omega)} 	\leq  \max_{t\in[0,T]} \ \big\|z_k(t)-z(t)\big\|_{H^1(\Omega)}.
	\end{flalign*}
	For the second and fifth term, applying H\"older's inequality, and the cross product property from Lemma \ref{CPP}, followed by the embedding $H^1(\Omega) \hookrightarrow L^4(\Omega)$, we get 
	\begin{flalign*}
		\|(z_k-z)\times &\Delta \widetilde{m}\|_{C([0,T];H^1(\Omega)^*)}  = \max_{t\in[0,T]} \ \sup_{\|v\|_{H^1(\Omega)=1}} \big|\big\langle  \big(z_k(t)-z(t)\big)\times \Delta \widetilde{m}(t), v\big\rangle_{H^1(\Omega)^*\times H^1(\Omega)}\big|&\\
		& 	= \max_{t\in[0,T]} \ \sup_{\|v\|_{H^1(\Omega)=1}} \big|\big( \nabla \widetilde{m}(t), \nabla \big(v\times (z_k(t)-z(t))\big) \big)_{L^2(\Omega)}\big|\\
		& 	\leq  \max_{t\in[0,T]} \ \sup_{\|v\|_{H^1(\Omega)=1}} \ \|\nabla \widetilde{m}(t)\|_{L^4(\Omega)} \ \Big( \|\nabla v\|_{L^2(\Omega)} \ \| z_k(t)-z(t)\|_{L^4(\Omega)}+ \|v\|_{L^4(\Omega)} \ \|\nabla (z_k(t)-z(t))\|_{L^2(\Omega)}\Big)\\
		&\leq C\  \|\widetilde{m}\|_{L^\infty(0,T;H^2(\Omega))}\ \max_{t\in[0,T]} \ \big\|z_k(t)-z(t)\big\|_{H^1(\Omega)},
	\end{flalign*}
	and 	
	\begin{flalign*}
		\|\big(\widetilde{m}\cdot (z_k-z)\big)\ \widetilde{m}\|_{C([0,T];H^1(\Omega)^*)}  &= \max_{t\in[0,T]} \ \sup_{\|v\|_{H^1(\Omega)=1}} \big|\big\langle  \big(\widetilde{m}\cdot (z_k-z)\big)\ \widetilde{m}, v\big\rangle_{H^1(\Omega)^*\times H^1(\Omega)}\big|&\\
		& 	\leq  \max_{t\in[0,T]} \ \sup_{\|v\|_{H^1(\Omega)=1}} \ \|\widetilde{m}(t)\|^2_{L^4(\Omega)} \  \|z_k(t)-z(t)\|_{L^4(\Omega)} \ \|v\|_{L^4(\Omega)}\\
		&\leq  \|\widetilde{m}\|^2_{L^\infty(0,T;H^2(\Omega))}\ \max_{t\in[0,T]} \ \big\|z_k(t)-z(t)\big\|_{H^1(\Omega)}.
	\end{flalign*}
	Again proceeding for the final term $\widetilde{m}\times h_k$, using the dual norm definition and the strong convergence of $h_k$ to $h$ in $C([0,T];H^1(\Omega)^*)$, we derive
	\begin{flalign*}
		\|\widetilde{m}\times (h_k-h)\|_{C([0,T];H^1(\Omega)^*)}  &= \max_{t\in[0,T]} \ \sup_{\|v\|_{H^1(\Omega)=1}} \big|\big\langle   \widetilde{m}(t)\times \big( h_k(t)-h(t)\big), v\big\rangle_{H^1(\Omega)^*\times H^1(\Omega)}\big|&\\
		& 	= \max_{t\in[0,T]} \ \sup_{\|v\|_{H^1(\Omega)=1}} \big|\big\langle   h_k(t)-h(t), v\times \widetilde{m}(t)\big\rangle_{H^1(\Omega)^*\times H^1(\Omega)}\big|\\
		& 	\leq  \max_{t\in[0,T]} \ \sup_{\|v\|_{H^1(\Omega)=1}} \ \|h_k(t)-h(t)\|_{H^1(\Omega)^*}\ \|v\times \widetilde{m}(t)\|_{H^1(\Omega)}\\
		& 	\leq  \max_{t\in[0,T]} \ \sup_{\|v\|_{H^1(\Omega)=1}} \ \|h_k(t)-h(t)\|_{H^1(\Omega)^*}\ \|v\|_{H^1(\Omega)}\ \|\widetilde{m}(t)\|_{H^2(\Omega)}\\
		&\leq  \|\widetilde{m}\|_{L^\infty(0,T;H^2(\Omega))}\ \max_{t\in[0,T]} \ \big\|h_k(t)-h(t)\big\|_{H^1(\Omega)^*}
	\end{flalign*}
	The remaining terms can be estimated in a similar fashion. By letting $k\to \infty$ and applying the strong convergence of $z_k$ in $C([0,T];H^1(\Omega))$ and $h_k$ in $C([0,T];H^1(\Omega)^*)$, the claim follows. This concludes the proof.
\end{proof}	

Let us now focus on the second-order optimality conditions for time-optimal control problems (TOCP), particularly when the control system is governed by a nonlinear parabolic equation. These problems are especially challenging due to the nonlinearity of the state equation, placing them in the realm of nonconvex optimization. The lack of convexity introduces significant difficulties in deriving optimality conditions, as the structural properties often used in linear problems no longer apply. This makes the analysis of nonlinear systems much more complex. Only a few studies (see, e.g., \cite{CTA, VB1, VB2, EFC, KKWL}) have addressed time-optimality conditions for such problems, highlighting the unique challenges introduced by the nonlinear nature of the state equation.

\begin{proof}[Proof of Theorem \ref{T-SOOC}]
	We will establish this theorem by the method of contradiction, which will be structured in three steps.
	
	\noindent \textbf{Step I:}	 Assume, for the sake of contradiction, that $\widetilde{u}$ does not fulfill the growth condition specified by inequality \eqref{SO-GC}. In this case, there must exists a sequence of control functions $\{u_k\}_{k=1}^\infty$ within $\mathcal{U}_{ad}$ such that $u_k\to \widetilde{u}$ in $W^{1,2}(0,T;H^1(\Omega),H^1(\Omega)^*)$ such that

	\begin{equation}\label{CQGC}
		\mathcal{J}(u_k) < \mathcal{J}(\widetilde{u}) + \frac{1}{k}\ \|u_k-\widetilde{u}\|^2_{W^{1,2}(0,T;H^1(\Omega),H^1(\Omega)^*)}\ \ \ \ \forall \ k.
	\end{equation}

	Define  \ $\rho_k:=\|u_k-\widetilde{u}\|_{W^{1,2}(0,T;H^1(\Omega),H^1(\Omega)^*)}$ and $h_k:=\frac{1}{\rho_k}\ (u_k-\widetilde{u})$. Given that $\|h_k\|_{W^{1,2}(0,T;H^1(\Omega),H^1(\Omega)^*)}=1$, we can select a subsequence, for simplicity again denoted as $\{h_k\}$, such that 
	\begin{align}\label{CON-H1}
		h_k &\rightharpoonup h\ \text{ weakly in}\  L^2(0,T;H^1(\Omega)),\nonumber\\
		\text{and}\ \ \ (h_k)_t &\overset{*}{\rightharpoonup} h_t\ \text{ weak$^*$ in}\  L^2(0,T;H^1(\Omega)^*),
	\end{align}
	for some element $h\in W^{1,2}(0,T;H^1(\Omega),H^1(\Omega)^*)$.  Moreover, as a result of the continuous embedding $W^{1,2}(0,T;H^1(\Omega),H^1(\Omega)^*) \hookrightarrow C([0,T];L^2(\Omega))$  (see \cite{LCE}, Subsection 5.9.2), it follows that $\{h_k\}$ is uniformly bounded in $C([0,T];L^2(\Omega))$. Moreover, implementing the Aubin-Lions-Simon Lemma (see Corollary 4,\cite{JS}), we have
	\begin{equation}\label{CON-H2}
	h_k \to h \  \text{strongly in}\ L^2(0,T;L^2(\Omega))\cap C([0,T];H^1(\Omega)^*).	
	\end{equation}
		
	Suppose $z_k$ and $z$ are the unique regular solutions of the linearized system \eqref{T-LS} with the right-hand side terms being replaced by $h_k + \widetilde{m}\times h_k$ and $h+\widetilde{m}\times h$, respectively. Since $z_k$ is the unique regular solution, through the energy estimate \eqref{LSSE}, we find $\|z_k\|_{W^{1,2}(0,T;H^3(\Omega),H^1(\Omega))}\leq C\ \|h_k\|_{L^2(0,T;H^1(\Omega))}\leq C$. Using this and the compact embedding of $W^{1,2}(0,T;H^3(\Omega),H^1(\Omega))$ in  $L^2(0,T;H^2(\Omega))$, we find a subsequence (again represented as $z_k$) such that 
	\begin{equation}\label{CON-Z1}
		z_k \ \text{converges strongly to}\ \widetilde{z} \ \text{in}\ L^2(0,T;H^2(\Omega)).
	\end{equation}	
	Furthermore, as $z_k$ is bounded in $L^\infty(0,T;H^2(\Omega))$ and $(z_k)_t$ is bounded in $L^2(0,T;H^1(\Omega))$, therefore by the compactness result from Corollary 4 of \cite{JS}, extraction of a subsequence (again represented as $z_k$) is possible such that the following convergence holds:
	\begin{equation}\label{CON-Z}
	z_k \to z \ \text{strongly in}\ C([0,T];H^1(\Omega)).
	\end{equation}
	\noindent Now, consider the following linear system:	
	\begin{equation}\label{XI-LS}
		\begin{cases}
			\begin{array}{l}
				\mathcal{L}_{\widetilde{u}}\xi_k=2\,  z_k\times \Delta z_k +2\, z_k\times h_k-4\,\big(z_k\cdot \widetilde{m}\big)\,z_k-2\,|z_k|^2\, \widetilde{m}  \ \ \ \ \text{in}\ \Omega_T,\\
				\frac{\partial \xi_k}{\partial \eta}=0 \ \ \ \ \ \  \text{in}\ \partial \Omega_T, \ \ \ \xi_k(x,0)=0\ \ \text{in}\ \Omega,
			\end{array}
		\end{cases}	
	\end{equation}	
	where, the operator $\mathcal{L}_{\widetilde{u}}$ is defined in \eqref{CLO}. Since $h_k \in L^2(0,T;H^1(\Omega))$ and $ \widetilde{m}, z_k \in L^2(0,T;H^2(\Omega)) \cap L^{\infty}(0,T;H^2(\Omega))$, therefore, using inequalities \eqref{EE-2}, \eqref{EE-3} and \eqref{EE-6} from Lemma \ref{PROP2}, we get that each term of on the right-hand side of \eqref{XI-LS} belongs to $L^2(0,T;H^1(\Omega))$. Finally, it follows from Lemma \ref{L-SLS} that the system \eqref{XI-LS} admits a unique regular solution $\xi_k$ in $\mathcal{M}$. 
	
 Since $\xi_k$ is the regular solution to the system \eqref{XI-LS}, where the right-hand side variables $z_k$ and $h_k$ converges to $z$ and $h$, respectively, as given in \eqref{CON-H1}, \eqref{CON-H2}, \eqref{CON-Z1} and \eqref{CON-Z}, therefore proceeding analogously as we did for $z_k$, we find that 
 \begin{equation}\label{CON-XI}
	 	\xi_k \to \xi\ \ \text{strongly in}\ L^2(0,T;H^2(\Omega)) \cap C([0,T];H^1(\Omega)).
 \end{equation}

 \begin{Lem}\label{L-PAN-XI}
 	Let $\widetilde{u}, h_k \in L^2(0,T;H^1(\Omega))$ and $\widetilde{m}, z_k, \xi_k \in L^2(0,T;H^3(\Omega))\cap L^{\infty}(0,T;H^2(\Omega))$. For any constant $\rho_k>0$, we consider the following non-linear system:
	\begin{equation}\label{ER-SY}
		\begin{cases}
			\begin{array}{l}
				\mathcal{L}_{\widetilde{u}}\zeta_k=\displaystyle \sum_{i=1}^6 \mathcal{X}_i, \\
				\frac{\partial \zeta_k}{\partial \eta}=0 \ \ \ \ \  \text{in}\ \partial \Omega_T, \ \ \ \ \ \ \  \ \ \zeta_k(x,0)=0\ \ \text{in}\ \Omega,
			\end{array}
		\end{cases}	
	\end{equation}
	where $\mathcal{L}_{\widetilde{u}}$ is defined in \eqref{CLO}, and the terms $\mathcal{X}_i$ are given by 
\begin{align*}
\mathcal{X}_1 &= \frac{\rho_k}{2}\, z_k\times \Delta \xi_k +\frac{\rho_k}{2}\, \xi_k\times \Delta z_k + \frac{\rho_k^2}{4} \, \xi_k \times \Delta \xi_k  +\frac{\rho_k}{2}\, \xi_k \times h_k   - \,\rho_k\  \big(z_k\cdot \xi_k\big)\, \widetilde{m}   - \rho_k\,  \big(\widetilde{m}\cdot  \xi_k\big)\, z_k  \\
&\hspace{0.6cm} - \rho_k^2\, \big(z_k\cdot \xi_k\big)\,z_k -\, \rho_k\, \big(\widetilde{m}\cdot z_k\big)\,\xi_k - \frac{\rho_k^2}{2}\, \big(\widetilde{m}\cdot \xi_k\big)\,\xi_k - \frac{\rho_k^3}{2}\, \big(z_k\cdot \xi_k\big)\,\xi_k -\frac{\rho_k^2}{4}\, \big|\xi_k\big|^2\ \widetilde{m}  -\, \rho_k\, \big|z_k\big|^2\, z_k \\
&\hspace{0.6cm} - \frac{\rho_K^3}{4} \, \big|\xi_k\big|^2\,z_k  - \frac{\rho_k^2}{2}\, \big|z_k\big|^2\, \xi_k -\, \frac{\rho_k^4}{8}\, \big|\xi_k\big|^2\,\xi_k,\\
\mathcal{X}_2 &=  \rho_k\, z_k\times \Delta \zeta_k +\frac{\rho_k^2}{2} \,\xi_k \times \Delta \zeta_k + \rho_k\,  \zeta_k \times \Delta z_k  +\, \frac{\rho_k^2}{2}\, \zeta_k \times \Delta \xi_k   + \rho_k^2\, \zeta_k \times \Delta \zeta_k,\\
\mathcal{X}_3 &=  \rho_k\, \zeta_k \times \widetilde{u} + \rho_k \,\zeta_k \times h_k,\\ 
\mathcal{X}_4&= -\,  2\rho_k\,  \big(z_k\cdot \zeta_k\big)\, \widetilde{m}- \rho_k^2\, \big(\xi_k\cdot \zeta_k\big)\, \widetilde{m} -2\rho_k\, \big(\widetilde{m} \cdot \zeta_k\big)\,z_k  - 2\rho_k^2\, \big(z_k\cdot \zeta_k\big)\,z_k- \rho_k^3\, \big(\xi_k \cdot \zeta_k\big)\, z_k  \\
&\hspace{0.6cm} -\rho_k^2\, \big(\widetilde{m}\cdot \zeta_k\big)\, \xi_k - \rho_k^3\, \big(z_k\cdot \zeta_k\big)\, \xi_k  - \frac{\rho_k^4}{2}\, \big(\xi_k\cdot \zeta_k\big)\,\xi_k   -2\rho_k\, \big(\widetilde{m} \cdot z_k\big)\,\zeta_k - \rho_k^2\, \big(\widetilde{m}\cdot \xi_k\big)\,\zeta_k \\
&\hspace{0.6cm}  - \rho_k^3\, \big(z_k\cdot \xi_k\big)\, \zeta_k -\, \rho_k^2\, \big|z_k\big|^2\,\zeta_k- \frac{\rho_k^4}{4}\, \big|\xi_k\big|^2\,\zeta_k,\\
\mathcal{X}_5&= - 2\rho_k^2\, \big(\widetilde{m}\cdot \zeta_k\big)\,\zeta_k  - 2\rho_k^3\, \big(z_k\cdot \zeta_k\big)\, \zeta_k  -\rho_k^4\, \big(\xi_k\cdot \zeta_k\big)\, \zeta_k -\rho_k^2\, \big|\zeta_k\big|^2\, \widetilde{m} - \rho_k^3\, \big|\zeta_k\big|^2\,z_k - \frac{\rho_k^4}{2}\, \big|\zeta_k\big|^2\, \xi_k,\\
\mathcal{X}_6&=-\rho_k^4\, \big|\zeta_k\big|^2\,\zeta_k.
\end{align*}
There exists a constant $\rho_0>0$ ensuring that for each $0<\rho_k \leq \rho_0$, the system \eqref{ER-SY} possesses a unique regular solution $\zeta_k\in \mathcal{M}$. Furthermore, the solution $\zeta_k$ vanishes in $\mathcal{M}$ as $k \to \infty$, that is, 
	\begin{equation}\label{CON-ZT}
	\zeta_k \ \text{converges to}\ \ 0\ \ \text{in}\  \mathcal{M}\ \ \text{as}\ k\to +\infty.
\end{equation}
\end{Lem}
\noindent The proof of Lemma \ref{L-PAN-XI} is given at the end of this section.

	Let us define the control $u_k=\widetilde{u} + \rho_k \,h_k$, and $m_k$ represents the unique regular solution of system \eqref{NLP} corresponding to the control $u_k$. Since $\widetilde{m}$, $z_{k}$, $\xi_{k}$, and $\zeta_{k}$ are the unique regular solutions of the systems \eqref{NLP}, \eqref{T-LS}, \eqref{XI-LS}, and \eqref{ER-SY}, respectively, we may proceed as follows. By weighting these equations with the factors $1$, $\rho_{k}$, $\tfrac{\rho_{k}^{2}}{2}$, and $\rho_{k}^{2}$, and then adding the resulting identities, we obtain an expression which, upon comparison with \eqref{NLP}, leads to	
	\begin{equation}\label{Z-0}
	m_k=\widetilde{m}+ \rho_k \ z_k + \frac{\rho_k^2}{2}\ \xi_k + \rho_k^2\ \zeta_k.
	\end{equation}
	
	\noindent \textbf{Step II:}	As a consequence of Step-II of Theorem \ref{T-FOOCT}, we define $T^*_k=T^*(u_k)$. Since $\big(\widetilde{m}(\widetilde{T})-m_{\Omega},\widetilde{m}_t(\widetilde{T})\big)_{L^2(\Omega)}<0$ from assumption \ref{CON-2}, therefore using the values of  $\mathcal{D}_{\widetilde{u}}$ and $\mathcal{G}_{\widetilde{u}}$ from the statement of the theorem, we will show that the remainder terms in the following two expressions converge to 0:	
	\begin{equation}\label{Z-1}
		T_k^*=\widetilde{T} + \rho_k \ \mathcal{D}_{\widetilde{u}}[h_k] + \rho_k\, \Psi_{\rho_k}, 
	\end{equation}
	and
	\begin{equation}\label{Z-2}
		T_k^*=\widetilde{T} + \rho_k \ \mathcal{D}_{\widetilde{u}}[h_k] +\rho_k^2\ \mathcal{G}_{\widetilde{u}}[h_k,h_k] + \rho_k^2\, \Psi_{\rho_k}.
	\end{equation}

	Note that the sequences $\Psi_{\rho_k}$ in equations \eqref{Z-1} and \eqref{Z-2} could be different.  Our goal is to show that the sequences $\Psi_{\rho_k}$  converge to 0 as $k\to \infty$.

	Our first claim in \eqref{Z-1} can be established by following a similar line of reasoning as presented in Step IV of Theorem \ref{L-EQ}. Moreover, since $\|h_k\|_{\mathcal{U}}=1$ and $\Psi_{\rho_k}$ converges to $0$ in \eqref{Z-1}, there exists a constant $C\geq 0$ and $k_0\in \mathbb{N}$ such that 
	\begin{equation}\label{Z-2.5}
		\left| \frac{T^*_k-\widetilde{T}}{\rho_k} \right| \leq \big| \mathcal{D}_{\widetilde{u}}[h_k] \big| + \big| \Psi_{\rho_k}\big| \leq C \ \ \ \ \ \  \text{for any } k\geq k_0.
	\end{equation}

 To prove \eqref{Z-2}, we will first demonstrate that the following expression holds as an equality in $H^1(\Omega)^*$ space:
	\begin{equation}\label{Z-3}
		z_k(T_k^*)=z_k(\widetilde{T})+z_{kt}(\widetilde{T})(T_k^*-\widetilde{T})+\rho_k \, \mathscr{E}_k,
	\end{equation}
	where $\mathscr{E}_k$ is a sequence converging to $0$ as $k$ tends to $+\infty$. 
	
	For any $v\in H^1(\Omega)$, let us consider the following equality:
	\begin{align}\label{Z-4}
		&\left<\frac{z_k(T_k^*)-z_k(\widetilde{T})-z_{kt}(\widetilde{T})(T_k^*-\widetilde{T})}{\rho_k},\, \upsilon \right>_{H^1(\Omega)^* \times H^1(\Omega)}= \frac{1}{\rho_k} \left<\int_{\widetilde{T}}^{T_k^*} \big(z_{kt}(\tau)-z_{kt}(\widetilde{T})\big)\ d\tau\, ,\, \upsilon \right>_{H^1(\Omega)^* \times H^1(\Omega)}\nonumber\\
		&\hspace{1cm} =  \frac{1}{\rho_k} \left<\int_{\widetilde{T}}^{T_k^*} \big(z_{kt}(\tau)- z_t(\tau)\big)\ d\tau\, ,\, \upsilon \right>_{H^1(\Omega)^* \times H^1(\Omega)}+  \frac{1}{\rho_k} \left<\int_{\widetilde{T}}^{T_k^*} \big(z_t(\tau)-z_t(\widetilde{T})\big)\ d\tau\, ,\, \upsilon \right>_{H^1(\Omega)^* \times H^1(\Omega)}\nonumber\\
		&\hspace{1.3cm}+ \frac{1}{\rho_k} \left<\int_{\widetilde{T}}^{T_k^*} \big(z_t(\widetilde{T})-z_{kt}(\widetilde{T})\big)\ d\tau\, ,\, \upsilon \right>_{H^1(\Omega)^* \times H^1(\Omega)}.
	\end{align}

	Next, we will show that each term on the right-hand side converges to $0$. For the first term applying Cauchy-Schwarz inequality and boundedness result \eqref{Z-2.5}, we get
	\begin{align*}
		&\frac{1}{\rho_k}\left| \left<\int_{\widetilde{T}}^{T_k^*} \big(z_{kt}(\tau)- z_t(\tau)\big)\ d\tau\, ,\, \upsilon \right>_{H^1(\Omega)^* \times H^1(\Omega)} \right|\leq \frac{1}{\rho_k} \int_{\widetilde{T}}^{T_k^*} \Big|\Big(z_{kt}(\tau)- z_t(\tau),\upsilon\Big)_{H^1(\Omega)^* \times H^1(\Omega)}\Big| \ d\tau\\
		&\hspace{1cm}\leq \frac{T_k^*-\widetilde{T}}{\rho_k}\ \|z_{kt}-z_t\|_{C([0,T];H^1(\Omega)^*)}\ \|\upsilon\|_{H^1(\Omega)}\leq C\ \|z_{kt}-z_t\|_{C([0,T];H^1(\Omega)^*)}\ \|\upsilon\|_{H^1(\Omega)}.
	\end{align*}
	Since $z_{kt}$ converges strongly to $z_t$ in $C([0,T];H^1(\Omega)^*)$ by Lemma \ref{Z-K}, it follows that the first term on the right-hand side of \eqref{Z-4} converges to $0$. Similarly, the third term can be shown to converge to $0$ using the same line of reasoning.
	
	Moreover, since $z\in C([0,T];H^2(\Omega))$ and $h\in C([0,T];L^2(\Omega))$, it follows from \eqref{T-LS} that $z_t$ belongs to $C([0,T];L^2(\Omega))$. 
	Therefore, applying H\"older's inequality, followed by estimate \eqref{Z-2.5}, we find that 	
	\begin{align*}
		&\frac{1}{\rho_k}\left| \left<\int_{\widetilde{T}}^{T_k^*} \big(z_{t}(\tau)- z_t(\widetilde{T})\big)\ d\tau\, ,\, \upsilon \right>_{H^1(\Omega)^* \times H^1(\Omega)} \right|\leq \frac{1}{\rho_k} \int_{\widetilde{T}}^{T_k^*} \Big|\Big(z_{t}(\tau)- z_t(\widetilde{T}),\upsilon\Big)_{L^2(\Omega)}\Big| \ d\tau\\
		&\leq \frac{1}{\rho_k} \int_{\widetilde{T}}^{T_k^*} \big\|z_{t}(\tau)- z_t(\widetilde{T})\big\|_{L^2(\Omega)}\ \|\upsilon\|_{L^2(\Omega)} \ d\tau \leq \frac{T_k^*-\widetilde{T}}{\rho_k}\ \big\|z_{t}(\tau)- z_t(\widetilde{T})\big\|_{C([\widetilde{T},T^*_k];L^2(\Omega))} \ \|\upsilon\|_{H^1(\Omega)}\\
		&\leq C\  \big\|z_{t}(\tau)- z_t(\widetilde{T})\big\|_{C([\widetilde{T},T^*_k];L^2(\Omega))} \ \|\upsilon\|_{H^1(\Omega)}\ \ \text{for } k\geq k_0.
	\end{align*}
	Since, $k\to \infty$ implies $\rho_k\to 0$, it follows from \eqref{Z-2.5} that $T^*_k \to \widetilde{T}$. Consequently, $$\big\|z_{t}(\tau)- z_t(\widetilde{T})\big\|_{C([\widetilde{T},T^*_k];L^2(\Omega))}\to0 \ \ \ \ \text{as }\  k \to \infty.$$
 The above reasoning demonstrates that each term on the right-hand side of \eqref{Z-4} converges to $0$. Therefore, we establish our claim that \eqref{Z-3} holds as an equality in the $H^1(\Omega)^*$ space.

 As it is evident from Remark \ref{MT-WPN}  that $\widetilde{m}_t(\widetilde{T})$ exists, and the assumption stated in the theorem validates $\widetilde{m}$ is twice Fr\'echet differentiable at $\widetilde{T},$  we can write the following Taylor expansion of $\widetilde{m}$ around $\widetilde{T}$:
\begin{equation}\label{Z-5}
	\widetilde{m}(T_k^*)-	\widetilde{m}(\widetilde{T})=	\widetilde{m}_t(\widetilde{T})\, (T_k^*-\widetilde{T}) + \frac{1}{2}\ \widetilde{m}_{tt}(\widetilde{T})\,(T_k^*-\widetilde{T})^2 + \rho_k^2 \ \mathscr{E}_k.
\end{equation}

Since $\|m_k(T_k^*)-m_{\Omega}\|_{L^2(\Omega)}=\delta$, it follows from Proposition \ref{L-EQ} and equality \eqref{Z-0} that 
\begin{align}\label{Z-6}
	&0=\big\| \widetilde{m}(T_k^*)+ \rho_k \, z_k(T_k^*) + \frac{\rho_k^2}{2}\ \xi_k(T_k^*) + \rho_k^2\ \zeta_k(T_k^*)  -\widetilde{m}(\widetilde{T})\big\|^2_{L^2(\Omega)}\nonumber\\
	&\hspace{1cm} +2\ \Big(\widetilde{m}(T_k^*)+ \rho_k \, z_k(T_k^*) + \frac{\rho_k^2}{2}\ \xi_k(T_k^*) + \rho_k^2\ \zeta_k(T_k^*)  -\widetilde{m}(\widetilde{T}) \,,\, \widetilde{m}(\widetilde{T}) -m_{\Omega}\Big)_{L^2(\Omega)}:=\sum_{i=1}^{2}\mathcal{B}_i
\end{align}
Now, we will estimate each term on the right-hand side of \eqref{Z-6} individually.

Before proceeding, by applying the convergence $z_k(\widetilde{T}) \to z(\widetilde{T})$ in $H^1(\Omega)$ from \eqref{CON-Z} to the definitions of $\mathcal{D}_{\widetilde{u}}$  given in equations \eqref{SOC-D}, we derive 
\begin{equation}\label{CON-D}
	\mathcal{D}_{\widetilde{u}}[h_k] \to \mathcal{D}_{\widetilde{u}}[h]\ \ \ \ \ \text{as }\ k\to \infty.
\end{equation} 
Next, using the convergence results \eqref{CON-H1}, \eqref{CON-H2}, \eqref{CON-Z1} and \eqref{CON-Z}, together with \eqref{CON-D}, we derive from the definition of $\mathcal{G}_{\widetilde{u}}$ given in equations \eqref{SOC-GN} that 
\begin{equation}\label{CON-G}
	\mathcal{G}_{\widetilde{u}}[h_k,h_k] \to \mathcal{G}_{\widetilde{u}}[h,h]\ \ \ \ \ \text{as }\ k\to \infty.
\end{equation}

Now, for the first term $\mathcal{B}_1$, substituting \eqref{Z-5} and using the convergences \eqref{CON-Z}, \eqref{CON-XI}, \eqref{CON-ZT} and \eqref{CON-D}, we find that
\begin{align}\label{Z-7}
	\lim_{k\to \infty}\frac{\mathcal{B}_1}{\rho_k^2}&= 	\lim_{k\to \infty} \frac{\big\|\widetilde{m}_t(\widetilde{T})\, (T_k^*-\widetilde{T}) + \frac{1}{2}\ \widetilde{m}_{tt}(\widetilde{T})\,(T_k^*-\widetilde{T})^2 + \rho_k^2 \ \mathscr{E}_k+ \rho_k \, z_k(T_k^*)+ \frac{\rho_k^2}{2}\ \xi_k(T_k^*) + \rho_k^2\ \zeta_k(T_k^*) \big\|^2_{L^2(\Omega)}}{\rho_k^2}\nonumber\\
	&= \|\widetilde{m}_t(\widetilde{T})\, \mathcal{D}_{\widetilde{u}}[h] + z(\widetilde{T})\|^2_{L^2(\Omega)}.
\end{align}
	Next, we will estimate the second term $\mathcal{B}_2$. Substituting \eqref{Z-3} and the expression
	\begin{align*}
	&\widetilde{m}(T_k^*) -\widetilde{m}(\widetilde{T})= \widetilde{m}_t(\widetilde{T})\, \Big(T_k^*-\widetilde{T}-\rho_k\, \mathcal{D}_{\widetilde{u}}[h_k]-\rho_k^2\,\mathcal{G}_{\widetilde{u}}[h_k,h_k]\Big)\\
	&\hspace{1cm}+ \rho_k\, \widetilde{m}_t(\widetilde{T}) \,\mathcal{D}_{\widetilde{u}}[h_k] +\rho_k^2\, \widetilde{m}_t(\widetilde{T})\,\mathcal{G}_{\widetilde{u}}[h_k,h_k]+\frac{1}{2}\ \widetilde{m}_{tt}(\widetilde{T})\,(T_k^*-\widetilde{T})^2 + \rho_k^2 \ \mathscr{E}_k,
    \end{align*}
which is derived from \eqref{Z-5}, into $\mathcal{B}_2$, we obtain
\begin{align*}
	\mathcal{B}_2=2\ &\bigg( \widetilde{m}_t(\widetilde{T})\, \Big(T_k^*-\widetilde{T}-\rho_k\, \mathcal{D}_{\widetilde{u}}[h_k]-\rho_k^2\,\mathcal{G}_{\widetilde{u}}[h_k,h_k]\Big)+ \rho_k\, \widetilde{m}_t(\widetilde{T}) \,\mathcal{D}_{\widetilde{u}}[h_k] +\rho_k^2\, \widetilde{m}_t(\widetilde{T})\,\mathcal{G}_{\widetilde{u}}[h_k,h_k]\nonumber\\
	& +\frac{1}{2}\  \widetilde{m}_{tt}(\widetilde{T})\,(T_k^*-\widetilde{T})^2 +\rho_k z_k(\widetilde{T}) + \frac{\rho_k^2}{2}\ \xi_k(T_k^*) + \rho_k^2\ \zeta_k(T_k^*)  \ ,\ \widetilde{m}(\widetilde{T}) -m_{\Omega}\bigg)_{L^2(\Omega)} \nonumber\\
	&+2\, \rho_k \ \big\langle z_{kt}(\widetilde{T})(T_k^*-\widetilde{T}) \,,\,\widetilde{m}(\widetilde{T}) -m_{\Omega}\big\rangle_{H^1(\Omega)^*\times H^1(\Omega)}+\rho_k^2 \ \mathscr{H}_k,
\end{align*}
where, $\mathscr{H}_k$ represents the error term, which depends on the parameters $\Psi_k$ and $\mathscr{E}_k$, and is such that $\mathscr{H}_k \to 0$ as $k\to \infty$.
	
	Furthermore, by leveraging the equality \eqref{Z-1}, we arrive at the expression
	$$z_{kt}(\widetilde{T})(T_k^*-\widetilde{T})=\rho_k \ z_{kt}(\widetilde{T}) \big( \mathcal{D}_{\widetilde{u}}[h_k]\big) + \rho_k \Psi_{\rho_k},$$ and 
$$\widetilde{m}_{tt}(T_k^*-\widetilde{T})^2=\widetilde{m}_{tt}\Big(\rho_k \ \mathcal{D}_{\widetilde{u}}[h_k] + \rho_k \Psi_{\rho_k}\Big)^2 = \rho_k^2\,\widetilde{m}_{tt}\big(\mathcal{D}_{\widetilde{u}}[h_k]^2\big) + \rho_k^2 \Psi_{\rho_k}. $$
Substituting these expressions into the previous equation, we obtain 
\begin{align}\label{Z-8}
\mathcal{B}_2=2\ &\bigg( \widetilde{m}_t(\widetilde{T})\, \Big(T_k^*-\widetilde{T}-\rho_k\, \mathcal{D}_{\widetilde{u}}[h_k]-\rho_k^2\,\mathcal{G}_{\widetilde{u}}[h_k,h_k]\Big)+ \rho_k\, \widetilde{m}_t(\widetilde{T}) \,\mathcal{D}_{\widetilde{u}}[h_k] +\rho_k^2\, \widetilde{m}_t(\widetilde{T})\,\mathcal{G}_{\widetilde{u}}[h_k,h_k]\nonumber\\
& +\frac{1}{2}\ \rho_k^2 \widetilde{m}_{tt}(\widetilde{T})\,\mathcal{D}_{\widetilde{u}}[h_k]^2 +\rho_k z_k(\widetilde{T}) + \frac{\rho_k^2}{2}\ \xi_k(T_k^*) + \rho_k^2\ \zeta_k(T_k^*)  \ ,\ \widetilde{m}(\widetilde{T}) -m_{\Omega}\bigg)_{L^2(\Omega)} \nonumber\\
&+2\, \rho_k^2 \ \big\langle z_{kt}(\widetilde{T})\big( \mathcal{D}_{\widetilde{u}}[h_k]\big) \,,\,\widetilde{m}(\widetilde{T}) -m_{\Omega}\big\rangle_{H^1(\Omega)^*\times H^1(\Omega)}+\rho_k^2 \ \mathscr{H}_k.
\end{align}

Finally, we apply the equality $z_k(\widetilde{T})+\widetilde{m}_t(\widetilde{T})\mathcal{D}_{\widetilde{u}}[h_k]=0$ from equation \eqref{SOC-D} into \eqref{Z-8}, divide by $2\rho_k^2$ and take the limit as $k\to \infty$. Additionally, by using the identity \eqref{SOC-D}, along with the convergences \eqref{CON-Z}, \eqref{CON-XI} and \eqref{CON-ZT}, as well as Lemma \ref{Z-K},  we derive 
\begin{align}\label{Z-9}
	&\lim_{k\to \infty} \frac{\mathcal{B}_2}{2\rho_k^2} =\big(\widetilde{m}_t(\widetilde{T})\,,\, \widetilde{m}(\widetilde{T})-m_{\Omega}\big)_{L^2(\Omega)}\ \lim_{k\to \infty} \frac{T_k^*-\widetilde{T}-\rho_k\, \mathcal{D}_{\widetilde{u}}[h_k]-\rho_k^2\,\mathcal{G}_{\widetilde{u}}[h_k,h_k]}{\rho_k^2} \nonumber\\
	&\hspace{0.3cm }+  \big(\widetilde{m}_t(\widetilde{T})\,,\, \widetilde{m}(\widetilde{T})-m_{\Omega}\big)_{L^2(\Omega)}\,\mathcal{G}_{\widetilde{u}}[h,h]+ \bigg( \frac{1}{2}\ \widetilde{m}_{tt}(\widetilde{T})\,\mathcal{D}_{\widetilde{u}}[h]^2  + z_{t}(\widetilde{T}) \mathcal{D}_{\widetilde{u}}[h] +  \frac{\xi(\widetilde{T})}{2}   \,,\, \widetilde{m}(\widetilde{T}) -m_{\Omega}\bigg)_{L^2(\Omega)}. 
\end{align}
Since $\big(\widetilde{m}_t(\widetilde{T}), \widetilde{m}(\widetilde{T})-m_{\Omega}\big)_{L^2(\Omega)}<0$ by assumption \eqref{CON-2}, we divide \eqref{Z-6} by $2\rho_k^2$, take the limit as $k\to \infty$, and substitute \eqref{Z-7} and \eqref{Z-9}. As a result, we obtain the following expression
\begin{align}\label{Z-9.5}
	\lim_{k\to \infty} &\frac{T_k^*-\widetilde{T}-\rho_k\, \mathcal{D}_{\widetilde{u}}[h_k]-\rho_k^2\,\mathcal{G}_{\widetilde{u}}[h_k,h_k]}{\rho_k^2}= -\frac{1}{2} \frac{ \|\widetilde{m}_t(\widetilde{T})\, \mathcal{D}_{\widetilde{u}}[h] + z(\widetilde{T})\|^2_{L^2(\Omega)}}{\big(\widetilde{m}_t(\widetilde{T}), \widetilde{m}(\widetilde{T})-m_{\Omega}\big)_{L^2(\Omega)}}-\mathcal{G}_{\widetilde{u}}[h,h]\nonumber\\
	& -\frac{1}{\big(\widetilde{m}_t(\widetilde{T}), \widetilde{m}(\widetilde{T})-m_{\Omega}\big)_{L^2(\Omega)}} \bigg( \frac{1}{2}\ \widetilde{m}_{tt}(\widetilde{T})\,\mathcal{D}_{\widetilde{u}}[h]^2  + z_{t}(\widetilde{T}) \mathcal{D}_{\widetilde{u}}[h] +  \frac{\xi(\widetilde{T})}{2}   ,\ \widetilde{m}(\widetilde{T}) -m_{\Omega}\bigg)_{L^2(\Omega)}. 
\end{align}
Before proceeding further, we extract the term $\xi$ from the previous equality \eqref{Z-9.5}. We will proceed by expressing the term $\Big( \xi(\widetilde{T}),\widetilde{m}(\widetilde{T})-m_{\Omega}\Big)$ by utilizing the adjoint and the linearized solutions. So, by taking the $L^2(\Omega)$ inner product of \eqref{XI-LS} with $\phi$, and adding it to the weak formulation \eqref{WEFO} of the adjoint system with $\vartheta=\xi$, followed by integration by parts of respective terms,  we derive
$$2\int_{\Omega_{\widetilde{T}}} \Big(z\times \Delta z + z \times h -2\, \big(z\cdot \widetilde{m}\big)\, z-|z|^2\,\widetilde{m}\Big)\ \phi \ dx\ dt =\frac{-\widetilde{T} \ \big(\xi(\widetilde{T})\,,\,\widetilde{m}(\widetilde{T})-m_{\Omega}\big)}{\big(\widetilde{m}_t(\widetilde{T})\,,\, \widetilde{m}(\widetilde{T})-m_{\Omega}\big)_{L^2(\Omega)} }.$$
Finally, substituting this equality into equation \eqref{Z-9.5}, and applying \eqref{SOC-GN}, we arrive at our desired conclusion that 
$$\lim_{k\to \infty} \frac{T_k^*-\widetilde{T}-\rho_k\, \mathcal{D}_{\widetilde{u}}[h_k]-\rho_k^2\,\mathcal{G}_{\widetilde{u}}[h_k,h_k]}{\rho_k^2}=0.$$


	\noindent \textbf{Step III:} Next, we will show that $h\in \Big(\mathcal{P}_{\mathcal{U}_{ad}}(\widetilde{u}) \cap \Upsilon(\widetilde{u})\Big)\backslash \{0\} $.

Since $	\mathcal{P}_{\mathcal{U}_{ad}}(\widetilde{u})$ is a closed and convex subset of the Hilbert space $\mathcal{U}$, it follows that it is weakly closed. Consequently, $h\in \mathcal{P}_{\mathcal{U}_{ad}}(\widetilde{u})$.
	
	Now, taking the difference $\mathcal{J}(u_k)-\mathcal{J}(\widetilde{u})$ and then substituting \eqref{Z-2}, we will obtain
	\begin{align*}
		\mathcal{J}(u_k)-\mathcal{J}(\widetilde{u}) &=\frac{1}{2}\ (T^*_k-\widetilde{T})\ (T^*_k+\widetilde{T}) +\rho_k \ (( \widetilde{u},h_k))_{\mathcal{U}} +\frac{1}{2}\ \rho_k^2 \ \| h_k\|^2_{\mathcal{U}}\\
		&=\frac{1}{2}\ \Big[\rho_k \mathcal{D}_{\widetilde{u}}[h_k]+ \rho_k^2 \mathcal{G}_{\widetilde{u}}[h_k,h_k] +\rho_k^2 \Psi_{\rho_k} \Big]\ \Big[2\widetilde{T}+\rho_k \mathcal{D}_{\widetilde{u}}[h_k]+ \rho_k^2 \mathcal{G}_{\widetilde{u}}[h_k,h_k] +\rho_k^2 \Psi_{\rho_k} \Big]\\
		&\hspace{3cm}+\rho_k \ ((\widetilde{u},h_k))_{\mathcal{U}} +\frac{1}{2}\ \rho_k^2 \ \| h_k\|^2_{\mathcal{U}}\\
		&=\rho_k \ \Big(\widetilde{T}\ \mathcal{D}_{\widetilde{u}}[h_k] + \ (( \widetilde{u},h_k))_{\mathcal{U}} \Big) +\frac{\rho_k^2}{2}\ \Big( \mathcal{D}_{\widetilde{u}}[h_k]^2 +2\widetilde{T} \ \mathcal{G}_{\widetilde{u}}[h_k,h_k] +\|h_k\|^2_{\mathcal{U}}  \Big) +\rho_k^2\ \Psi_{\rho_k},
	\end{align*}
	where the terms containing powers of $\rho_k$ higher than 2 have been grouped into the remainder $\Psi_{\rho_k}$.
	
	Now, substituting \eqref{FO-E7} and using the definitions $\mathcal{Y}(h_k)$ and $\mathcal{Q}(h_k)$ from \eqref{FOOC} and \eqref{SO-ME}, we obtain
	\begin{align*}
		\mathcal{J}(u_k)-\mathcal{J}(\widetilde{u}) &= \rho_k \ \left(\int_{\Omega_{\widetilde{T}}} \big(\phi \times \widetilde{m}\big)\cdot h_k\ dx\ dt + \int_{\Omega_{\widetilde{T}}} \phi \cdot h_k\ dx\ dt+ \ (( \widetilde{u},h_k))_{\mathcal{U}} \right)\\
		&\hspace{2cm} +\frac{\rho_k^2}{2}\ \Big( \mathcal{D}_{\widetilde{u}}[h_k]^2 +2\widetilde{T} \ \mathcal{G}_{\widetilde{u}}[h_k,h_k] +\|h_k\|^2_{\mathcal{U}}  \Big) +\rho_k^2\ \Psi_{\rho_k} \\
		&=\rho_k\ \mathcal{Y}(h_k) +\frac{\rho_k^2}{2}\ \mathcal{Q}(h_k) + \rho_k^2\ \Psi_{\rho_k}.
	\end{align*}
	Since from \eqref{CQGC} $\mathcal{J}(u_k)-\mathcal{J}(\widetilde{u}) <\frac{\rho_k^2}{k}$, substituting this in the above estimate we find 
	\begin{equation}\label{Z-10}
	\mathcal{Y}(h_k)< \frac{\rho_k}{k} -\frac{\rho_k}{2}\ \mathcal{Q}(h_k) - \rho_k\ \Psi_{\rho_k}.	
	\end{equation}
	Using the convergences \eqref{CON-D} and \eqref{CON-G}, along with the norm equality $\|h_k\|_{\mathcal{U}}=1$\ for all $k$ , it follows that $\mathcal{Q}(h_k)$ is bounded. Therefore, taking the limit as $k \to \infty$ in \eqref{Z-10}, and using the convergences \eqref{CON-H1} and \eqref{CON-H2}, we obtain
	\begin{equation}\label{Z-11}
	\mathcal{Y}(h)\leq 0.	
	\end{equation} 
	Since $h\in \mathcal{P}_{\mathcal{U}_{ad}}(\widetilde{u})$, by combining \eqref{Z-11} with \eqref{FOOC}, we obtain 	$\mathcal{Y}(h)= 0.	$ Therefore, by the definition given in \eqref{SET-EQ0}, we have $h\in \Upsilon (\widetilde{u}).$
	
	Using the convergences \eqref{CON-D} and \eqref{CON-G}, and the weak sequential lower semi-continuity of $h_k$, from the definition of $\mathcal{Q}(h_k)$ in equation \eqref{SO-ME}, we find
	\begin{equation}\label{Z-12}
		\mathcal{D}_{\widetilde{u}}[h]^2+2\ \widetilde{T} \ \mathcal{G}_{\widetilde{u}}[h,h] + \|h\|^2_{\mathcal{U}} \leq \liminf_{k \to \infty} \Big( \mathcal{D}_{\widetilde{u}}[h_k]^2+2\ \widetilde{T} \ \mathcal{G}_{\widetilde{u}}[h_k,h_k] + \|h_k\|^2_{\mathcal{U}} \Big)=\liminf_{k \to \infty} \mathcal{Q}(h_k).
	\end{equation}
	Moreover, since $\mathcal{Y}(h_k)\geq 0$ for all $k$, it follows from \eqref{Z-10} that
	\begin{equation}\label{Z-13}
		\mathcal{Q}(h_k) < 2\left( \frac{1}{k}-\Psi_{\rho_k}\right).
	\end{equation}
	By combining \eqref{Z-12} and \eqref{Z-13}, it is clear that $h$ must be non-zero.

	\noindent \textbf{Step IV:} $\mathcal{Q}(h)\leq 0.$ 

Again, exploiting the weak sequential lower semi-continuity along with the convergences \eqref{CON-D} and \eqref{CON-G}, we find from \eqref{Z-13} that 
$$\mathcal{Q}(h)\leq \liminf_{k \to \infty} \mathcal{Q}(h_k)\leq  \lim_{k\to \infty} 2\left( \frac{1}{k}-\Psi_{\rho_k}\right) =0.$$ 
This is a contradiction to our assumption \eqref{SO-ME}. Thus, the proof is complete.
\end{proof}

\begin{proof}[Proof of Lemma \ref{ER-SY}]
	We have considered a simplified version of system \eqref{ER-SY} in Lemma \ref{L-SER}, focusing on specific types of nonlinearities. Other forms of nonlinearity can be addressed using similar arguments. 
	
	By analyzing equations \eqref{ER-SY} and \eqref{ER-SF}, we observe that each term in $\mathcal{X}_1$ closely resembles a corresponding term in $\mathcal{F}_1$. Since $\widetilde{m}, z_k, \xi_k$ are in $L^2(0,T;H^3(\Omega))\cap L^{\infty}(0,T;H^2(\Omega))$, they also belong to the space required by the $\phi_i$'s in Lemma \ref{L-SER}. 
	
	Similarly, by comparing, we can find that each term of $\mathcal{X}_i$ in the equation \eqref{ER-SY} finds resemblance with one of the terms $\mathcal{F}_i$ in the equation \eqref{ER-SF}. Moreover, they also satisfy the regularity requirement given in Lemma \ref{L-SER}. Therefore, following the pathway of Lemma \ref{L-SER}, we can establish the existence of a regular solution $\zeta_k$ for the system \eqref{ER-SY}. Furthermore, since $\rho_k=\|u_k-\widetilde{u}\|_{\mathcal{M}}$, we have $\rho_k \to 0$ as $k\to \infty$. Therefore, by deriving the energy estimate as in Lemma \ref{L-SER}, and using the convergence of $\rho_k$, we can find that $\zeta_k\to 0$ in $\mathcal{M}$ as  $k\to \infty$.	
\end{proof}

\section{Conclusion and Future Works}

	This work establishes a rigorous framework for the \emph{time–optimal} control of the LLB dynamics on bounded domains in one, two, and three dimensions, proving the existence of optimal solutions and deriving both first-order necessary and second-order sufficient conditions under the required regularity of the state and controls. The analysis overcomes the difficulties due to the  nonlinearities $m\times\Delta m$ and $|m|^{2}m$ as well as the nonlinear control appearance in the effective field by developing a unified solvability theory for the linearized and second-variation systems, and by working with controls in $L^{2}(0,T;H^{1}(\Omega))$, we  demonstrate that  $m\in W^{1,2}\big(0,T;H^{3}(\Omega),H^{1}(\Omega)\big)$. These results clarify the local structure of optimizers for minimal-time switching targets and provide a mathematically stable entry point for finite element time-marchers and adjoint-based algorithms, thereby laying the groundwork for quantitative studies relevant to ultrafast spin dynamics and thermally assisted magnetic switching.

	Recent numerical work around micromagnetic control offers useful examples. For single-particle switching on ellipsoids, \cite{FAKB} constructs space-uniform fields that exponentially stabilize uniform states and gives a PDE-level approximate switching via a Lyapunov argument. For networks of ellipsoidal cells, \cite{SAGC} establishes stability/controllability under dipolar coupling and supports it with explicit-Euler-renormalized simulations. For PDE-constrained LLG control, \cite{TDMKAP} analyzes a semi-implicit Euler–finite element scheme, proves stability/compactness for the semidiscrete optimality system, and contrasts it with a projection–penalization variant, which enforces $\lvert m\rvert=1$.

	More closure to our study, \cite{ECIG} treats some minimal-time control problems by making the final time an additional decision variable in a bi-objective quadratic setup. Controls are piecewise-constant and updated by direct as well as iterative steps; for ODEs this technique even yields closed-form updates. For PDEs (e.g., the heat equation), they use a reduced state-adjoint system with a terminal-tube condition: for each  final time candidate, a single scalar multiplier enforces the terminal distance, and the optimal time follows from a simple scalar relation. The computation of state and adjoint systems is advanced with standard finite elements and implicit time-stepping methods.

	Finally, the time-optimal computations for the LLB equation can treat the final time as an additional decision variable and enforce the terminal $\delta$-tube via a scalar multiplier or a penalty. A standard state-adjoint loop may be built with conforming finite elements in space and stable time stepping, while the control uses variational discretization and a discrete adjoint to drive projected or augmented-Lagrangian updates in $(u,T)$. Further, initial checks can use macrospin and 1D reductions before 2D/3D studies, reporting switching-time/tolerance trade-offs and examining bang-bang tendencies of the control.

\section{Appendix}  \label{app}
In this section, we present a concise proof of Lemma \ref{L-MTT}. For a detailed exposition of the intermediate steps, the reader is referred to the proof of Theorem 2.2 in \cite{SPSK3}. Furthermore, the proof relies on the estimates and inequalities stated in Lemmas \ref{CPP} and \ref{EN} as well as Proposition  \ref{PROP1}. 
\begin{proof}[Proof of Lemma \ref{L-MTT}]
Consider the same set of eigenfunctions and Galerkin approximated system of \eqref{NLP} introduced in the proof of Theorem 2.1 in \cite{SPSK3}:
	\begin{equation}\label{4-GA}
	\begin{cases}
		(m_n)_t- \Delta m_n=\mathbb{P}_n \left[ m_n \times \Delta m_n + m_n \times u-\big(1+|m_n|^2\big)\ m_n +u \right],\\
		m_n(0)=\mathbb{P}_n(m_0).
	\end{cases}
\end{equation}
Taking ``$\nabla \Delta$" in equation \eqref{4-GA} and considering the $L^2$ inner product with $-\nabla \Delta^2 m_n$, we find 
\begin{align}\label{REQ-1}
	\frac{1}{2} \frac{d}{dt}\|\Delta^2 m_n(t)\|^2_{L^2(\Omega)} &+ \|\nabla \Delta^2 m_n(t)\|^2_{L^2(\Omega)} = - \Big(\nabla \Delta \left(m_n\times \Delta m_n\right), \nabla \Delta^2 m_n\Big) - \Big( \nabla \Delta \big(m_n\times u\big),\nabla  \Delta^2 m_n\Big)\ \ \ \ \nonumber\\
	& +\left(\nabla\Delta \left(\left(1+|m_n|^2\right)m_n\right),\nabla \Delta^2 m_n\right) -  \big(\nabla\Delta  u,\nabla \Delta^2 m_n\big):= \sum_{i=1}^{4} E_i. \ \ \ \ \ \ \ 
\end{align}
Before proceeding to estimate the right-hand side terms in \eqref{REQ-1}, it can be  noted that the estimates presented in Proposition \ref{PROP1}  also hold when $m$ is replaced by $\Delta m_n$, since $\frac{\partial \Delta m_n}{\partial \eta}=0$. \\
\indent For the first term $E_1,$ utilizing the equality $a\times a=0$, $a\cdot (a \times b)=0$, applying H\"older's inequality, incorporating the embeddings $H^1(\Omega) \hookrightarrow L^4(\Omega)$ and $H^2(\Omega) \hookrightarrow L^\infty(\Omega)$ along with estimates \eqref{ES3} and \eqref{ES4}, we determine 
\begin{flalign*}
	E_1 &= - \int_\Omega \left(\nabla m_n(t) \times \Delta^2 m_n(t) + 2\ \nabla\, \Big(\nabla m_n(t) \times \nabla \Delta m_n(t)\Big)\right) \cdot \nabla \Delta^2 m_n(t)\ dx\\
	& \leq \Big( \|\nabla m_n(t)\|_{L^{\infty}(\Omega)}\, \|\Delta^2 m_n(t)\|_{L^2(\Omega)} + 2\, \|D^2m_n(t)\|_{L^4(\Omega)}\, \|\nabla \Delta m_n(t)\|_{L^4(\Omega)}\\
	&\hspace{1cm} +2\, \|\nabla m_n(t)\|_{L^{\infty}(\Omega)}\, \|D^2\Delta m_n(t)\|_{L^2(\Omega)}\Big) \ \|\nabla \Delta^2 m_n(t)\|_{L^2(\Omega)} &\\
	& \leq \epsilon \ \|\nabla \Delta^2 m_n(t)\|_{L^2(\Omega)}^2 + C\ \Big( \|\nabla m_n(t)\|^2_{H^2(\Omega)}+ \|m_n(t)\|^2_{H^3(\Omega)}\Big) \ \|\Delta^2 m_n(t)\|^2_{L^2(\Omega)}. 
\end{flalign*}
Note that the cross product between the gradients is in the sense that $\displaystyle \nabla a \times \nabla b= \sum_{k=1}^n \frac{\partial a_k}{\partial x_k} \times  \frac{\partial b_k}{\partial x_k}$.\\
\noindent For the second term $E_2$ , applying H\"older's inequality followed by the embedding $H^1(\Omega)\hookrightarrow L^4(\Omega)$  and $H^2(\Omega) \hookrightarrow L^\infty(\Omega)$, we derive
\begin{flalign*}
	E_2 &=  - \int_\Omega \bigg[ \nabla  \Big(\Delta m_n(t) \times u(t)+ m_n(t) \times \Delta u(t) + 2\, \nabla m_n \times \nabla u\Big)\cdot  \nabla \Delta^2 m_n(t)\bigg]\ dx&\\
	&\leq \Big(\|\nabla \Delta  m_n(t)\|_{L^2(\Omega)} \ \|u(t)\|_{L^\infty(\Omega)} +  \| \Delta m_n(t)\|_{L^4(\Omega)} \|\nabla u(t)\|_{L^4(\Omega)} + \|\nabla m_n(t)\|_{L^\infty(\Omega)} \|\Delta u(t)\|_{L^2(\Omega)} \\
	&\hspace{0.21cm} + \|m_n(t)\|_{L^\infty(\Omega)}  \|\nabla\Delta  u(t)\|_{L^2(\Omega)} + 2\, \|D^2m_n(t)\|_{L^4(\Omega)} \|\nabla u(t)\|_{L^4(\Omega)} \\
    &\hspace{0.21cm} +2\, \|\nabla m_n(t)\|_{L^4(\Omega)} \|D^2 u(t)\|_{L^4(\Omega)} \Big) \, \|\nabla \Delta^2 m_n(t)\|_{L^2(\Omega)}\\
	&\leq \epsilon   \ \|\nabla \Delta m_n(t)\|^2_{L^2(\Omega)} + C(\epsilon) \ \|m_n(t)\|^2_{H^3(\Omega)}\ \|u(t)\|^2_{H^2(\Omega)} + \|m_n(t)\|^2_{H^2(\Omega)} \ \|u(t)\|^2_{H^3(\Omega)}.
\end{flalign*}
Finally, proceeding in the same way for the last two terms $E_3$ and $E_4$, we arrive at
\begin{flalign*}
	E_3&= \int_{\Omega} \nabla \Big( \Delta m_n(t) + 2\, \big( \Delta m_n(t) \cdot m_n(t) \big) m_n(t) + \big( m_n(t)\cdot m_n(t)\big)\Delta m_n(t) +2\big(\nabla m_n(t) \cdot \nabla m_n(t)\big)m_n(t)\\
	&\hspace{1.5cm} + 4\,\big(m_n(t)\cdot \nabla m_n(t)\big)\nabla m_n(t)\Big) \cdot \nabla \Delta^2 m_n(t)\ dx&\\
	&\leq \epsilon \ \|\nabla \Delta^2 m_n(t)\|^2_{L^2(\Omega)} + C(\epsilon)\ \left(1+  \|m_n(t)\|^4_{H^2(\Omega)} \right)\ \|m_n(t)\|^2_{H^3(\Omega)},
\end{flalign*}
and
\begin{flalign*}
	E_4 &= -\int_\Omega \nabla \Delta u(t)\cdot \nabla \Delta^2 m_n(t) \ dx \leq \epsilon \ \|\nabla \Delta^2 m_n(t)\|^2_{L^2(\Omega)} + C(\epsilon) \ \|u(t)\|^2_{H^3(\Omega)}.&
\end{flalign*}
Substituting the estimates for $E_i$ from $i=1$ to $4$ in \eqref{REQ-1}, choosing $\epsilon = 1/8$, we obtain
\begin{flalign*}
	&\frac{d}{dt} \|\Delta^2 m_n(t)\|^2_{L^2(\Omega)} + \|\nabla \Delta^2 m_n(t)\|^2_{L^2(\Omega)} \leq C\, \|m_n(t)\|^2_{H^3(\Omega)} \ \|\Delta^2 m_n(t)\|^2_{L^2(\Omega)} \\
	&+  C \left(1+\|m_n(t)\|^4_{H^2(\Omega)}+\|u(t)\|^2_{H^2(\Omega)} \right)\ \|m_n(t)\|^2_{H^3(\Omega)} +C   \left(1+\|m_n(t)\|^2_{H^2(\Omega)}\right)\ \|u(t)\|^2_{H^3(\Omega)}.
\end{flalign*}
Now, applying Gr\"onwall's inequality and implementing the estimate $\|\Delta^2 m_n(0)\|_{L^2(\Omega)}\leq \|\Delta^2 m_0\|_{L^2(\Omega)}$, we derive 
\begin{flalign*}
	& \|\Delta^2 m_n(t)\|^2_{L^2(\Omega)} +\int_0^t \|\nabla \Delta^2 m_n(\tau)\|^2_{L^2(\Omega)}\ d\tau \leq \bigg(\|\Delta^2 m_0\|^2_{L^2(\Omega)} \nonumber\\
	&\hspace{1cm} + C(\Omega,T)\  \Big(1+\|u\|^2_{L^\infty(0,T;H^2(\Omega))} + \|m_n\|^4_{L^\infty(0,T;H^2(\Omega))}\Big)\ \|m_n\|^2_{L^2(0,T;H^3(\Omega))}\nonumber \\
	&\hspace{1cm} +C(\Omega,T)\  \Big(1+\|m_n\|^2_{L^\infty(0,T;H^2(\Omega))}\Big)\,\|u\|^2_{L^2(0,T;H^3(\Omega))}\bigg)     \exp\left\{C(\Omega)\, \|m_n\|^2_{L^2(0,T;H^3(\Omega))}\right\}.
\end{flalign*}
\indent Applying the uniform bounds for $m_n$ in $L^\infty(0,T;H^2(\Omega))\cap L^2(0,T;H^3(\Omega))$ obtained in the proof of Theorem 2.2 in \cite{SPSK3},  and using the equality of norms from Lemma \ref{EN} (along with estimate \eqref{ES7} and \eqref{ES10}), we can find an uniform bound of $\{m_n\}$ in $L^\infty(0,T;H^4(\Omega))\cap L^2(0,T;H^5(\Omega))$. Now, proceeding in the same way as  we have done in  that  proof,  we can  find   that   $\{m_n\}$ has a subsequence that converges to a solution  $m\in L^\infty(0,T;H^4(\Omega))\cap L^2(0,T;H^5(\Omega))$ of \eqref{NLP} with $m_t \in L^2(0,T;H^3(\Omega))$. Finally, using Theorem 2 in Section 5.9 of \cite{LCE}, it follows that $m\in C([0,T];H^4(\Omega))$. \\
\indent Moreover, invoking the same regularity results for $u\in L^2(0,T;H^3(\Omega))$ with $u_t\in C([0,T];L^2(\Omega))$, we can conclude that $u\in C([0,T];H^2(\Omega))$. Substituting all these regularities  for the control and the state into equation \eqref{NLP}, we further obtain $m_t\in C([0,T];H^2(\Omega))$, which, together with the regularity $u_t\in C([0,T];L^2(\Omega))$, implies $m_{tt}\in C([0,T];L^2(\Omega))$. This completes the proof. 
\end{proof}


\section*{Acknowledgments}
The authors sincerely thank the anonymous reviewers for their insightful comments and constructive suggestions, which have significantly improved the quality and clarity of this manuscript.


\end{document}